\documentclass[a4paper,twoside,11pt,reqno]{amsart}
\usepackage[english]{babel}
\usepackage[utf8]{inputenc}
\usepackage{amsmath}
\usepackage{amssymb}
\usepackage{amsfonts}
\usepackage{amsthm}
\usepackage{mathrsfs}
\usepackage{hyperref}
\hypersetup{
     colorlinks=true,
     linkcolor=blue,
     filecolor=blue,
     citecolor=purple,      
     urlcolor=cyan,
     }
\usepackage{xcolor}
\usepackage{dsfont}
\usepackage{physics}
\usepackage{stmaryrd}
\usepackage{upgreek}

\usepackage{tikz-cd}
\usepackage{tikz}
\usepackage{tikzsymbols}
\usepackage{mathtools}
\usetikzlibrary{matrix,decorations.pathreplacing}

\usepackage{enumerate}

\usepackage[matrix,arrow,cmtip]{xy} 
\SelectTips{cm}{}
\newdir{(}{{}*!/-5pt/@^{(}}
\newdir{(x}{{}*!/-5pt/@_{(}}
\newdir{+}{{}*!/-9pt/{}}
\newdir{>+}{@{>}*!/-9pt/{}}
\entrymodifiers={+!!<0pt,\fontdimen22\textfont2>}

\theoremstyle{plain}
\newtheorem{theorem}[subsubsection]{Theorem}
\newtheorem*{theorem*}{Theorem}
\newtheorem{lemma}[subsubsection]{Lemma}
\newtheorem{proposition}[subsubsection]{Proposition}
\newtheorem{corollary}[subsubsection]{Corollary}

\newtheorem{question}[subsubsection]{Question}

\newtheorem{thmx}{Theorem}

\newtheorem{corx}[thmx]{Corollary}

\theoremstyle{definition}
\newtheorem{definition}[subsubsection]{Definition}

\newtheorem{example}[subsubsection]{Example}

\theoremstyle{remark}
\newtheorem{remark}[subsubsection]{Remark}

\numberwithin{equation}{section}

\newcommand{\NN}{\mathbf{N}} 
\newcommand{\ZZ}{\mathbf{Z}} 
\newcommand{\QQ}{\mathbf{Q}} 
\newcommand{\RR}{\mathbf{R}} 
\newcommand{\CC}{\mathbf{C}} 
\newcommand{\PP}{\mathbf{P}} 
\newcommand{\GG}{\mathbf{G}} %
\newcommand{\kk}{\mathbf{k}} 

\newcommand{\q}{\; / \;}
\newcommand\inj{\hookrightarrow}

\DeclareMathOperator{\Hom}{Hom}
\newcommand{\Coker}{\mathrm{Coker}} 
\newcommand{\longeq}{=\joinrel=} 
\newcommand{\Eu}{\mathrm{Eu}} 

\newcommand{\ba}{\mathbf{a}} 
\newcommand{\bB}{\mathbf{B}} 
\newcommand{\bI}{\mathbf{I}} 
\newcommand{\be}{\mathbf{e}} 
\newcommand{\bE}{\mathbf{E}} 
\newcommand{\bw}{\mathbf{w}} 
\newcommand{\bv}{\mathbf{v}} 
\newcommand{\bu}{\mathbf{u}} 
\newcommand{\bn}{\mathbf{n}} 
\newcommand{\bX}{\mathbf{X}} 
\newcommand{\bzero}{\mathbf{0}} 

\newcommand{\uu}{\mathrm{u}}

\newcommand\fa{\mathfrak{a}}

\newcommand\fS{\mathfrak{S}}

\newcommand\fl{\mathfrak{l}}

\newcommand\sM{\mathscr{M}} 

\newcommand\sO{\mathscr{O}}

\newcommand\sS{\mathscr{S}}

\newcommand\sX{\mathscr{X}}

\newcommand\cF{\mathcal{F}}
\newcommand\cM{\mathcal{M}}

\newcommand\cS{\mathcal{S}}

\newcommand\dB{\mathds{B}}

\renewcommand{\AA}{\mathbf{A}} 
\newcommand{\LL}{\mathbf{L}} %

\DeclareMathOperator{\Spec}{Spec} 
\DeclareMathOperator{\Proj}{Proj} 
\DeclareMathOperator{\bl}{\mathbf{Bl}} 
\newcommand{\sing}{\mathrm{sing}} 
\newcommand{\GrpSch}{\mathrm{Grp-Sch}} 
\newcommand{\tg}{\mathrm{top}} 
\newcommand{\mot}{\mathrm{mot}} 

\newcommand{\Gm}{\mathbf{G}_\mathrm{m}}

\newcommand{\Gmu}{\pmb{\mu}}

\makeatletter
\newcommand\xmapsfrom[2][]{%
  \ext@arrow 0395{\longmapsfromfill@}{#1}{#2}}
\newcommand\longmapsfromfill@{%
  \arrowfill@\leftarrow\relbar{\relbar\mapsfromchar}}
\makeatother

\newcommand{\vertex}{\mathrm{vert}}

\newcommand{\std}{\mathrm{std}}
\newcommand{\exc}{\mathrm{ex}}

\DeclareMathOperator{\Var}{\textup{\sffamily Var}} 
\newcommand{\relint}{\mathrm{relint}} 

\makeatletter
\renewcommand*\env@matrix[1][*\c@MaxMatrixCols c]{%
  \hskip -\arraycolsep
  \let\@ifnextchar\new@ifnextchar
  \array{#1}}
\makeatother

\newcommand{\ChMingHao}[1]{{\color{black} #1}}

\begin{document}

\title[Around the non-degenerate motivic monodromy conjecture]{Around the motivic monodromy conjecture for non-degenerate hypersurfaces}

\author{Ming Hao Quek}
\address{Department of Mathematics, Brown University, Box 1917, 151 Thayer Street, Providence, RI 02912}
\email{\href{mailto:ming_hao_quek@brown.edu}{ming\_hao\_quek@brown.edu}}


\maketitle

\begin{abstract}
We provide a new, geometric proof of the motivic monodromy conjecture for non-degenerate hypersurfaces in dimension $3$, which has been proven previously by the work of Lemahieu--Van Proeyen and Bories--Veys. More generally, given a non-degenerate complex polynomial $f$ in any number of variables and a set $\dB$ of $B_1$-facets of the Newton polyhedron of $f$ with consistent base directions, we construct a stack-theoretic embedded desingularization of $f^{-1}(0)$ above the origin, whose set of numerical data excludes any known candidate pole of the motivic zeta function of $f$ at the origin that arises solely from facets in $\dB$. We anticipate that the constructions herein might inspire new insights as well as new possibilities towards a solution of the conjecture.
\end{abstract}

\section{Introduction}\label{S:introduction}

Throughout this paper, let $\kk$ be a subfield of $\CC$, and fix $0 \neq n \in \NN$. For every $\ba := (a_1,\dotsc,a_n) \in \NN^n$, let $\pmb{x}^\ba$ denote the monomial $x_1^{a_1}\dotsm x_n^{a_n}$ in $\kk[x_1,\dotsc,x_n]$. Let $f = \sum_{\ba \in \NN^n}{c_\ba \cdot \pmb{x}^\ba} \; \in \; \kk[x_1,\dotsc,x_n]$ be a non-constant polynomial satisfying $c_\bzero = f(\bzero) = 0$, and let $V(f)$ be the hypersurface defined by $f = 0$ in $\AA^n := \Spec(\kk[x_1,\dotsc,x_n])$. Let $\Upgamma(f)$ denote the Newton polyhedron of $f$, defined as the convex hull in $\RR^n$ of the finite union \[
    \bigcup \bigl\{\ba + \RR_{\geq 0}^n \colon \ba \in \NN^n,\; c_\ba \neq 0\bigr\}.
\]
For every face $\varsigma$ of $\Upgamma(f)$, we set \begin{equation}\label{EQ:f-tau}
    f_\varsigma \; := \; \sum_{\ba \in \NN^n \cap \varsigma}{c_\ba \cdot \pmb{x}^\ba}.
\end{equation}
We then say that $f$ is {\sffamily non-degenerate}, if for every compact face $\varsigma$ of $\Upgamma(f)$, the closed subscheme $V(f_\varsigma) \subset \AA^n$ is non-singular in the torus $\Gm^n \subset \AA^n$. This non-degeneracy condition was first introduced in \cite{kouchnirenko-newton-polyehdra}, and it guarantees that the singularity theory of $V(f) \subset \AA^n$ at the origin $\bzero \in \AA^n$ is, to a certain extent, governed by $\Upgamma(f)$. The extent to which the former is governed by the latter is the main interest of this paper. 

Namely, this paper provides a \emph{geometric} explanation (Theorem~\ref{T:refined-desingularization}) for the proposition \cite[Proposition 3.8]{esterov-lemahieu-takeuchi-monodromy-conjecture} that any pole of the topological zeta function of $f$ at $\bzero \in \AA^n$ \cite{denef-loeser-local-zeta-functions} cannot arise exclusively from a set of $B_1$-facets of $\Upgamma(f)$ with consistent base directions. In the process, we obtain a \emph{smaller} set of candidate poles for the motivic zeta function of $f$ at $\bzero \in \AA^n$ \cite{denef-loeser-geometry-arc-spaces} than what was previously known in general (Theorem~\ref{T:fake-poles}), and in particular we deduce (via Theorem~\ref{T:fake-poles-n=3}) a new, geometric proof of:

\begin{theorem*}[\textrm{$=$ \cite[Theorem 10.3]{bories-veys-monodromy-conjecture}}]
The motivic monodromy conjecture holds for non-degenerate polynomials in $n=3$ variables.
\end{theorem*}

\subsection{Statement of objectives, motivations, and results}\label{1.1}

We assume throughout this introduction that $f \in \kk[x_1,\dotsc,x_n]$ is non-degenerate. 

\begin{xpar}[Conventions on the Newton polyhedron of $f$]\label{X:convention-1.2-A}
For $m \in \NN$, set $[m] := \{1,2,\dotsc,m\}$ (i.e. $[0] = \varnothing$). Let $N = \ZZ^n$, with standard basis vectors $\be_i$ ($i \in [n]$), and for $R$ a subring of $\RR$, we set $N_R = N \otimes_\ZZ R$ with positive half-space $N_R^+ = R^n_{\geq 0} \subset N_R$ (where $R_{\geq 0} = R \cap \RR_{\geq 0}$). Set $M = N^\vee$ be the dual lattice, with standard dual basis vectors $\be_i^\vee$ ($i \in [n]$). Likewise for $R$ a subring of $\RR$, we set $M_R = M \otimes_\ZZ R$ with positive half-space $M_R^+ = \Hom_\NN(R^n_{\geq 0},R_{\geq 0}) \subset M_R$. We also write $N^+$ for $N_\ZZ^+$ and $M^+$ for $M_\ZZ^+$.

For reasons related to toric geometry, we view $\Upgamma(f)$ as a polyhedron in $M_\RR^+$ (instead of $N_\RR^+$). For a face $\varsigma$ of $\Upgamma(f)$, \[
    \textrm{we write } \;\; \begin{cases}
    \varsigma' \prec \varsigma \\
    \varsigma' \prec^{\ChMingHao{1}} \varsigma \\
    \vertex(\varsigma) \\
    \dim(\varsigma)
    \end{cases} \quad \textrm{for} \;\; \begin{cases}
    \textrm{a face $\varsigma'$ of $\varsigma$}. \\
    \textrm{a \ChMingHao{facet ($=$ codimension $1$ face)} $\varsigma'$ of $\varsigma$}. \\
    \textrm{the set of vertices of $\varsigma$}. \\
    \textrm{the dimension of the affine span of $\varsigma$.}
    \end{cases}
\]
We usually use the letter $\tau$ instead of $\varsigma$ to denote facets of $\Upgamma(f)$. If two facets $\tau_1$ and $\tau_2$ of $\Upgamma(f)$ intersect in a common facet (i.e. $\tau_1 \cap \tau_2 \prec^1 \tau_1, \tau_2$), we say that $\tau_1$ and $\tau_2$ are {\sffamily adjacent}, and write \[
    \tau_1 \frown \tau_2.
\]
Finally, for $i \in [n]$, let $H_i$ denote the coordinate hyperplane in $M_\RR$ defined by $\be_i = 0$. For $\tau \prec^1 \Upgamma(f)$, let $H_\tau$ be its affine span in $M_\RR$, with equation $\bigl\{\ba \in M_\RR \colon \ba \cdot \bu_\tau = N_\tau\bigr\}$, where the vector $\bu_\tau := (u_{\tau,i})_{i=1}^n$ is the unique primitive vector in $N^+$ that is normal to $H_\tau$. If $N_\tau > 0$ (i.e. $\tau$ is not contained in any coordinate hyperplane $H_i$ in $M_\RR$), we define the {\sffamily numerical datum} of $\tau$ as: \begin{equation}\label{EQ:numerical-datum}
    \upeta_\tau \; := \; (N_\tau,\abs{\bu_\tau}) \; := \; \bigl(N_\tau,u_{\tau,1} + u_{\tau,2} + \dotsb + u_{\tau,n}\bigr)
\end{equation}
and the \ChMingHao{\sffamily candidate pole} of $\tau$ as the root of the polynomial $N_\tau s + \abs{\bu_\tau}$, i.e. \begin{equation}\label{EQ:ratio}
    s_\tau \; := \; -\abs{\bu_\tau}/N_\tau.
\end{equation}
Finally, for $s_\circ \in \QQ_{<0}$, we let $\cF(f;s_\circ) := \{\tau \prec^1 \Upgamma(f) \colon N_\tau > 0$ and $s_\tau = s_\circ \}$.
\end{xpar}

\begin{xpar}\label{X:motivic-monodromy-conjecture}
The first main theorem of this paper concerns the na\"ive\footnote{\ChMingHao{Our main theorems apply more generally to a refined version of $Z_{\mot,\bzero}(f;s)$ \cite[Chapter 7, \S4.1.3 and Remark 4.2.8]{chambert-loir-nicaise-sebag-motivic-integration}. For simplicity, we do not discuss that version here.}} motivic zeta function of $f$ at $\bzero \in \AA^n$ (cf. \cite[Definition 3.2.1]{denef-loeser-geometry-arc-spaces}, and \cite[Chapter 7, \S 3.3.1]{chambert-loir-nicaise-sebag-motivic-integration}), which we shall denote by $Z_{\mot,\bzero}(f;s)$, and is tied to the singularity theory of $V(f) \subset \AA^n$ at $\bzero \in \AA^n$ via the motivic monodromy conjecture of \ChMingHao{Denef--Loeser}.

In our setting, their conjecture states that there should exist a set of candidate poles $\Uptheta$ for $Z_{\mot,\bzero}(f;s)$ (in the sense of \cite[Definition 5.4.1]{bultot-nicaise-log-smooth-models}) such that every $s_\circ \in \Uptheta$ {\sffamily induces a monodromy eigenvalue of $f$ near $\bzero \in \CC^n$} in the following sense. Given any neighbourhood $U$ of $\bzero$ in $f^{-1}(0) \subset \CC^n$, there exists $x \in U$ such that $\exp(2\pi 
\sqrt{-1} s_\circ)$ is an eigenvalue of the monodromy transformation acting on the singular cohomology $\bigoplus_{i \geq 0}{H^i_\sing(F_{f,x},\ZZ)}$ of the Milnor fiber $F_{f,x}$ of $f$ at $x$, cf. \cite{milnor-singular-points-complex-hypersurfaces} and \cite[Chapter 1, \S 3.4.1]{chambert-loir-nicaise-sebag-motivic-integration}.
\end{xpar}

\begin{xpar}\label{X:candidate-poles}
To start, it has been established in the literature (cf. \cite[Theorem 10.5]{bories-veys-monodromy-conjecture} or \cite[Theorem 8.3.5]{bultot-nicaise-log-smooth-models}) that \begin{equation}\label{EQ:candidate-poles-0}
    \Uptheta(f) \; := \; \{-1\} \; \cup \; \bigl\{s_\tau \colon \tau \prec^1 \Upgamma(f) \textrm{ with } N_\tau > 0\bigr\}
\end{equation}
is a set of candidate poles for $Z_{\mot,\bzero}(f;s)$. More precisely, the preceding statement can be explicated as follows: \begin{equation}\label{EQ:candidate-poles}
    Z_{\mot,\bzero}(f;s) \; \in \; \sM_\kk\bigl[\LL^{-s}\bigr]\left[\frac{1}{1-\LL^{-(Ns+\nu)}} \colon (N,\nu) \in \upeta(f)\right]
\end{equation}
where \begin{equation}\label{EQ:numerical-data}
    \upeta(f) \; := \; \{(1,1)\} \; \cup \; \bigl\{\upeta_\tau \colon \tau \prec^1 \Upgamma(f) \textrm{ with } N_\tau > 0\bigr\}
\end{equation}
and $\sM_\kk$ denotes the localization of the Grothendieck ring $K_0(\Var_\kk)$ of $\kk$-varieties ($=$ finite-type $\kk$-schemes) with respect to the class $\LL$ of $\AA^1$. Note that the letter $T$ is sometimes used in place of the indeterminate $\LL^{-s}$. 
\end{xpar}

\begin{xpar}\label{X:difficulty}
Unfortunately, the main difficulty in establishing the motivic monodromy conjecture for a non-degenerate polynomial $f$ lies in the fact that \emph{not every} candidate pole in $\Uptheta(f)$ induces a monodromy eigenvalue of $f$ near $\bzero \in \AA^n$. Therefore, one desires for a smaller set of candidate poles for $Z_{\mot,\bzero}(f;s)$. This paper gives a partial answer to the question of when a strictly smaller set of candidate poles than $\Uptheta(f)$ exists for $Z_{\mot,\bzero}(f;s)$, which can be seen as a \emph{motivic upgrade} of some existing general results in the literature pertaining to a ``close relative'' of $Z_{\mot,\bzero}(f;s)$, namely the topological zeta function $Z_{\tg,\bzero}(f;s)$ of $f$ at $\bzero \in \AA^n$, cf. \cite{denef-loeser-local-zeta-functions} and \cite[Chapter 1, \S 3.3.1, equation (3.3.1.3)]{chambert-loir-nicaise-sebag-motivic-integration}. 
\end{xpar}

\begin{remark}\label{R:topological-zeta-function}
Indeed $Z_{\tg,\bzero}(f;s)$ is a ``close relative'' of $Z_{\mot,\bzero}(f;s)$ in the sense that $Z_{\mot,\bzero}(f;s)$ {\sffamily specializes} to $Z_{\tg,\bzero}(f;s)$ via the motivic measure: \[
    \Eu \; \colon \; \sM_\kk \; \xrightarrow{\quad\quad} \; \ZZ
\]
which sends a $\kk$-variety $X$ to the topological Euler characteristic of $X \otimes_\kk \CC$, cf. \cite[Section 3.4]{denef-loeser-geometry-arc-spaces} for details. In particular, one recovers in this way an analogue of \eqref{EQ:candidate-poles} for $Z_{\tg,\bzero}(f;s)$ (which was observed earlier in \cite[Theorem 5.3(ii)]{denef-loeser-local-zeta-functions}), namely that every pole of $Z_{\tg,\bzero}(f;s)$ lies in $\Uptheta(f)$. 
\end{remark}

\begin{xpar}
To segue into our main results, it is useful (as hinted in \ref{X:difficulty}) to first recall some existing results in the literature which demonstrate that occasionally some candidate poles $s_\tau$ in $\Uptheta(f) \smallsetminus \{-1\}$ are \emph{not} actual poles of $Z_{\tg,\bzero}(f;s)$. Few of these results are known for $Z_{\mot,\bzero}(f;s)$ prior to this paper, especially for general $n$. We start with the following definition:
\end{xpar}

\begin{definition}[$B_1$-facets, cf. {\cite[Definition 3.1]{esterov-lemahieu-takeuchi-monodromy-conjecture}}, {\cite[Definition 1.4.1]{larson-payne-stapledon-simplicial-nondegenerate-monodromy-conjecture}}]\label{D:B1-facets}
A facet $\tau \prec^1 \Upgamma(f)$ is called a {\sffamily $B_1$-facet} if there exists $\bv \in \vertex(\tau)$ and $i \in [n]$ such that: \begin{enumerate}
    \item[(a)] The $i$\textsuperscript{th} coordinate of $\bv$ is $1$.
    \item[(b)] \ChMingHao{$\varnothing \neq \vertex(\tau) \smallsetminus \{\bv\} \subset H_i$.}
    \item[(c)] $\tau$ is {\sffamily compact in the $i$\textsuperscript{th} coordinate}, i.e. $\tau + \RR_{\geq 0}\be_i^\vee \not\subset \tau$ (cf. \ref{X:first-meet-locus}(iii)).
\end{enumerate}
Note that in particular, (b) and (c) imply that $H_i \cap \tau \prec^1 \tau$. In this case, we call $\bv$ an {\sffamily apex} of $\tau$ with corresponding {\sffamily base direction} $i \in [n]$. Note that the apex $\bv$ and the base direction $i$ uniquely determine each other.
\end{definition}

\begin{xpar}
Fix $-1 \neq s_\circ \in \QQ_{<0}$. It is known that if $\cF(f;s_\circ)$ only consists of one $B_1$-facet, then $s_\circ$ is not a pole of $Z_{\tg,\bzero}(f;s)$, cf. \cite[Proposition 3.7]{esterov-lemahieu-takeuchi-monodromy-conjecture}. More generally one might guess that conclusion is true whenever $\cF(f;s_\circ)$ comprises of only $B_1$-facets. However, this is \emph{false}, cf. Example~\ref{EX:x^2+yz} and Remark~\ref{R:x^2+yz} for a simple counterexample. One rectifies that guess (cf. \cite[Proposition 3.8]{esterov-lemahieu-takeuchi-monodromy-conjecture}) by further imposing the following condition on $\cF(f;s_\circ)$:
\end{xpar}  

\begin{definition}\label{D:consistent-B1-facets}
A set $\dB$ of $B_1$-facets of $\Upgamma(f)$ has {\sffamily consistent base directions} if there exists, for each facet $\tau \in \dB$, a choice of a distinguished base direction $b(\tau) \in [n]$, such that $b(\tau_1) = b(\tau_2)$ for every pair of adjacent facets $\tau_1, \tau_2 \in \dB$. In this case we call $\{b(\tau) \colon \tau \in \dB\}$ a {\sffamily set of consistent base directions} for $\dB$.
\end{definition}

The main contribution of this paper can now be stated as follows:

\begin{thmx}\label{T:fake-poles}
Let $\dB$ be a set of $B_1$-facets of $\Upgamma(f)$ with consistent base directions. Then \[
    \Uptheta^{\dag,\dB}(f) \; := \; \{-1\} \; \cup \; \bigl\{s_\tau \colon \tau \prec^1 \Upgamma(f) \textrm{ with } N_\tau > 0 \textrm{ and } \tau \notin \dB\bigr\}
\]
is a set of candidate poles for $Z_{\mot,\bzero}(f;s)$.
\end{thmx}

\begin{xpar}\label{X:new-contribution}
We prove Theorem~\ref{T:fake-poles} towards the end of \S\ref{4.3}. The centerpiece of our proof ($=$ Theorem~\ref{T:refined-desingularization} below) is perhaps more satisfying than Theorem~\ref{T:fake-poles} itself, especially given that previous attempts to understand the topological zeta function analogue of Theorem~\ref{T:fake-poles}, or even special cases of Theorem~\ref{T:fake-poles}, used roundabout methods: namely, they typically involve a \emph{manipulation} of some explicit formula for $Z_{\tg,\bzero}(f;s)$ or $Z_{\mot,\bzero}(f;s)$, cf. formulae in \cite[Theorem 5.3(iii)]{denef-loeser-local-zeta-functions}, \cite[Theorem 4.2]{denef-hoornaert-newton-polyhedra}, and \cite[Theorem 10.5]{bories-veys-monodromy-conjecture}. In contrast, our proof is \emph{geometric} in nature, in the sense that we construct an appropriate embedded desingularization of $V(f) \subset \AA^n$ above $\bzero \in \AA^n$ that bears witness to Theorem~\ref{T:fake-poles}.
\end{xpar}

\begin{xpar}\label{X:approach-1}
To put our approach to Theorem~\ref{T:fake-poles} into perspective, we \ChMingHao{shift our attention} to our approach towards its weaker counterpart \eqref{EQ:candidate-poles-0}, i.e. \eqref{EQ:candidate-poles}. Given that there is a motivic change of variables formula for $Z_{\mot,\bzero}(f;s)$ under any proper, birational morphism $\uppi \colon X \to \AA^n$ (cf. \cite[Chapter 6, \S 4.3]{chambert-loir-nicaise-sebag-motivic-integration}), one natural hope towards proving \eqref{EQ:candidate-poles} would be to apply the change of variables to an appropriate embedded desingularization $\uppi \colon X \to \AA^n$ of $V(f) \subset \AA^n$ above $\bzero \in \AA^n$. A natural candidate for $\uppi$ would be the toric modification $\uppi_{\Sigma'} \colon X_{\Sigma'} \to \AA^n$ induced by any smooth subdivision $\Sigma'$ of the normal fan $\Sigma(f)$ of $\Upgamma(f)$. Indeed, one can show that the non-degeneracy condition on $f$ implies that $\uppi_{\Sigma'}$ desingularizes $V(f) \subset \AA^n$ above $\bzero \in \AA^n$, cf. \cite[Sections 9 and 10]{varchenko-monodromy-and-newton-diagram}. Unfortunately, subdividing $\Sigma(f)$ into $\Sigma'$ usually \emph{introduces new rays} to $\Sigma(f)$. One can show this process of adding new rays \ChMingHao{cannot show in general the existence of a set of candidate poles for $Z_{\mot,\bzero}(f;s)$ as small as $\Uptheta(f)$}.
\end{xpar}

\begin{xpar}\label{X:approach-2}
The above discussion suggests that one should perhaps \emph{avoid} the process of adding new rays, and instead work directly on $\Sigma(f)$ and its associated toric modification $\uppi_{\Sigma(f)} \colon X_{\Sigma(f)} \to \AA^n$, despite the fact that $\uppi_{\Sigma(f)}$ is usually not an embedded desingularization for $V(f) \subset \AA^n$ above $\bzero \in \AA^n$ (as $X_{\Sigma(f)}$ is usually singular).

Nevertheless, this was the approach in a recent paper of Bultot--Nicaise \cite{bultot-nicaise-log-smooth-models}, where they instead showed that if one endows $X_{\Sigma(f)}$ with the divisorial logarithmic structure $\cM$ associated to \ChMingHao{the reduction of} \[
    \uppi_{\Sigma(f)}^{-1}\bigl(V(f) \; \cup \; V(x_1) \; \cup \; V(x_2) \; \cup \; \dotsb \; \cup \; V(x_n)\bigr) \; \subset \; X_{\Sigma(f)}
\]
the logarithmic scheme $(X_{\Sigma(f)},\cM)$ is logarithmically smooth. They then related $Z_{\mot,\bzero}(f;s)$ to a different motivic zeta function associated to $X_{\Sigma(f)}$ and the Gelfand--Leray form $dx_1 \wedge dx_2 \wedge \dotsb \wedge dx_n / df$ (cf. Loeser--Sebag \cite{loeser-sebag-motivic-integration-smooth-rigid-varieties} and \cite[Definition 5.2.2]{bultot-nicaise-log-smooth-models}). Finally, the logarithmic smoothness of $(X_{\Sigma(f)},\cM)$ enables them to deduce an explicit formula for the latter zeta function, from which \eqref{EQ:candidate-poles} follows.
\end{xpar}

\begin{xpar}\label{X:approach-3}
In contrast, our approach towards \eqref{EQ:candidate-poles} is a \emph{stack-theoretic} re-interpretation of Bultot--Nicaise's approach, and allows one to work directly on $\Sigma(f)$ while still remaining in the realm of \emph{smooth ambient spaces}. The point here is that one can associate, to the potentially singular toric variety $X_{\Sigma(f)}$, a smooth toric Artin stack $\sX_{\Sigma(f)}$ whose good moduli space (in the sense of \cite{alper-good-moduli-spaces}) is $X_{\Sigma(f)}$, cf. \S\ref{3.1}. One can then show that the composition \[
    \Uppi_{\Sigma(f)} \; \colon \sX_{\Sigma(f)} \; \xrightarrow{\quad\quad} \; X_{\Sigma(f)} \; \xrightarrow{\quad\uppi_{\Sigma(f)}\quad} \; \AA^n
\]
desingularizes $V(f) \subset \AA^n$ above $\bzero \in \AA^n$ in the following sense:
\end{xpar}

\begin{definition}\label{D:stack-desingularization}
A {\sffamily stack-theoretic embedded desingularization of $V(f) \subset \AA^n$ above $\bzero \in \AA^n$} is a morphism $\Uppi \colon \sX \to \AA^n$ where: \begin{enumerate}
    \item $\sX$ is a smooth Artin stack over $\kk$ admitting a good moduli space $\sX \to \bX$, and the induced morphism $\uppi \colon \bX \to \AA^n$ is proper and birational.
    \item $\Uppi^{-1}\bigl(V(f)\bigr)$ is a simple normal crossings divisor at \ChMingHao{every point} in $\Uppi^{-1}(\bzero)$ (in the stack-theoretic sense, cf. \cite[Definition 3.1]{bergh-rydh-destackification}).
\end{enumerate}
\end{definition}

\begin{xpar}\label{X:approach-4}
In \S\ref{3.1} we also discuss a motivic change of variables for $Z_{\mot,\bzero}(f;s)$ that is applicable to $\Uppi_{\Sigma(f)}$, although indirectly. By this we mean that one has to first take a \emph{simplicial} subdivision $\pmb{\Upsigma}(f)$ of $\Sigma(f)$ \emph{without adding new rays}. The effect of doing so is that the corresponding toric stack $\sX_{\pmb{\Upsigma}(f)}$ is Deligne--Mumford, and the morphism $\Uppi_{\pmb{\Upsigma}(f)} \colon \sX_{\pmb{\Upsigma}(f)} \to \AA^n$ factors through $\Uppi_{\Sigma(f)} \colon \sX_{\Sigma(f)} \to \AA^n$ as an open substack, i.e. $\Uppi_{\pmb{\Upsigma}(f)}$ also desingularizes $V(f) \subset \AA^n$ above $\bzero \in \AA^n$. Finally we compute the set of numerical data associated to $(f,\Uppi_{\pmb{\Upsigma}(f)})$ (in the sense of Definition~\ref{D:numerical-datum} below), and show that it is the set $\upeta(f)$ in \eqref{EQ:numerical-data}. Applying the aforementioned motivic change of variables to $\Uppi_{\pmb{\Upsigma}(f)}$, the preceding sentence then implies \eqref{EQ:candidate-poles}.
\end{xpar}

\begin{definition}\label{D:numerical-datum}
Let $\Uppi \colon \sX \to \AA^n$ be a stack-theoretic embedded desingularization of $V(f) \subset \AA^n$ above $\bzero \in \AA^n$, such that $\sX$ is a Deligne--Mumford stack. Let $\{E_i \colon i \in I\}$ denote the set of irreducible components of $\Uppi^{-1}\bigl(V(f)\bigr)$. For each $i \in I$, let $N_i$ \ChMingHao{(resp. $\nu_i-1$)} denote the multiplicity of $E_i$ in the divisor $\Uppi^{-1}\bigl(V(f)\bigr)$ \ChMingHao{(resp. the relative canonical divisor $K_\Uppi$ of $\Uppi$)}. Then the {\sffamily set of numerical data} associated to the pair $(f,\Uppi)$ is: \[
    \upeta(f,\Uppi) \; := \; \bigl\{(N_i,\nu_i) \colon i \in I\bigr\}
\]
where each $(N_i,\nu_i)$ is referred to as the {\sffamily numerical datum} of the corresponding irreducible component $E_i \subset \Uppi^{-1}\bigl(V(f)\bigr)$.
\end{definition}

Similar to how the motivic change of variables in \ref{X:approach-4} reduces \eqref{EQ:candidate-poles-0} to the existence of a stack-theoretic desingularization of $V(f) \subset \AA^n$ above $\bzero \in \AA^n$ \ChMingHao{by a Deligne--Mumford stack}, whose set of numerical data equal to $\upeta(f)$, that same change of variables would also reduce Theorem~\ref{T:fake-poles} to the following:

\begin{thmx}[$\impliedby$Theorem~\ref{T:desingularization-new}]\label{T:refined-desingularization}
Given a set $\dB$ of $B_1$-facets of $\Upgamma(f)$ with consistent base directions, there exists a stack-theoretic embedded desingularization $\Uppi \colon \sX \to \AA^n$ of $V(f) \subset \AA^n$ above $\bzero \in \AA^n$, such that $\sX$ is a Deligne--Mumford stack, and whose set of numerical data is: \[
    \upeta^{\dag,\dB}(f) \; := \; \{(1,1)\} \; \cup \; \bigl\{\upeta_\tau \colon \tau \prec^1 \Upgamma(f) \textrm{ with } N_\tau > 0 \textrm{ and } \tau \notin \dB\bigr\}.
\]
\end{thmx}

\begin{xpar}\label{X:approach-5}
Our proof of Theorem~\ref{T:refined-desingularization} occupies the entirety of \S\ref{4}. As one might expect from the discussion in \ref{X:approach-1} and \ref{X:approach-3}, the proof should involve the construction of a fan $\Sigma^\dag$ that subdivides $N_\RR^+$ and satisfies the following: \begin{enumerate}
    \item The set of rays in $\Sigma^\dag$ comprises of rays in $\Sigma(f)$ \emph{except} those that are dual to facets in $\dB$.
    \item The induced toric modification $\Uppi_{\Sigma^\dag} \colon \sX_{\Sigma^\dag} \to \AA^n$ is a stack-theoretic embedded desingularization of $V(f) \subset \AA^n$ above $\bzero \in \AA^n$.
\end{enumerate}
In the first two paragraphs of \S\ref{3.2}, we give a brief sketch as to how one could accomplish this construction, and in \S\ref{4.1} and \S\ref{4.2}, we provide the details of the construction. In addition, in \S\ref{3.2} we also verify our methods for three non-degenerate polynomials in $n=3$ variables. We hope to highlight, through these examples, various aspects of Theorems~\ref{T:fake-poles} and \ref{T:refined-desingularization}. 
\end{xpar}


\begin{xpar}\label{X:future}
Finally, we indicate in \S\ref{5} the various aspects in which Theorem~\ref{T:fake-poles} is incomplete for the motivic monodromy conjecture for non-degenerate polynomials (\ref{X:motivic-monodromy-conjecture}), most of which we are pursuing separately in a sequel, using methods that are motivated by and similar to the ones in this paper. 

Nevertheless, Theorem~\ref{T:fake-poles} in particular recovers the motivic monodromy conjecture for non-degenerate polynomials in $n=3$ variables, which was proven previously by Bories--Veys \cite[Theorem 10.3]{bories-veys-monodromy-conjecture}, although (as hinted in \ref{X:new-contribution}) via an approach different from Theorem~\ref{T:refined-desingularization}. Indeed, in \S\ref{5.1}, we first show that Theorem~\ref{T:fake-poles} implies:
\end{xpar}

\begin{thmx}[$=$ Theorem~\ref{T:fake-poles-n=3:compatible}]\label{T:fake-poles-n=3}
Let $n=3$, and let $\cS_\circ \subset \Uptheta(f) \smallsetminus \{-1\}$. If $\cF(f;s_\circ)$ is a set of $B_1$-facets of $\Upgamma(f)$ with consistent base directions for each $s_\circ \in \cS_\circ$, then $\Uptheta(f) \smallsetminus \cS_\circ$ is a set of candidate poles for $Z_{\mot,\bzero}(f;s)$.
\end{thmx}

Note that by specializing $Z_{\mot,\bzero}(f;s)$ to $Z_{\tg,\bzero}(f;s)$ (cf. Remark~ \ref{R:topological-zeta-function}), Theorem~\ref{T:fake-poles-n=3} in particular recovers \cite[Proposition 14]{lemahieu-van-proeyen-nondegenerate-surface-singularities}. Moreover, the authors in loc. cit. showed that $s_\circ \in \Uptheta(f) \smallsetminus \{-1\}$ induces a monodromy eigenvalue of $f$ near $\bzero \in \CC^n$ (in the sense indicated in \ref{X:motivic-monodromy-conjecture}) whenever $\cF(f;s_\circ)$ satisfies either of the following hypotheses: \begin{enumerate}
    \item $\cF(f;s_\circ)$ contains a \emph{non}--$B_1$-facet of $\Upgamma(f)$ \cite[Theorem 10]{lemahieu-van-proeyen-nondegenerate-surface-singularities}.
    \item $\cF(f;s_\circ)$ is a set of $B_1$-facets of $\Upgamma(f)$, but \emph{without} consistent base directions \cite[Theorem 15]{lemahieu-van-proeyen-nondegenerate-surface-singularities}.
\end{enumerate}
Therefore, we conclude from Theorem~\ref{T:fake-poles-n=3} and the preceding sentence that:

\begin{corx}\label{C:motivic-nondegenerate-monodromy-conjecture-n=3}
The motivic monodromy conjecture holds for non-degenerate polynomials in $n = 3$ variables.
\end{corx}

\subsection{Acknowledgements}

This research work was supported in part by funds from BSF grant 2018193 and NSF grant DMS-2100548. 

The author first conceived the idea behind this paper during the final weeks of his time ($8$\textsuperscript{th} September -- $6$\textsuperscript{th} December 2021) at the Moduli and Algebraic Cycles program in Institut Mittag-Leffler, where he was also supported by the Swedish Research Council grant no. 2016-06596. Therefore he would like to thank John Christian Ottem, Dan Petersen, and David Rydh for organizing the program and giving him the opportunity to participate in it, for without the opportunity he would not have the privilege to meet and discuss mathematics with Johannes Nicaise. He is grateful to \ChMingHao{the referee, Nero Budur, Matt Larson}, Johannes Nicaise, Sam Payne, Matthew Satriano, Jeremy Usatine \ChMingHao{and Willem Veys} for their time, suggestions and questions, and his advisor Dan Abramovich for his guidance, patience and wisdom. Finally, the author also thanks Mihnea Popa for his course on Hodge ideals in Spring 2021, where the author first learnt about the monodromy conjecture among other topics.

\setcounter{tocdepth}{1}
\tableofcontents

\newpage
\section{Nuts and bolts}

\subsection{Newton \texorpdfstring{$\QQ$}{Q}-polyhedra and piecewise-linear convex \texorpdfstring{$\QQ$}{Q}-functions}\label{2.1}

We begin by reviewing some fundamentals in convex geometry in \S\ref{2.1} and \S\ref{2.2}. \ChMingHao{A reader who is familiar with convex geometry can skip to \S\ref{2.3}. Along the way we also fix conventions and notations for the remainder of the paper.}

\begin{definition}[Newton $\QQ$-polyhedra]\label{D:newton-Q-polyhedra}
By a {\sffamily rational, positive half-space} in $M_\RR^+$, we mean any set of the form \[
    H_{\bu,m}^+ \; := \; \bigl\{\ba \in M_\RR^+ \colon \ba \cdot \bu \geq m \bigr\} \; \subset \; M_\RR^+
\]
for some $\bzero \neq \bu \in N^+$ and $m \in \ChMingHao{\ZZ_{>0}}$. We also set: \[
    H_{\bu,m} \; := \; \bigl\{\ba \in M_\RR^+ \colon \ba \cdot \bu = m \bigr\} \; \subset \; M_\RR^+.
\]
We call an intersection of finitely many rational, positive half-spaces in $M_\RR^+$ a {\sffamily Newton $\QQ$-polyhedron} (with the empty intersection defined as $M_\RR^+$), typically denoted by the letter $\Upgamma$. Equivalently, a Newton $\QQ$-polyhedron is the convex hull in $M_\RR$ of $\bigcup\{\ba + M_\RR^+ \colon \ba \in S\}$ for a finite subset of points $S \subset M_\QQ^+$.
\end{definition}

\begin{remark}\label{R:newton-Q-polyhedra}
If the vertices of a Newton $\QQ$-polyhedron $\Upgamma$ also lie in $M^+$, then $\Upgamma$ is simply referred to as a Newton polyhedron.
\end{remark}

\begin{xpar}[Conventions on Newton $\QQ$-polyhedra]\label{X:convention-2.1.A}
In \ref{X:convention-1.2-A}, we outlined a few conventions on the Newton polyhedron $\Upgamma(f)$ of a polynomial $f \in \kk[x_1,\dotsc,x_n]$. The same conventions make sense for a Newton $\QQ$-polyhedron, and moving ahead we will also adopt them for Newton $\QQ$-polyhedra.
\end{xpar}

\begin{xpar}[Piecewise-linear convex $\QQ$-functions]\label{X:piecewise-linear-convex-function}
To every Newton $\QQ$-polyhedron $\Upgamma$, we can associate a piecewise-linear, convex function $\varphi \colon N_\RR^+ \to \RR_{\geq 0}$ as follows: \[
    \varphi(\bu) \; := \; \inf_{\ba \in \Upgamma}{\ba \cdot \bu} \qquad \textrm{ for every $\bu \in N_\RR^+$}.
\]
This means that there exists a finite set $S \subset M_\RR^+$ so that $\varphi(\bu) = \min_{\ba \in S}{\ba \cdot \bu}$ for every $\bu \in N_\RR^+$. Indeed, for the above $\varphi$, we may take $S = \vertex(\Upgamma)$. In fact, since $\vertex(\Upgamma) \subset M_\QQ^+$, $\varphi$ is a {\sffamily piecewise-linear, convex $\QQ$-function}, that is, either of the following equivalent conditions hold for $\varphi$: \begin{enumerate}
    \item $\varphi$ is a piecewise-linear, convex function such that $\varphi(N^+) \subset \QQ_{\geq 0}$.
    \item There exists a finite set $S \subset M_\QQ^+$ such that $\varphi(\bu) = \min_{\ba \in S}{\ba \cdot \bu}$.
\end{enumerate}
This sets up a one-to-one correspondence between: \[
    \bigl\{\textrm{Newton $\QQ$-polyhedra in $M_\RR^+$}\bigr\} \; \xleftrightarrow{\quad\quad} \; \left\{\parbox{4.6cm}{piecewise linear, convex, $\QQ$-functions $\varphi \colon N_\RR^+ \to \RR_{\geq 0}$}\right\}.
\]
Indeed, every $\varphi$ in the right hand side arises uniquely from the following Newton $\QQ$-polyhedron: \begin{equation}\label{EQ:associated-newton-Q-polyhedron}
    \Upgamma \; = \; \bigl\{\ba \in M_\RR^+ \colon \ba \cdot \bu \geq \varphi(\bu) \textrm{ for all } \bu \in N_\RR^+\bigr\} \; = \; \bigcap_{\bu \in N_\RR^+}{H_{\bu,\varphi(\bu)}^+}.
\end{equation}
\ChMingHao{This claim is immediate once we show that $\Upgamma$ is a Newton $\QQ$-polyhedron. Indeed, let $S = \{\ba_1,\dotsc,\ba_r\} \subset M_\QQ^+$ be such that $\varphi(\bu) = \min_{i \in [r]}{\ba_i \cdot \bu}$ for every $\bu \in N_\RR^+$. Then $\Upgamma$ is the intersection of all rational, positive half-spaces in $M_\RR^+$ containing $S$, i.e.} $\Upgamma$ is the convex hull in $M_\RR$ of $\bigcup\{\ba + M_\RR^+ \colon \ba \in S\}$. Since $S$ is finite, so is $\vertex(\Upgamma) \subset S$. Thus, $\Upgamma$ has finitely many faces. For each facet $\tau$ of $\Upgamma$, let $\bu_\tau$ be the unique primitive vector in $N^+$ that is normal to the affine hyperplane spanned by $\tau$. Then $\Upgamma$ is the following finite intersection of rational, positive half-spaces in $M_\RR^+$: \begin{equation}\label{EQ:newton-Q-polyhedron-as-finite-intersection}
    \Upgamma \; = \; \bigcap_{\tau \prec^1 \Upgamma}{H_{\bu_\tau,\varphi(\bu_\tau)}^+}.
\end{equation}
\end{xpar}

\begin{xpar}\label{X:alt-characterization}
\ChMingHao{By \eqref{EQ:newton-Q-polyhedron-as-finite-intersection}}, we obtain the following alternative description of $\varphi$ in terms of facets of $\Upgamma$ (as opposed to points in $\Upgamma$): \[
    \varphi \; = \; \min\sS
\]
where \[
    \sS \; := \; \left\{\parbox{7cm}{linear functions $\ell \colon N_\RR^+ \to \RR_{\geq 0}$ such that $\ell(\bu_\tau) \geq N_\tau$ for every facet $\tau \prec^1 \Upgamma$}\right\}.
\]
Recall from \ref{X:convention-2.1.A} that for every $\tau \prec^1 \Upgamma$, $N_\tau \in \QQ_{>0}$ is defined via the equation $\bigl\{\ba \in M_\RR \colon \ba \cdot \bu_\tau = N_\tau\bigr\}$ of the affine span $H_\tau$ of $\tau$ in $M_\RR$. Note that since $H_\tau = H_{\bu_\tau,\varphi(\bu_\tau)}$, we have $\varphi(\bu_\tau) = N_\tau$.
\end{xpar}

\subsection{Newton \texorpdfstring{$\QQ$}{Q}-polyhedra and their normal fans}\label{2.2}

\begin{xpar}[Conventions on fans]\label{X:convention-2.2-A}
Let $\Sigma$ be a fan in $N_\RR$. For $0 \leq d \leq n$, let $\Sigma[d]$ denote the set of $d$-dimensional cones $\sigma$ in $\Sigma$. In particular, $\Sigma[1]$ is the set of rays in $\Sigma$, and $\Sigma[n]$ is the set of full-dimensional cones in $\Sigma$. We also denote by $\Sigma[\max]$ the set of maximal cones in $\Sigma$, and denote by $\abs{\Sigma}$ the support of $\Sigma$, i.e. $\abs{\Sigma} = \bigcup\{\sigma \colon \sigma \in \Sigma\}$. In this paper we usually consider fans $\Sigma$ in $N_\RR^+$ satisfying $\abs{\Sigma} = N_\RR^+$, in which case $\Sigma[\max] = \Sigma[n]$. 

We also usually use the letter $\rho$ for rays in $\Sigma$ instead of $\sigma$ (akin to how we use a different letter $\tau$ for facets of $\Upgamma(f)$ instead of $\varsigma$, cf. \ref{X:convention-1.2-A}), and we let $\bu_\rho = (u_{\rho,i})_{i=1}^n$ denote the first lattice point on a ray $\rho$ in $\Sigma$, i.e. the unique primitive generator in $N^+$ of $\rho$. In addition, given two \ChMingHao{convex rational polyhedral cones} $\sigma$ and $\sigma'$ in $N_\RR$, we write $\sigma' \prec \sigma$ if $\sigma'$ is a face of $\sigma$. We also write $\sigma[d]$ for the set of $d$-dimensional faces $\sigma' \prec \sigma$, and write $\dim(\sigma)$ for the dimension of $\sigma$. \ChMingHao{Finally, for $S \subset N_\RR$, we write $\langle S \rangle$ for the cone in $N_\RR^+$ generated by $S$.}
\end{xpar}

\begin{xpar}[Normal fans]\label{X:normal-fans}
Every Newton $\QQ$-polyhedron $\Upgamma$ also naturally induces a fan $\Sigma$ in $N_\RR^+$, called the {\sffamily normal fan} of $\Upgamma$, whose cones $\sigma$ correspond in an inclusion-reversing manner with faces $\varsigma \prec \Upgamma$. Namely, let $\varphi$ be the piecewise linear, convex, rational function associated to $\Upgamma$, and we define: \[
    \Sigma \; := \; \left\{\sigma_\ba \colon \ba \in M_\RR^+\right\}
\]
where for each $\ba \in \ChMingHao{\Upgamma}$, \[
    \sigma_\ba \; := \; \bigl\{\bu \in N_\RR^+ \colon \varphi(\bu) = \ba \cdot \bu\bigr\}.
\]
This is a closed convex cone in $N_\RR^+$. Indeed, if $\bu_1$, $\bu_2 \in \sigma_\ba$, then $\ba \cdot (\bu_1+\bu_2) = \ba \cdot \bu_1 + \ba \cdot \bu_2 = \varphi(\bu_1) + \varphi(\bu_2) \leq \varphi(\bu_1+\bu_2) \leq \ba \cdot (\bu_1+\bu_2)$, which forces equality throughout, i.e. $\bu_1 + \bu_2 \in \sigma_\ba$. In particular, we obtain an alternative characterization of $\sigma_\ba$:

\begin{corollary}\label{C:linearity-cone}
For $\ba \in \ChMingHao{\Upgamma}$, $\sigma_\ba$ is the largest closed convex cone in $N_\RR^+$ on which $\varphi$ is the linear function $\bu \mapsto \ba \cdot \bu$. \qed
\end{corollary}
\end{xpar}

\begin{xpar}\label{X:first-meet-locus}
Our next goal is to explicate $\sigma_\ba$ further; in particular, we will see that $\sigma_\ba$ is a convex rational polyhedral cone in $N_\RR^+$ . To do this, let us first introduce a notion dual to $\sigma_\ba$. Namely, for each $\bu \in N_\RR^+$, the {\sffamily first meet locus} of $\bu$ is defined as: \[
    \varsigma_\bu \; = \; \{\ba \in \ChMingHao{\Upgamma} \colon \varphi(\bu) = \ba \cdot \bu\} \; \prec \; \Upgamma.
\]
\ChMingHao{Note that $\varsigma_\bzero = \Upgamma$}. Here are some other observations about $\varsigma_\bu$: \begin{enumerate}
    \item \ChMingHao{For $\bzero \neq \bu \in N_\RR^+$, $\varsigma_\bu = H_{\bu,\varphi(\bu)} \cap \Upgamma$, i.e. $H_{\bu,\varphi(\bu)}$ is a supporting hyperplane of $\Upgamma$. Every proper face of $\Upgamma$ is $\varsigma_\bu$ for some $\bzero \neq \bu \in N_\RR^+$.}
    \item \ChMingHao{Every facet $\tau$ of $\Upgamma$ is $\varsigma_{\bu_\tau}$ for a unique primitive vector ${\bu_\tau} \in N^+$, namely the one that is normal to the affine span $H_\tau$ of $\tau$ in $M_\RR$.} 
    \item For each $i \in [n]$, the following statements are equivalent: \begin{enumerate}
    \item $\bu_i = 0$.
    \item $\varsigma_\bu$ is non-compact in the $i\textsuperscript{th}$ coordinate, i.e. $\varsigma_\bu + \RR_{\geq 0}\be_i^\vee = \varsigma_\bu$.
    \item There exists $\ba \in \varsigma_\bu$ such that $\ba + \be_i^\vee \in \varsigma_\bu$.
    \end{enumerate}
\end{enumerate}
In particular, (iii) says that $\varsigma_\bu$ is compact if and only if all coordinates of $\bu$ are non-zero. We will also need the following lemma:
\end{xpar}

\ChMingHao{\begin{lemma}\label{L:intersection-of-first-meet-loci-and-cones}\ \begin{enumerate}
    \item Let $\bu_1,\bu_2 \in N_\RR^+$. Then $\varsigma_{\bu_1} \cap \varsigma_{\bu_2} \subset \varsigma_{\bu_1+\bu_2}$, with equality if and only if $\varsigma_{\bu_1} \cap \varsigma_{\bu_2} \neq \varnothing$.
    \item Let $\ba, \ba' \in \Upgamma$. For every $0 < t < 1$, $\sigma_\ba \cap \sigma_{\ba'} = \sigma_{t\ba + (1-t)\ba'}$.
\end{enumerate}
\end{lemma}}

The next lemma is the key step towards explicating $\sigma_\ba$:

\begin{lemma}\label{L:extremal-rays}
For \ChMingHao{$\ba \in \Upgamma$} and $\bu \in N_\RR^+$, $\bu$ generates an extremal ray of $\sigma_\ba$ if and only if $\ba \in \varsigma_\bu \prec^1 \Upgamma$.
\end{lemma}

\begin{proof}
\ChMingHao{We may assume $\bu \neq \bzero$.} For the reverse implication, let $\bu_1$, $\bu_2 \in \sigma_\ba$ such that $\bu_1 + \bu_2 \in \ChMingHao{\langle \bu \rangle}$. By Lemma~\ref{L:intersection-of-first-meet-loci-and-cones}(i), $\varsigma_{\bu_1} \cap \varsigma_{\bu_2} = \varsigma_{\bu_1+\bu_2} = \varsigma_{\bu}$. By hypothesis, $\varsigma_{\bu_i} = \varsigma_\bu$ for $i=1,2$. Since the affine span of $\varsigma_\bu$ is an affine hyperplane in $N_\RR$, this means that $\bu_i \in \ChMingHao{\langle \bu \rangle}$ for $i=1,2$, as desired.

Next, we show the forward implication. Since $\bu \in \sigma_\ba$, $\ba \in \varsigma_\bu$. It remains to show that $\varsigma_\bu$ is maximal among all proper faces $\varsigma \prec \Upgamma$ containing $\ba$. To this end, let $\varsigma$ be a proper face of $\Upgamma$ containing $\varsigma_\bu$, and let $\bzero \neq \bu' \in \sigma_\ba$ such that $\varsigma = \varsigma_{\bu'}$. For every $i \in [n]$ such that $\bu_i = 0$, we have \[
    \varsigma_\bu + \RR_{\geq 0}\be_i^\vee \; = \; \varsigma_\bu \; \subset \; \varsigma_{\bu'}
\]
which implies $\bu'_i = 0$ (cf. \ref{X:first-meet-locus}(iii)). Therefore, for $N \gg 0$, $N\bu - \bu' \in N_\RR^+$. In fact, we \emph{\underline{claim}} that for $N \gg 0$, we have $N\bu - \bu' \in \sigma_\ba$. If not, for every $N \gg 0$, we have $N\bu - \bu' \in N_\RR^+ \smallsetminus \sigma_\ba$, i.e. there exists $\ba'_N \in \vertex(\Upgamma)$ so that \[
    \varphi\left(\bu - \frac{1}{N}\bu'\right) \; = \; \ba'_N \cdot \left(\bu - \frac{1}{N}\bu'\right) \; < \; \ba \cdot \left(\bu - \frac{1}{N}\bu'\right).
\]
Since $\vertex(\Upgamma)$ is finite, there exists a constant subsequence $(\ba'_{N_k})_{k \geq 1} = (\ba',\ba',\ba',\dotsb)$ of $(\ba'_N)_{N \gg 0}$. For all $k \geq 1$, we have \begin{align}\label{EQ:sequence-inequality}
    \ba' \cdot \left(\bu - \frac{1}{N_k}\bu'\right) \; < \; \ba \cdot \left(\bu - \frac{1}{N_k}\bu'\right).
\end{align}
Letting $k \to \infty$, $\ba' \cdot \bu \leq \ba \cdot \bu$. But $\ba \in \varsigma_\bu$, so $\ba \cdot \bu = \varphi(\bu) \leq \ba' \cdot \bu$. Thus, $\ba' \cdot \bu = \ba \cdot \bu$, i.e. $\ba' \in \varsigma_\bu$. In addition, $\varsigma_\bu \subset \varsigma_{\bu'}$, so $\ba, \ba' \in \varsigma_{\bu'}$, i.e. $\ba' \cdot \bu' = \varphi(\bu') = \ba \cdot \bu'$. However, both $\ba' \cdot \bu = \ba \cdot \bu$ and $\ba' \cdot \bu' = \ba \cdot \bu'$ contradict \eqref{EQ:sequence-inequality}. Consequently, our earlier \emph{\underline{claim}} holds, i.e. by replacing $\bu$ by a sufficiently large multiple of itself, we may assume $\bu-\bu' \in \sigma_\ba$. Since $\bu$ generates an extremal ray of $\sigma_\ba$, $\bu'$ and $\bu-\bu'$ both lie in $\RR_{\geq 0}\bu$. In particular, $\varsigma = \varsigma_{\bu'} = \varsigma_\bu$, as desired. 
\end{proof}

\begin{corollary}\label{C:structure-of-cones}
For $\ba \in \ChMingHao{\Upgamma}$, $\sigma_\ba$ is a convex rational polyhedral cone in $N_\RR^+$. More precisely, \[
    \sigma_\ba \; = \; \ChMingHao{\bigl\langle \bu_\tau \colon \ba \in \tau \prec^1 \Upgamma\bigr\rangle.}
\]
In particular, $\sigma_\ba \neq \{\bzero\}$ if and only if $\ba$ lies on the boundary of $\Upgamma$.
\end{corollary}

\begin{proof}
Since $\sigma_\ba$ is a closed convex cone in $N_\RR^+$, $\sigma_\ba$ is generated by its extremal rays, cf. \cite[Theorem 18.5]{rockafellar-convex-analysis}. \ChMingHao{Apply the preceding lemma.}
\end{proof}

In what follows, recall that the relative interior $\relint(\sigma)$ of a cone $\sigma$ in $N_\RR$ is the interior of $\sigma$ in its \ChMingHao{$\RR$-span} in $N_\RR$, \ChMingHao{and the relative interior $\relint(\varsigma)$ of a polyhedron $\varsigma$ in $M_\RR$ is the interior of $\varsigma$ in its affine span in $M_\RR$.}

\begin{corollary}\label{C:cone-face-correspondence}
For $\ba \in \ChMingHao{\Upgamma}$ and $\bu \in N_\RR^+$, the following statements are equivalent: \begin{enumerate}
    \item $\bu \in \relint(\sigma_\ba)$.
    \item $\varsigma_\bu = \bigcap\{\tau \prec^1 \Upgamma \colon \ba \in \tau\}$, \ChMingHao{where $\bigcap\varnothing := \Upgamma$}.
    \item $\ba \in \relint(\varsigma_\bu)$.
\end{enumerate}
Moreover, for $\bu \in \sigma_\ba$, $\bigcap\{\tau \prec^1 \Upgamma \colon \ba \in \tau\} \prec \varsigma_\bu$.
\end{corollary}

\begin{proof}
\ChMingHao{We may assume $\bu \neq \bzero$.} Note that (ii)$\iff$(iii), since $\bigcap\{\tau \prec^1 \Upgamma \colon \ba \in \tau\}$ is the unique face $\varsigma$ of $\Upgamma$ such that $\ba \in \relint(\varsigma)$. For (i)$\iff$(ii), it suffices to focus on the case $\bu \in \sigma_\ba$ (since otherwise, $\ba \notin \varsigma_\bu$), and by the preceding corollary $\bu = \sum_{\ba \in \tau \prec^1 \Upgamma}{\lambda_\tau\bu_\tau}$ for some $\lambda_\tau \in \RR_{\geq 0}$. By Lemma~\ref{L:intersection-of-first-meet-loci-and-cones}(i), we have \[
    \varsigma_\bu \; = \; \bigcap\bigl\{\varsigma_{(\lambda_\tau\bu_\tau)} \colon \ba \in \tau \prec^1 \Upgamma\bigr\} \; = \; \bigcap\bigl\{\tau \prec^1 \Upgamma \colon \ba \in \tau \textrm{ and } \lambda_\tau > 0\bigr\}
\]
which contains $\bigcap\{\tau \prec^1 \Upgamma \colon \ba \in \tau\}$ as a face. It remains to note that \ChMingHao{we can arrange for $\{\lambda_\tau \colon \ba \in \tau \prec^1 \Upgamma\} \subset \RR_{>0}$} if and only if $\bu \in \relint(\sigma_\ba)$.
\end{proof}

\begin{xpar}\label{X:cone-face-correspondence}
The preceding corollary sets up a natural correspondence between: \begin{align*}
    \bigl\{\textrm{faces $\varsigma \prec \Upgamma$}\bigr\} \quad &\xleftrightarrow{\quad\quad} \quad \bigl\{\textrm{cones $\sigma$ in $\Sigma$}\bigr\} \\
    \varsigma \qquad &\xmapsto{\quad\quad} \qquad \sigma_\varsigma \\
    \varsigma_\sigma \qquad &\xmapsfrom{\quad\quad} \qquad \sigma
\end{align*}
which is defined as follows. Given $\varsigma \prec \Upgamma$, \ChMingHao{$\sigma_\varsigma := \sigma_\ba$ for any $\ba \in \relint(\varsigma)$. We call $\sigma_\varsigma$ the {\sffamily cone in $\Sigma$ dual to $\varsigma$}. Conversely, given $\sigma$ in $\Sigma$, $\varsigma_\sigma := \varsigma_\bu$ for any $\bu \in \relint(\sigma)$. We call $\varsigma_\sigma$ the {\sffamily face of $\Upgamma$ dual to $\sigma$}.} Then: \begin{enumerate}
    \item If faces $\varsigma, \varsigma' \prec \Upgamma$ correspond to cones $\sigma, \sigma'$ in $\Sigma$, then $\varsigma \prec \varsigma'$ if and only if $\sigma \succ \sigma'$. \ChMingHao{Corollary~\ref{C:structure-of-cones} gives the forward implication, and the converse follows from the preceding corollary.}
    \item If a face $\varsigma \prec \Upgamma$ corresponds to a cone $\sigma$ in $\Sigma$, then $\dim(\varsigma) + \dim(\sigma) = n$. This follows by induction on $\dim(\sigma)$, where the induction step is supplied by (i).
    \item If a facet $\tau \prec^1 \Upgamma$ corresponds to a ray $\rho$ in $\Sigma$, note that $\bu_\tau = \bu_\rho$.
\end{enumerate}
\end{xpar}

\begin{corollary}\label{C:fan}
$\Sigma$ is a fan in $N_\RR$ whose support $\abs{\Sigma}$ equals $N_\RR^+$.
\end{corollary}

\begin{proof}
\ChMingHao{This follows from Lemma~\ref{L:intersection-of-first-meet-loci-and-cones}(ii), Corollary~\ref{C:structure-of-cones}, and \ref{X:cone-face-correspondence}.}
\end{proof}

\begin{xpar}[Notation]\label{X:convention-2.2-B}
For $\rho \in \Sigma[1]$, we will denote the facet $\varsigma_\rho = \varsigma_\bu \prec^1 \Upgamma$ dual to $\rho$ by $\tau_\rho$ or $\tau_\bu$ instead, cf. \ref{X:convention-1.2-A}. Likewise, for $\tau \prec^1 \Upgamma$, we denote the ray $\sigma_\tau \in \Sigma[1]$ dual to $\tau$ by $\rho_\tau$ instead, cf. \ref{X:convention-2.2-A}. Then the following corollary is immediate from Corollary~\ref{C:structure-of-cones} and Corollary~\ref{C:cone-face-correspondence}:
\end{xpar}

\begin{corollary}\label{C:cone-face-correspondence-A}\;
For a face $\varsigma \prec \Upgamma$, we have: \[
    \sigma_\varsigma \; = \; \ChMingHao{\bigl\langle}\rho_\tau \colon \varsigma \prec \tau \prec^1 \Upgamma\ChMingHao{\bigr\rangle}.
\]
Dually, for a cone $\sigma$ in $\Sigma$, we have: \[
    \varsigma_\sigma \; = \; \bigcap\bigl\{\tau_\rho \colon \rho \in \sigma[1]\bigr\}, \qquad \ChMingHao{\textrm{where } \; \bigcap\varnothing \; := \; \Upgamma}.
\]
\end{corollary}

The next corollary follows from the preceding corollary, and \ref{X:first-meet-locus}(iii).

\begin{corollary}\label{C:cone-face-correspondence-B}
Let $\varsigma$ be a face of $\Upgamma$, and $\sigma$ be the cone in $\Sigma$ dual to $\varsigma$. For $i \in [n]$, let $\{\be_i^\vee = 0\}$ denote the coordinate hyperplane in $N_\RR$ defined by $\be_i^\vee = 0$. Then the following statements are equivalent: \begin{enumerate}
    \item $\sigma \subset \{\be_i^\vee = 0\}$, i.e. for every $\rho \in \sigma[1]$, $u_{\rho,i} = 0$.
    \item $\varsigma$ is non-compact in the $i\textsuperscript{th}$ coordinate, i.e. $\varsigma + \RR_{\geq 0}\be_i^\vee = \varsigma$.
    \item There exists $\ba \in \varsigma$ such that $\ba + \be_i^\vee \in \varsigma$.
\end{enumerate}
In particular, $\varsigma$ is compact if and only if $\sigma$ is not contained in any coordinate hyperplane $\{\be_i^\vee = 0\}$ in $N_\RR$.
\end{corollary}



\subsection{Fantastacks and multi-weighted blow-ups}\label{2.3}

\begin{xpar}\label{X:subdivision-toric-morphism}
Let $\Sigma$ be a fan in $N_\RR$ whose support $\abs{\Sigma}$ is $N_\RR^+$. Then $\Sigma$ is a subdivision or refinement \cite[p. 130]{cox-little-schenck-toric-varieties} of the standard fan $\Sigma_\std$ in $N_\RR$ generated by the standard cone $\sigma_\std = \ChMingHao{\bigl\langle}\be_i \colon i \in [n]\ChMingHao{\bigr\rangle}$. By \cite[Theorem 3.4.7]{cox-little-schenck-toric-varieties}, there is a toric, proper, birational morphism \[
    X_\Sigma \; \xrightarrow{\quad\uppi_\Sigma\quad} \; X_{\Sigma_\std} \; = \; \AA^n
\]
where $X_\Sigma$ (resp. $X_{\Sigma_\std}$) is the toric variety associated to $\Sigma$ (resp. $\Sigma_\std$).
\end{xpar}

\begin{xpar}\label{X:cox-construction}
While $X_\Sigma$ is possibly singular, there is nevertheless a canonical smooth Artin stack $\sX_\Sigma$ whose good moduli space is $X_\Sigma$, which is supplied by Cox's construction, cf. \cite[\S 5.1]{cox-little-schenck-toric-varieties}. To start, we define a $\ZZ$-lattice homomorphism \[
    \widehat{N} \; := \; \ZZ^{\Sigma[1]} \; \xrightarrow{\quad\beta\quad} \; \ZZ^n \; = \; N
\]
by mapping the standard basis vector $\be_\rho$ indexed by $\rho \in \Sigma[1]$ to the first lattice point $\bu_\rho$ on $\rho$ (\ref{X:convention-2.2-A}). 

Next, let $\overline{\Sigma}$ denote the set of convex rational polyhedral cones $\sigma$ in $N_\RR^+$ such that $\sigma[1] \subset \sigma'[1]$ for some $\sigma' \in \Sigma$, in which case we say $\sigma$ {\sffamily can be inscribed in} $\sigma'$, and write $\sigma \sqsubset \sigma'$. We call $\overline{\Sigma}$ the {\sffamily augmentation} of $\Sigma$. For each cone $\sigma$ in $\overline{\Sigma}$, we associate to it the following \emph{smooth} cone in $\widehat{N}_\RR := \widehat{N} \otimes_\ZZ \RR$: \[
    \widehat{\sigma} \; := \; \ChMingHao{\bigl\langle} \be_\rho \colon \beta(\be_\rho) = \bu_\rho \in \sigma \ChMingHao{\bigr\rangle} \; = \; \ChMingHao{\bigl\langle} \be_\rho \colon \rho \in \sigma[1] \ChMingHao{\bigr\rangle}.
\]
Note that for every $\sigma_1, \sigma_2 \in \overline{\Sigma}$, we have $\sigma_1 \sqsubset \sigma_2$ if and only if $\widehat{\sigma}_1 \prec \widehat{\sigma}_2$. 

Finally, let $\widehat{\Sigma}$ denote the smooth fan $\{ \widehat{\sigma} \colon \sigma \in \overline{\Sigma} \}$, which is generated by $\{\widehat{\sigma} \colon \sigma \in \Sigma[\max]\}$. Then $(\widehat{\Sigma},\beta)$ is the {\sffamily stacky fan associated to $\Sigma$} \cite[Definitions 2.4 and 4.1]{geraschenko-satriano-toric-stacks-I}. By definition, $\beta$ is compatible with the fans $\widehat{\Sigma}$ and $\Sigma$, and thus induces a toric morphism \[
    X_{\widehat{\Sigma}} \; \xrightarrow{\quad\quad} \; X_{\Sigma}.
\]
Let $\beta^\vee \colon N^\vee \to \widehat{N}^\vee$ denote the dual of $\beta$, and let \[
    \GG_\beta := \Hom_\GrpSch\bigl(\Coker(\beta^\vee),\Gm\bigr) \subset \Hom_\GrpSch\bigl(\widehat{N}^\vee,\Gm\bigr) =: T_{\widehat{N}} = \Gm^{\Sigma[1]}
\]
which is the kernel of \[
    \textrm{\small{$\Gm^{\Sigma[1]} = T_{\widehat{N}} := \Hom_\GrpSch(\widehat{N}^\vee,\Gm) \xrightarrow{\quad T_\beta \quad} \Hom_\GrpSch(N^\vee,\Gm) =: T_N = \Gm^n.$}}
\]
Then $\GG_\beta$ acts on $X_{\widehat{\Sigma}}$ via the torus action $\GG_\beta \subset \Gm^{\Sigma[1]} \curvearrowright X_{\widehat{\Sigma}}$, and the above toric morphism $X_{\widehat{\Sigma}} \to X_\Sigma$ descends to the stack quotient: \[
    \sX_\Sigma \; := \; \bigl[X_{\widehat{\Sigma}} \q \GG_\beta\bigr] \; \xrightarrow{\quad\quad} \; X_{\Sigma}
\]
which is a good moduli space of the smooth Artin stack $\sX_\Sigma$ \cite[Example 6.24]{geraschenko-satriano-toric-stacks-I}. We denote by $\Uppi_\Sigma$ the composition \[
    \sX_\Sigma \; \xrightarrow{\quad\quad} \; X_{\Sigma} \; \xrightarrow{\quad\uppi_\Sigma\quad} \; X_{\Sigma_\std} \; = \; \AA^n
\]
and call $\sX_\Sigma$ the {\sffamily fantastack associated to $\Sigma$} \cite{geraschenko-satriano-toric-stacks-I}. This morphism is birational, universally closed and surjective, cf. \cite[Remark 2.1.8]{abramovich-quek-multi-weighted-resolution}. 
\end{xpar}

\begin{definition}\label{D:multi-weighted-blow-up}
If $\Sigma$ is the normal fan of a Newton $\QQ$-polyhedron $\Upgamma$ in $M_\RR^+$, we call the morphism $\Uppi_\Sigma \colon \sX_\Sigma \to \AA^n$ (\ref{X:cox-construction}) the {\sffamily multi-weighted blow-up of $\AA^n$ along $\Upgamma$}. (See the appendix to this section, as well as \ref{X:explicating-multi-weighted-blow-up}, for an explanation to this name.)
\end{definition}

\begin{remark}
Note that $\sX_\Sigma$ is a smooth toric Artin stack with \ChMingHao{trivial generic} stabilizer. More precisely, it has a dense open that is a torus, namely \[
    \Gm^{\Sigma[1]} \q \GG_\beta \; \xrightarrow[\simeq]{\quad T_\beta\quad} \; \Gm^n.
\]
The action $\Gm^{\Sigma[1]} \curvearrowright X_{\widehat{\Sigma}}$ descends to an action $(\Gm^{\Sigma[1]} \q \GG_\beta) \curvearrowright \sX_\Sigma$, which extends the multiplicative action of the torus $(\Gm^{\Sigma[1]} \q \GG_\beta)$ on itself.
\end{remark}

\begin{xpar}[Description of the morphism $\Uppi_\Sigma \colon X_\Sigma \to \AA^n$]\label{X:explicating-multi-weighted-blow-up}
For the remainder of this paper, we will make the obvious identification $\Sigma_\std[1] \longleftrightarrow [n]$. Given a fan $\Sigma$ in $N_\RR$ whose support is $N_\RR^+$, we have $\Sigma[1] \supset \Sigma_\std[1] = [n]$, and we denote the complement $\Sigma[1] \smallsetminus [n]$ by $\Sigma[\exc]$. We call the rays in $\Sigma[\exc]$ {\sffamily exceptional rays}. 

To explicate the morphism $\Uppi_\Sigma \colon \sX_\Sigma \to \AA^n$, first note that the homomorphism $\beta \colon \ZZ^{\Sigma_\fa(1)} \to N = \ZZ^n$ fits into the short exact sequence: \begin{equation}\label{EQ:SES}
    0 \; \xrightarrow{\quad} \; \ZZ^{\Sigma[\exc]} \; \xrightarrow{\quad\alpha \ = \ \tiny \begin{bmatrix}\bB\\-\bI\end{bmatrix}\quad} \; \ZZ^{\Sigma[1]} \; \xrightarrow{\quad\footnotesize \beta \ = \ \begin{bmatrix}\bI&\bB\end{bmatrix}\quad} \; \ZZ^n \; \xrightarrow{\quad} \; 0
\end{equation}
where $\bI$ denotes the identity matrix of order $\#\Sigma[\exc]$ and $\bB$ is the matrix whose $\rho$\textsuperscript{th}-indexed column is $\bu_\rho$ for each $\rho \in \bE(\fa)$. Using \eqref{EQ:SES} and unraveling definitions in \ref{X:cox-construction} then yields the following description: \[
    \begin{tikzcd}
        \sX_\Sigma = \left[X_{\widehat{\Sigma}} \q \GG_\beta\right] = \left[\Spec\bigl(\kk[x_1',\dotsc,x_n']\bigl[x_\rho' \colon \rho \in \Sigma[\exc]\bigr]\bigr) \smallsetminus V(J_\Sigma) \q \Gm^{\Sigma[\exc]}\right] \arrow[to=2-1, "\Uppi_\Sigma"] \\
        \AA^n = \Spec(\kk[x_1,\dotsc,x_n])
    \end{tikzcd}
\]
where: \begin{enumerate}
    \item $\Uppi_\Sigma$ is induced by the $\kk$-algebra homomorphism \[
        \qquad \qquad \Uppi_\Sigma^\# \; \colon \; \kk[x_1,\dotsc,x_n] \; \xrightarrow{\quad\quad} \; \kk[x_1',\dotsc,x_n']\bigl[x_\rho' \colon \rho \in \Sigma[\exc]\bigr]
    \]
    which maps \[
        \qquad \qquad x_i \; \mapsto \; \prod_{\rho \in \Sigma[1]}{(x_\rho')^{u_{\rho,i}}} \; = \; x_i' \cdot \prod_{\rho \in \Sigma[\exc]}{(x_\rho')^{u_{\rho,i}}} \qquad \textrm{for $i \in [n]$}
    \]
    where for every $\rho \in \Sigma[1]$, $u_{\rho,i}$ is the $i$\textsuperscript{th} coordinate of $\bu_\rho$.
    
    \item The ideal $J_\Sigma$ is called the {\sffamily irrelevant ideal}, and equals to: \[
        \qquad \qquad J_\Sigma \; := \; \bigl(x_\sigma' \colon \sigma \in \Sigma\bigr) \; = \; \bigl(x_\sigma' \colon \sigma \in \Sigma[\max]\bigr)
    \]
    where for every cone $\sigma$ in $\Sigma$, or more generally $\sigma$ in $\overline{\Sigma}$, \[
        \qquad \qquad x_\sigma' \; := \; \prod_{\substack{\rho \in \Sigma[1] \smallsetminus \sigma[1]}}{x_\rho'} \; = \; \prod_{\substack{\rho \in \Sigma[1] \\ \varsigma_\sigma \not\subset \tau_\rho}}{x_\rho'}
    \]
    and the open substack of $\sX_\Sigma$ on which $x_\sigma'$ is invertible: \begin{align*}
        \qquad \qquad D_+(\sigma) \; &:= \; \bigl[U_\sigma \q \Gm^{\Sigma[\exc]}\bigr] \\
        \; &:= \; \left[\Spec\bigl(\kk[x_1',\dotsc,x_n']\bigl[x_\rho' \colon \rho \in \Sigma[\exc]\bigr]\bigl[x_\sigma'^{-1}\bigr]\bigr) \q \Gm^{\Sigma[\exc]}\right]
    \end{align*}
    is called the {\sffamily $\sigma$-chart} of $\sX_\Sigma$. Together the charts in $\{D_+(\sigma) \colon \sigma \in \Sigma[\max]\}$ form an open cover for $\sX_\Sigma$.
    
    \item For each $i \in [n]$, the $\ZZ^{\Sigma[\exc]}$-weight of $x_i'$ is $(u_{\rho,i})_{\rho \in \Sigma[\exc]}$, i.e. the $i$\textsuperscript{th} row of the matrix $\bB$ in \eqref{EQ:SES}. For each $\rho \in \Sigma[\exc]$, the $\ZZ^{\Sigma[\exc]}$-weight of $x_\rho'$ is $-\be_\rho \in \ZZ^{\Sigma[\exc]}$. This describes the action $\GG_\beta = \Gm^{\Sigma[\exc]} \curvearrowright X_{\widehat{\Sigma}}$.
    
    \item The orbit-cone correspondence for $X_{\widehat{\Sigma}}$ descends to an orbit-cone correspondence for $\sX_\Sigma$. More precisely, for every cone $\sigma$ in $\overline{\Sigma}$, its corresponding $\Gm^{\Sigma[1]}$-orbit $O_\sigma$ of $X_{\widehat{\Sigma}}$: \[
        \qquad \qquad O_\sigma \; := \; U_\sigma \smallsetminus \bigcup\bigl\{U_{\sigma'} \colon \sigma' \sqsubset \sigma,\; \sigma' \neq \sigma\bigr\} \; \xhookrightarrow{\quad\textrm{closed}\quad} \; U_\sigma
    \]
    descends to its corresponding $(\Gm^{\Sigma[1]} \q \Gm^{\Sigma[\exc]})$-orbit $O(\sigma)$ of $\sX_\Sigma$: \begin{align*}
        \qquad \qquad O(\sigma) \; &:= \; \left[O_\sigma \q \GG_\beta\right] \\
        \; &= \; D_+(\sigma) \smallsetminus \bigcup\bigl\{D_+(\sigma') \colon \sigma' \sqsubset \sigma,\; \sigma' \neq \sigma\bigr\} \\
        \; &= \; V\bigl(x_\rho' \colon \rho \in \sigma[1]\bigr) \; \xhookrightarrow{\quad\textrm{closed}\quad} \; D_+(\sigma).
    \end{align*}
    Since $U_\sigma = \bigsqcup\{O_\sigma \colon \sigma' \sqsubset \sigma\}$, we also have \[
        \qquad \qquad D_+(\sigma) \; = \; \bigsqcup\bigl\{O(\sigma) \colon \sigma' \sqsubset \sigma\bigr\}.
    \]

    \item By the description of $\Uppi_\Sigma$ in (i) and the description of $J_\Sigma$ in (ii), observe that $\Uppi_\Sigma$ maps \[
        \qquad \qquad U \; := \; \textrm{open substack of $\sX_\Sigma$ on which $\prod_{\rho \in \Sigma[\exc]}{x_\rho'}$ is invertible}
    \]
    isomorphically onto the complement of the closed subscheme \[
        \qquad \qquad V\left(\prod_{\substack{i \in [n] \smallsetminus \sigma[1]}}{x_i} \; \colon \; \sigma \in \Sigma[\max] \right) \; \subset \; \AA^n.
    \]
    We call the divisors in $\bigl\{V(x_\rho') \subset \sX_\Sigma \colon \rho \in \Sigma[\exc]\bigr\}$ the {\sffamily irreducible exceptional divisors} of $\Uppi_\Sigma$.
\end{enumerate}
\end{xpar}

The next lemma, and more importantly its corollary, will be useful for later purposes:

\begin{lemma}\label{L:preimage-of-origin}
If a cone $\sigma$ in $\overline{\Sigma}$ satisfies the condition: \[
    D_+(\sigma) \cap \Uppi_\Sigma^{-1}(\bzero) \; \neq \; \varnothing
\]
then $\sigma$ is not contained in any coordinate hyperplane $\{\be_i^\vee = 0\}$ in $N_\RR$.
\end{lemma}

\begin{proof}
Indeed, for $i \in [n]$, we have: \begin{align*}
    \sigma \subset \{\be_i^\vee = 0\} \; &\iff \; u_{\rho,i} = 0 \textrm{ for every } \rho \in \sigma[1] \\
    &\iff \; \Uppi_\Sigma^\#(x_i) \; = \; \prod_{\rho \in \Sigma[1]}{(x_\rho')^{u_{\rho,i}}} \textrm{ is invertible on } D_+(\sigma) \\
    &\iff \; D_+(\sigma) \cap \Uppi_\Sigma^{-1}\bigl(V(x_i)\bigr) = \varnothing \\
    &\ \implies \; D_+(\sigma) \cap \Uppi_\Sigma^{-1}(\bzero) = \varnothing. \qedhere
\end{align*}
\end{proof}

\begin{corollary}\label{C:preimage-of-origin}
We have: \[
    \Uppi_\Sigma^{-1}(\bzero) \; \subset \; \bigsqcup\left\{O(\sigma) \; \colon \; \parbox{5cm}{$\sigma \in \overline{\Sigma}$ not contained in any coordinate hyperplane in $N_\RR$}\right\}.
\]
\end{corollary}

\begin{xpar}\label{X:preimage-of-origin}
If $\Sigma$ is the normal fan of a Newton $\QQ$-polytope $\Upgamma$, note that a cone $\sigma \in \overline{\Sigma}$ is not contained in any coordinate hyperplane in $N_\RR$ if and only if $\bigcap\{\tau_\rho \colon \rho \in \sigma[1]\} \prec \Upgamma$ is compact. Indeed, if $\sigma \in \Sigma$, we have $\bigcap\{\tau_\rho \colon \rho \in \sigma[1]\} = \varsigma_\sigma$ (cf. Corollary~\ref{C:cone-face-correspondence-A}), so the assertion follows from Corollary~\ref{C:cone-face-correspondence-B}. Otherwise, let $\sigma'$ be the \emph{smallest} cone in $\Sigma$ such that $\sigma \sqsubset \sigma'$. Then $\bigcap\{\tau_\rho \colon \rho \in \sigma[1]\} = \bigcap\{\tau_\rho \colon \rho \in \sigma'[1]\} = \varsigma_{\sigma'}$ (cf. \ref{X:cone-face-correspondence} and Corollary~\ref{C:cone-face-correspondence-A}), so the assertion still follows from Corollary~\ref{C:cone-face-correspondence-B}.
\end{xpar}

\subsection*{Appendix to \S\ref{2.3}}

In this appendix, we sketch (without proof) how the multi-weighted blow-ups in Definition~\ref{D:multi-weighted-blow-up} can indeed be interpreted as blow-ups on $\AA^n$. Recall that a {\sffamily Rees algebra} on $\AA^n$ is simply a finitely generated, $\NN$-graded $\sO_{\AA^n}$-subalgebra $\fa_\bullet = \bigoplus_{m \in \NN}{\fa_m \cdot t^m} \; \subset \; \sO_{\AA^n}[t]$ such that $\fa_0 = \sO_{\AA^n}$ and $\fa_m \supset \fa_{m+1}$ for every $m \in \NN$. We say $\fa_\bullet$ is {\sffamily monomial} if for each $m \in \NN$, $\fa_m$ is a monomial ideal of $\kk[x_1,\dotsc,x_n]$, and we also say that $\fa_\bullet$ is {\sffamily integrally closed} if it is integrally closed in $\sO_{\AA^n}[t]$. Then:

\begin{xpar}
Definition~\ref{D:multi-weighted-blow-up} should be understood via a correspondence between: \[
    \bigl\{\textrm{Newton $\QQ$-polyhedra $\Upgamma$ in $M_\RR^+$}\bigr\} \;\longleftrightarrow\; \left\{\parbox{5cm}{integrally closed, monomial Rees algebras $\fa_\bullet$ on $\AA^n$}\right\}.
\]
to be explicated in \ref{X:correspondence}. Namely, for a Newton $\QQ$-polyhedron $\Upgamma$ with normal fan $\Sigma$ and corresponding integrally closed, monomial Rees algebra $\fa_\bullet$ on $\AA^n$, fix a sufficiently large $\ell \in \ChMingHao{\ZZ_{>0}}$ such that the $\ell$\textsuperscript{th} Veronese subalgebra $\fa_{\ell\bullet}$ of $\fa_\bullet$ is generated in degree $1$. Then the multi-weighted blow-up $\Uppi_\Sigma \colon \sX_\Sigma \to \AA^n$ along $\Upgamma$ is the same as the {\sffamily multi-weighted blow-up of $\AA^n$ along $\fa_\ell$}, in the sense of \cite[Definition 3.1.1]{abramovich-quek-multi-weighted-resolution}, and the morphism $\uppi_\Sigma \colon X_\Sigma \to \AA^n$ in \ref{X:subdivision-toric-morphism} is also the schematic blow-up of $\AA^n$ along $\fa_\ell$. 
\end{xpar}

\begin{xpar}\label{X:correspondence}
The above correspondence can be sketched as follows. Given a Newton $\QQ$-polyhedron $\Upgamma$ in $M_\RR^+$, the corresponding integrally closed, monomial Rees algebra $\fa_\bullet$ on $\AA^n$ is given by \[
    \fa_\bullet \; = \; \left\{\pmb{x}^\ba \cdot t^\ell \colon \frac{1}{\ell}\ba \in \Upgamma\right\}
\]
or equivalently, the integral closure in $\sO_{\AA^n}[t]$ of the $\sO_{\AA^n}$-subalgebra generated by the finite set $\bigl\{\pmb{x}^{\ell(\ba)\ba} \cdot t^{\ell(\ba)} \colon \ba \in \vertex(\Upgamma)\bigr\}$, where $\ell(\ba) := \min\{\ell \in \ChMingHao{\ZZ_{>0}} \colon \ell\ba \in M^+\}$ for every $\ba \in \vertex(\Upgamma)$. Conversely, given an integrally closed, monomial Rees algebra $\fa_\bullet$ on $\AA^n$, let $\pmb{x}^{\ba_i} \cdot t^{\ell_i}$ for $i=1,2,\dotsc,r$ be generators of $\fa_\bullet$ as a $\sO_{\AA^n}$-algebra, and the corresponding Newton $\QQ$-polyhedron $\Upgamma$ is \[
    \Upgamma \; = \; \left\{\frac{1}{\ell}\ba \colon \pmb{x}^\ba \cdot t^\ell \in \fa_\bullet\right\}
\]
or equivalently, the convex hull in $M_\RR$ of $\bigcup\bigl\{\frac{1}{\ell_i}\ba_i + M_\RR^+ \colon i \in [r]\bigr\}$. In particular, the above correspondence in particular restricts to a correspondence that is perhaps more familiar to the reader: \[
    \bigl\{\textrm{Newton polyhedra in $M_\RR^+$}\bigr\} \; \xleftrightarrow{\quad\quad} \; \left\{\parbox{5cm}{integrally closed monomial ideals $\fa \subset \kk[x_1,\dotsc,x_n]$}\right\}.
\]
Since these claims are not needed for this paper, we omit their proofs.
\end{xpar}

\section{Preliminaries and examples}

\subsection{A stack-theoretic re-interpretation of a classical embedded desingularization of non-degenerate polynomials}\label{3.1}

We return to the setting at the start of \S\ref{1.1}: namely, let $f = \sum_{\ba \in \NN^n}{c_\ba \cdot \pmb{x}^\ba} \in \kk[x_1,\dotsc,x_n]$ be a non-degenerate polynomial, and let $\Upgamma(f)$ denote the Newton polyhedron of $f$. Let $\Sigma(f)$ denote the normal fan of $\Upgamma(f)$, cf. \S\ref{2.2}. 

It is known in the literature that one can construct, using $\Sigma(f)$, an embedded desingularization of $V(f) \subset \AA^n$ \ChMingHao{above $\bzero \in \AA^n$ (in fact, more is true, cf. the next theorem)}. This construction manifests in various equivalent forms in the literature, e.g. in Varchenko \cite[\S 10]{varchenko-monodromy-and-newton-diagram} and more recently, in Bultot--Nicaise \cite[Proposition 8.31]{bultot-nicaise-log-smooth-models} and Abramovich--Quek \cite[Theorem 5.1.2]{abramovich-quek-multi-weighted-resolution}. As motivated in the introduction (cf. \ref{X:approach-1}, \ref{X:approach-2}, \ref{X:approach-3}), we follow the last approach. Indeed, by following the description in \ref{X:explicating-multi-weighted-blow-up}, the proof of \cite[Theorem 5.1.2]{abramovich-quek-multi-weighted-resolution} shows:

\begin{theorem}\label{T:toric-desingularization}
The multi-weighted blow-up of $\AA^n$ along $\Upgamma(f)$: \[
    \Uppi_{\Sigma(f)} \; \colon \; \sX_{\Sigma(f)} \; \xrightarrow{\quad\quad} \; \AA^n
\]
is a stack-theoretic embedded desingularization of \ChMingHao{$V(f) \cup V(x_1x_2 \dotsm x_n) \subset \AA^n$} above the origin $\bzero \in \AA^n$.
\end{theorem}

\begin{xpar}\label{X:toric-desingularization}
\ChMingHao{This means that $\Uppi_{\Sigma(f)}^{-1}\bigl(V(f) \cup V(x_1x_2 \dotsm x_n)\bigr)$} is a simple normal crossings divisor at \ChMingHao{every point} in $\Uppi_{\Sigma(f)}^{-1}(\bzero)$. To explicate this, we note from \ref{X:explicating-multi-weighted-blow-up}(i) that: \[
    \Uppi_{\Sigma(f)}^\#(f) \; = \; \sum_{\ba \in \NN^n}{c_\ba \cdot (\pmb{x}')^\ba \cdot \prod_{\rho \in \Sigma(f)[\exc]}{(x_\rho')^{\ba \cdot \bu_\rho}}}
\]
where for each $\ba = (a_1,\dotsc,a_n) \in \NN^n$, $(\pmb{x}')^\ba := (x_1')^{a_1} \dotsm (x_n')^{a_n}$. Setting $N_\rho := N_{\tau_\rho} = \inf_{\ba \in \Upgamma(f)}{\ba \cdot \bu_\rho}$ for each $\rho \in \Sigma[1]$ (cf. \ref{X:alt-characterization}, \ref{X:convention-2.2-B}), \ChMingHao{we define} the proper transform of $f$ under $\Uppi_{\Sigma(f)}$ \ChMingHao{as}: \begin{equation}\label{EQ:toric-desingularization}
    f' \; := \; \frac{\Uppi_{\Sigma(f)}^\#(f)}{\prod_{\rho \in \Sigma(f)[1]}{(x_\rho')^{N_\rho}}} \; = \; \sum_{\ba \in \NN^n}{c_\ba \cdot (\pmb{x'})^{\ba-\bn} \cdot \prod_{\rho \in \Sigma(f)[\exc]}{(x_\rho')^{\ba \cdot \bu_\rho - N_\rho}}}
\end{equation}
where $\bn := \bigl(N_i \colon i \in [n]\bigr)$. \ChMingHao{In other words, $V(f') \subset \sX_{\Sigma(f)}$ is the proper transform of all irreducible components of $V(f) \subset \AA^n$ that are not contained in $V(x_1x_2\dotsm x_n) \subset \AA^n$. Note that since $f$ is non-degenerate, $V(f') \subset \sX_{\Sigma(f)}$ is reduced.} Then the preceding theorem is asserting that at \ChMingHao{every point} in $\Uppi_{\Sigma(f)}^{-1}(\bzero) \subset \sX_{\Sigma(f)}$, $V(f') \subset \sX_{\Sigma(f)}$ is smooth, and intersects the smooth divisors \ChMingHao{$\bigl\{V(x_\rho') \subset \sX_{\Sigma(f)} \colon \rho \in \Sigma(f)[1]\bigr\}$} transversely.
\end{xpar}

\begin{xpar}\label{X:motivic-change-of-variables}
We next claim that via an \emph{appropriate} motivic change of variables formula, the desingularization $\Uppi_{\Sigma(f)}$ of $V(f) \subset \AA^n$ \ChMingHao{above $\bzero \in \AA^n$} supplies a set of candidate poles for $Z_{\mot,\bzero}(f;s)$ given by \[
    \Uptheta(f) \; := \; \{-1\} \; \cup \; \bigl\{ s_\tau \colon \tau \prec^1 \Upgamma(f) \textrm{ with } N_\tau > 0 \bigr\} \qquad \textrm{\small{(cf. \eqref{EQ:ratio})}}.
\]
\ChMingHao{However, as it is, the stack $\sX_{\Sigma(f)}$ is typically not Deligne--Mumford, and the change of variables formula in \cite[Theorem 3.41]{yasuda-motivic-integration-deligne-mumford} only applies to $\Uppi_{\Sigma(f)} \colon \sX_{\Sigma(f)} \to \AA^n$ whenever $\sX_{\Sigma(f)}$ is Deligne--Mumford. Likewise, the good moduli space $X_{\Sigma(f)}$ of $\sX_{\Sigma(f)}$ typically has worse than finite quotient singularities, so the formula in \cite[Theorem 4]{veys-motivic-zeta-functions-on-Q-Gorenstein-varieties} does not apply to $\uppi_{\Sigma(f)} \colon X_{\Sigma(f)} \to \AA^n$. 

Nevertheless, there are other formulae in the literature that do apply directly to our context, e.g. \cite[Theorem 5.3.1]{bultot-nicaise-log-smooth-models} or \cite[Theorem 1.3]{satriano-usatine-motivic-integration-Artin-stacks}, although these either demand some background on logarithmic geometry, or background on stacks. Instead, we adopt a more direct and self-contained approach. As already hinted in \ref{X:approach-4}, we circumvent the earlier issue by passing to a further modification that is now a stack-theoretic embedded desingularization of $V(f) \subset \AA^n$ above $\bzero \in \AA^n$ \emph{by a Deligne--Mumford stack}. Namely, this is given by the fantastack $\Uppi_{\pmb{\Upsigma}(f)} \colon \sX_{\pmb{\Upsigma}(f)} \to \AA^n$ associated to any simplicial subdivision $\pmb{\Upsigma}(f)$ of the fan $\Sigma(f)$ that does \emph{not} involve the \emph{addition of new rays}.}
\end{xpar}

\begin{xpar}[Frugal simplicial subdivisions]\label{X:simplicial-refinement}
From \ref{X:simplicial-refinement} to \ref{X:finite-quotient-singularities}, let $\Sigma$ be a fan in $N_\RR$ whose support $\abs{\Sigma}$ is $N_\RR^+$, and we fix a subdivision $\pmb{\Upsigma}$ of $\Sigma$ such that: \begin{enumerate}
    \item $\pmb{\Upsigma}$ is a simplicial fan, i.e. every cone $\pmb{\upsigma}$ in $\pmb{\Upsigma}$ is a simplicial cone.
    \item Every cone $\pmb{\upsigma}$ in $\pmb{\Upsigma}$ can be inscribed in some cone $\sigma$ in $\Sigma$ (in which case one writes $\sigma \sqsubset \sigma'$), cf. \ref{X:cox-construction}.
\end{enumerate}
Such a $\pmb{\Upsigma}$ always exists by \cite[Lemma 2.8]{denef-hoornaert-newton-polyhedra}, and we call any such $\pmb{\Upsigma}$ a {\sffamily frugal simplicial subdivision} of $\Sigma$. Note too that $\pmb{\Upsigma}[1] = \Sigma[1]$.
\end{xpar}

\begin{xpar}\label{X:open-immersion-of-fantastacks}
Let $(\widehat{\Sigma},\beta)$ denote the stacky fan associated to $\Sigma$ in $N_\RR$, cf. \ref{X:cox-construction}. Since $\pmb{\Upsigma}[1] = \Sigma[1]$, the stacky fan associated to $\widehat{\Sigma}$ is of the form $(\widehat{\pmb{\Upsigma}},\beta)$ for the same homomorphism $\beta \colon \ZZ^{\Sigma[1]} = \widehat{N} \to N = \ZZ^n$ appearing in $(\widehat{\Sigma},\beta)$. Moreover, $\widehat{\pmb{\Upsigma}}$ is a sub-fan of $\widehat{\Sigma}$. Indeed, recall from \ref{X:cox-construction} that $\widehat{\pmb{\Upsigma}}$ is generated by $\{\widehat{\pmb{\upsigma}} \colon \pmb{\upsigma} \in \pmb{\Upsigma}\}$, where for every cone $\pmb{\upsigma}$ in $\pmb{\Upsigma}$, \[
    \widehat{\pmb{\upsigma}} \; = \; \ChMingHao{\bigl\langle} \be_\rho \colon \rho \in \pmb{\upsigma}[1] \ChMingHao{\bigr\rangle} \; \subset \; \ZZ^{\pmb{\Upsigma}[1]} \; = \; \widehat{N}.
\]
If $\sigma$ is a cone in $\Sigma$ such that $\pmb{\upsigma} \sqsubset \sigma$, $\widehat{\pmb{\upsigma}}$ is then a face of the cone $\widehat{\sigma} = \ChMingHao{\langle} \be_\rho \colon \rho \in \sigma[1]\ChMingHao{\rangle}$ in $\widehat{\Sigma}$, and hence is in $\widehat{\Sigma}$, as desired. Consequently, the toric morphism induced by the inclusion $\widehat{\pmb{\Upsigma}} \subset \widehat{\Sigma}$ is a $\GG_\beta$-equivariant open immersion $X_{\widehat{\pmb{\Upsigma}}} \hookrightarrow X_{\widehat{\Sigma}}$, which descends to the open immersion of stacks in the following commutative diagram: \begin{equation}\label{EQ:open-immersion-of-fantastacks}
    \begin{tikzcd}
    \sX_{\pmb{\Upsigma}} \; = \; \left[X_{\widehat{\pmb{\Upsigma}}} \q \GG_\beta\right] \arrow[to=1-3, hookrightarrow, "\textrm{open}"] \arrow[to=2-2, swap, "\Uppi_{\pmb{\Upsigma}}"] & & \left[X_{\widehat{\Sigma}} \q \GG_\beta\right] = \; \sX_\Sigma \arrow[to=2-2, "\Uppi_\Sigma"] \\
    & \AA^n &
    \end{tikzcd}
\end{equation}
Explicitly, adopting the notations in the description of $\Uppi_\Sigma \colon \sX_\Sigma \to \AA^n$ in \ref{X:explicating-multi-weighted-blow-up}, the open immersion $\sX_{\pmb{\Upsigma}} \inj \sX_\Sigma$ identifies the former with the following open substack of the latter:
\[
    \sX_{\pmb{\Upsigma}} \; = \; \bigcup\bigl\{D_+(\pmb{\upsigma}) \colon \pmb{\upsigma} \in \pmb{\Upsigma}[\max]\bigr\} \; \subset \; \sX_\Sigma
\]
where for each $\pmb{\upsigma} \in \pmb{\Upsigma}[\max]$,  we set \[
    x_{\pmb{\upsigma}}' \; := \; \prod_{\rho \in \Sigma[1] \smallsetminus \pmb{\upsigma}[1]}{x_\rho'}
\]
and \begin{equation}\label{EQ:pmb-upsigma-chart}
    D_+(\pmb{\upsigma}) \; := \; \left[\Spec\bigl(\kk[x_1',\dotsc,x_n']\bigl[x_\rho' \colon \rho \in \Sigma[\exc]\bigr]\bigl[x_{\pmb{\upsigma}}'^{-1}\bigr]\bigr) \q \Gm^{\Sigma[\exc]}\right] \; \subset \; \sX_\Sigma
\end{equation}
is the $\pmb{\upsigma}$-chart of $\sX_{\pmb{\Upsigma}}$ (\ref{X:explicating-multi-weighted-blow-up}(ii)). Note too that for every $\sigma$ in $\Sigma$ and $\pmb{\upsigma}$ in $\pmb{\Upsigma}$ such that $\pmb{\upsigma} \sqsubset \sigma$, we have $D_+(\pmb{\upsigma}) \subset D_+(\sigma)$, since $x_\sigma'$ divides $x_{\pmb{\upsigma}}'$.
\end{xpar}

\begin{xpar}\label{X:finite-quotient-singularities}
Since $\pmb{\Upsigma}$ is a simplicial fan, $\sX_{\pmb{\Upsigma}}$ has finite stabilizers, i.e. $X_{\pmb{\Upsigma}}$ has finite quotient singularities. While this assertion is classical in toric geometry \cite[Theorem 11.4.8]{cox-little-schenck-toric-varieties}, we will need, for each $\pmb{\upsigma} \in \pmb{\Upsigma}[\max]$, an explicit presentation of the $\sigma$-chart $D_+(\pmb{\upsigma}) \subset \sX_{\pmb{\Upsigma}}$ as the stack quotient of a smooth $\kk$-scheme by an action of a finite abelian group. This presentation will be used later in \ref{X:relative-canonical-divisor}.

Let us continue from the expression in \eqref{EQ:pmb-upsigma-chart}. Firstly, since $x_\rho'$ is invertible on $D_+(\pmb{\upsigma})$ for $\rho \in \Sigma[\exc] \smallsetminus \pmb{\upsigma}[1]$, and their $\ZZ^{\Sigma[\exc]}$-weights $\{-\be_\rho \colon \rho \in \Sigma[\exc] \smallsetminus \pmb{\upsigma}[1]\}$ are linearly independent over $\ZZ$, we observe from \cite[Lemma 1.3.1]{quek-rydh-weighted-blow-up} that by setting \[
    x_\rho' \; = \; 1 \qquad  \textrm{ for every } \rho \in \Sigma[\exc] \smallsetminus \pmb{\upsigma}[1]
\]
we obtain an isomorphism \begin{equation}\label{EQ:first-isom}
    D_+(\pmb{\upsigma}) \; = \; \left[\Spec\bigl(\kk[x_1',\dotsc,x_n']\bigl[x_\rho' \colon \rho \in \pmb{\upsigma}[\exc]\bigr]\bigl[x_{\pmb{\upsigma}}'^{-1}\bigr]\bigr) \q \Gm^{\pmb{\upsigma}[\exc]}\right]
\end{equation}
where: \begin{enumerate}
    \item $\pmb{\upsigma}[\exc] := \Sigma[\exc] \cap \pmb{\upsigma}[1]$.
    \item $x_{\pmb{\upsigma}}'$ becomes $\prod_{i \in [n]\smallsetminus \pmb{\upsigma}[1]}{x_i'}$.
    \item The action $\Gm^{\pmb{\upsigma}[\exc]} \curvearrowright \Spec(\kk[x_1',\dotsc,x_n'][x_\rho' \colon \rho \in \pmb{\upsigma}[\exc]][x_{\pmb{\upsigma}}'^{-1}])$ is as follows. For each $i \in [n]$, the $\ZZ^{\pmb{\upsigma}[\exc]}$-weight of $x_i'$ is $(u_{\rho,i})_{\rho \in \pmb{\upsigma}[\exc]}$, and for each $\rho \in \pmb{\upsigma}[\exc]$, the $\ZZ^{\pmb{\upsigma}[\exc]}$-weight of $x_\rho'$ is $-\be_\rho$.
\end{enumerate}
Secondly, since $\pmb{\upsigma}$ is simplicial, \[
    \bigl\{\bu_\rho \colon \rho \in \pmb{\upsigma}[1]\bigr\} \; = \; \bigl\{\be_i \colon i \in [n] \cap \pmb{\upsigma}[1]\bigr\} \; \sqcup \; \bigl\{\bu_\rho \colon \rho \in \pmb{\upsigma}[\exc]\bigr\}
\]
is linearly independent, and hence, so is \begin{equation}\label{EQ:lin-indep-set}
    \left\{(u_{\rho,i})_{i \in [n] \smallsetminus \pmb{\upsigma}[1]} \; = \; \bu_\rho \; - \; \ChMingHao{\sum_{i \in [n] \cap \pmb{\upsigma}[1]}{u_{\rho,i}\be_i}} \; \colon \; \rho \in \pmb{\upsigma}[\exc]\right\}.
\end{equation}
Moreover, since $\dim(\pmb{\upsigma}) = n$, we have $\#\pmb{\upsigma}[1] = n$, so that: \begin{align*}
    \#\pmb{\upsigma}[\exc] + n \; &= \; \#\pmb{\upsigma}[\exc] \; + \; \#\bigl([n] \cap \pmb{\upsigma}[1]\bigr) \; + \; \#\bigl([n] \smallsetminus \pmb{\upsigma}[1]\bigr) \\
    &= \; \#\pmb{\upsigma}[1] \; + \; \#\bigl([n] \smallsetminus \pmb{\upsigma}[1]\bigr) \; = \; n \; + \; \#\bigl([n] \smallsetminus \pmb{\upsigma}[1]\bigr)
\end{align*}
i.e. $\#([n] \smallsetminus \pmb{\upsigma}[1]) = \#\pmb{\upsigma}[\exc]$. Consequently, the vectors in \eqref{EQ:lin-indep-set} are the columns of an invertible square matrix $\widetilde{\bB}$ of order $\#\pmb{\upsigma}[\exc]$, which implies that the set of rows of $\widetilde{\bB}$: \[
    \bigl\{ (u_{\rho,i})_{\rho \in \pmb{\upsigma}[\exc]} \colon i \in [n] \smallsetminus \pmb{\upsigma}[1] \bigr\} \; = \; \bigl\{\ZZ^{\pmb{\upsigma}[\exc]}\textrm{-weights of } x_i' \colon i \in [n] \smallsetminus \pmb{\upsigma}[1]\bigr\}
\]
is linearly independent. Together with the fact that $x_i'$ is invertible on $D_+(\pmb{\upsigma})$ for $i \in [n] \smallsetminus \pmb{\upsigma}[1]$, we observe again from \cite[Lemma 1.3.1]{quek-rydh-weighted-blow-up} that by setting \[
    x_i' \; = \; 1 \qquad \textrm{for every } i \in [n] \smallsetminus \pmb{\upsigma}[1]
\]
in \eqref{EQ:first-isom}, we obtain an isomorphism \begin{equation}\label{EQ:second-isom}
    D_+(\pmb{\upsigma}) \; = \; \bigl[\Spec\bigl(\kk\bigl[x_\rho' \colon \rho \in \pmb{\upsigma}[1]\bigr]\bigr) \q \Gmu\bigr]
\end{equation}
where: \begin{enumerate}
    \item $\Gmu := \Hom_{\GrpSch}(A,\Gm)$, where $A$ is the finite abelian group \[
        \qquad A \; := \; \frac{\ZZ^{\pmb{\upsigma}[\exc]}}{\bigl\langle (u_{\rho,i})_{\rho \in \pmb{\upsigma}[\exc]} \colon i \in [n] \smallsetminus \pmb{\upsigma}[1] \bigr\rangle}.
    \]
    \item Letting $\overline{(-)}$ denote the quotient $\ZZ^{\pmb{\upsigma}[\exc]} \twoheadrightarrow A$, we specify the action $\Gmu \curvearrowright \Spec(\kk[x_\rho' \colon \rho \in \pmb{\upsigma}[1]])$ as follows. If $i \in [n] \cap \pmb{\upsigma}[1]$, the $A$-weight of $x_i'$ is $\overline{(u_{\rho,i})}_{\rho \in \pmb{\upsigma}[\exc]}$. If $\rho \in \pmb{\upsigma}[\exc]$, the $A$-weight of $x_\rho'$ is $-\overline{\be_\rho}$. 
\end{enumerate}
Since $\{D_+(\pmb{\upsigma}) \colon \pmb{\upsigma} \in \pmb{\Upsigma}[\max]\}$ covers $\sX_{\pmb{\Upsigma}}$, the expression in \eqref{EQ:second-isom} in particular shows that $\sX_{\pmb{\Upsigma}}$ has finite stabilizers.
\end{xpar}

\begin{xpar}\label{X:relative-canonical-divisor}
In this paragraph, we compute the relative canonical divisor $K_{\Uppi_{\pmb{\Upsigma}}}$ of $\Uppi_{\pmb{\Upsigma}}$. For each $\pmb{\upsigma} \in \pmb{\Upsigma}[\max]$, recall that the composition \[
    \Uppi_{\pmb{\Upsigma}(\pmb{\upsigma})} \; \colon \; D_+(\pmb{\upsigma}) \; \stackrel{\eqref{EQ:second-isom}}{\longeq} \; \bigl[\Spec\bigl(\kk\bigl[x_\rho' \colon \rho \in \pmb{\upsigma}[1]\bigr]\bigr) \q \Gmu\bigr] \; \xhookrightarrow{\;\textrm{open}\;} \; \sX_{\pmb{\Upsigma}} \; \xrightarrow{\quad\Uppi_{\pmb{\Upsigma}}\quad} \; \AA^n
\]
is induced by the $\kk$-algebra homomorphism \begin{align*}
    \Uppi_{\pmb{\Upsigma}}(\pmb{\upsigma})^\# \; \colon \; \kk[x_1,\dotsc,x_n] \; &\xrightarrow{\quad\quad} \; \kk\bigl[x_\rho' \colon \rho \in \pmb{\upsigma}[1]\bigr] \\
    x_i \; &\xmapsto{\quad\quad} \; \prod_{\rho \in \pmb{\upsigma}[1]}{(x_\rho')^{u_{\rho,i}}} \; =: \; \alpha_i.
\end{align*}
We then compute, for each $i \in [n]$: \begin{equation}\label{EQ:pullback-of-dx_i}
    \Uppi_{\pmb{\Upsigma}}(\pmb{\upsigma})^\ast\bigl(dx_i\bigr) \; = \; \sum_{\rho \in \pmb{\upsigma}[1]}{\frac{u_{\rho,i}\alpha_i}{x_\rho'} \cdot dx_\rho'}.
\end{equation}
Letting $\fS([n],\pmb{\upsigma}[1])$ denote the set of bijections $\theta \colon [n] \xrightarrow{\simeq} \pmb{\upsigma}[1]$, we have: \begin{align*}
    &\Uppi_{\pmb{\Upsigma}}(\pmb{\upsigma})^\ast\bigl(dx_1 \wedge dx_2 \wedge \dotsb \wedge dx_n\bigr) \\
    &\stackrel{\ChMingHao{\eqref{EQ:pullback-of-dx_i}}}{\longeq} \; \sum_{\theta \in \fS([n],\pmb{\upsigma}[1])}{\prod_{i \in [n]}{\frac{u_{\theta(i),i}\alpha_i}{x_{\theta(i)}'}} \cdot dx_{\theta(1)}' \wedge dx_{\theta(2)}' \wedge \dotsb \wedge dx_{\theta(n)}'} \\
    &\longeq \; \frac{\prod_{i \in [n]}{\alpha_i}}{\prod_{\rho \in \pmb{\upsigma}[1]}{x_\rho'}} \cdot \left(\sum_{\theta \in \fS([n],\pmb{\upsigma}[1])}{\prod_{i \in [n]}}{u_{\theta(i),i}} \cdot dx_{\theta(1)}' \wedge dx_{\theta(2)}' \wedge \dotsb \wedge dx_{\theta(n)}'\right) \\
    &\longeq \; \prod_{\rho \in \pmb{\upsigma}[1]}{(x_\rho')^{\abs{\bu_\rho}-1}} \cdot \bigl( \det(\bB_{\pmb{\upsigma}}) \cdot \wedge_{\rho \in \pmb{\upsigma}[1]}{dx_\rho'} \bigr)
\end{align*}
where: \begin{enumerate}
    \item $\abs{\bu_\rho} := \uu_{\rho,1}+\uu_{\rho,2}+\dotsb+\uu_{\rho,n}$,
    \item $\bB_{\pmb{\upsigma}}$ denotes the square matrix of order $n$ whose $\rho$\textsuperscript{th} column is the vector $\bu_\rho$ for $\rho \in \pmb{\upsigma}[1]$, \ChMingHao{which is} invertible since $\pmb{\upsigma}$ is simplicial,
    \item $\wedge_{\rho \in \pmb{\upsigma}[1]}{dx_\rho'} := dx_{\uptheta(1)}' \wedge dx_{\uptheta(2)}' \wedge \dotsb \wedge dx_{\uptheta(n)}'$ for a \underline{fixed} $\uptheta \in \fS([n],\pmb{\upsigma}[1])$.
\end{enumerate}
 From the above computation, we obtain \[
    K_{\Uppi_{\pmb{\Upsigma}}}|_{D_+(\pmb{\upsigma})} \; = \; \sum_{\rho \in \pmb{\upsigma}[1]}{\bigl(\abs{\bu_\rho}-1\bigr) \cdot V(x_\rho')}.
\]
Finally, since $\{D_+(\pmb{\upsigma}) \colon \pmb{\upsigma} \in \pmb{\Upsigma}[\max]\}$ is an open cover of $\sX_{\pmb{\Upsigma}}$, we deduce that \begin{equation}\label{EQ:relative-canonical-divisor}
    K_{\Uppi_{\pmb{\Upsigma}}} \; = \; \sum_{\rho \in  \Sigma[1]}{\bigl(\abs{\bu_\rho}-1\bigr) \cdot V(x_\rho')} .
\end{equation}
\end{xpar}

\begin{xpar}\label{X:motivic-change-of-variables-II}
Returning to our claim in \ref{X:motivic-change-of-variables}, fix a frugal simplicial subdivision $\pmb{\Upsigma}(f)$ of the normal fan $\Sigma(f)$. We then have: \[
    \begin{tikzcd}
    \Uppi_{\pmb{\Upsigma}(f)}^{-1}\bigl(V(f)\bigr) \arrow[to=2-1, twoheadrightarrow, swap, "\textrm{coarse space}"] \arrow[to=1-2, "\textrm{closed}", hookrightarrow] & \sX_{\pmb{\Upsigma}(f)} \arrow[to=2-2, swap, "\textrm{coarse space}", twoheadrightarrow] \arrow[to=1-3, hookrightarrow, "\textrm{open}"] & \sX_{\Sigma(f)} \arrow[to=1-4, "\Uppi_{\Sigma(f)}"] & \AA^n \\
    \uppi_{\pmb{\Upsigma}(f)}^{-1}\bigl(V(f)\bigr) \arrow[to=2-2, "\textrm{closed}", hookrightarrow] & X_{\pmb{\Upsigma}(f)} \arrow[to=1-4, "\uppi_{\Sigma(f)}", swap, bend right=15]
    \end{tikzcd}
\]
where: \begin{enumerate}
    \item $\uppi_{\pmb{\Upsigma}(f)}$ is proper and birational.
    \item $X_{\pmb{\Upsigma}(f)}$ has finite quotient singularities (\ref{X:finite-quotient-singularities}).
    \item $\uppi_{\pmb{\Upsigma}(f)}^{-1}\bigl(V(f)\bigr)$ is a $\QQ$-simple normal crossings divisor \ChMingHao{\cite[Definition 1.6]{veys-motivic-zeta-functions-on-Q-Gorenstein-varieties}} at \ChMingHao{every point} in $\uppi_{\pmb{\Upsigma}(f)}^{-1}(\bzero) \subset X_{\pmb{\Upsigma}(f)}$. Indeed, $\Uppi_{\pmb{\Upsigma}(f)}$ \ChMingHao{factors as $\sX_{\pmb{\Upsigma}(f)} \xhookrightarrow{\textrm{open}} \sX_{\Sigma(f)} \xrightarrow{\Uppi_{\Sigma(f)}} \AA^n$ in the above diagram}. We therefore deduce, from \eqref{EQ:toric-desingularization}, that: \begin{equation}\label{EQ:motivic-change-of-variables-II}
        \qquad \qquad \Uppi_{\pmb{\Upsigma}(f)}^{-1}\bigl(V(f)\bigr) \; = \; V(f') + \sum_{\rho \in \Sigma(f)[1]}{N_\rho \cdot V(x_\rho')}
    \end{equation}
    where each $V(x_\rho')$, as well as $V(f')$, is now regarded as a divisor in $\sX_{\pmb{\Upsigma}(f)} \xhookrightarrow{\textrm{open}} \sX_{\Sigma(f)}$. By Theorem~\ref{T:toric-desingularization}, $\Uppi_{\pmb{\Upsigma}(f)}^{-1}\bigl(V(f)\bigr)$ is a simple normal crossings divisor at \ChMingHao{every point} in $\ChMingHao{\Uppi_{\pmb{\Upsigma}(f)}^{-1}(\bzero) \subset \sX_{\pmb{\Upsigma}(f)}}$. It remains to note that $\uppi_{\pmb{\Upsigma}(f)}^{-1}\bigl(V(f)\bigr)$ is the coarse space of $\Uppi_{\pmb{\Upsigma}(f)}^{-1}\bigl(V(f)\bigr)$, since the coarse space $\sX_{\pmb{\Upsigma}(f)} \to X_{\pmb{\Upsigma}(f)}$ maps the latter onto the former.
\end{enumerate}
That is, $\uppi_{\pmb{\Upsigma}(f)} \colon X_{\pmb{\Upsigma}(f)} \to \AA^n$ is an {\sffamily embedded $\QQ$-desingularization of $V(f) \subset \AA^n$ above the origin $\bzero \in \AA^n$}, in the sense that it satisfies (i), (ii) and (iii) above.

We additionally note that the motivic change of variables formula in \cite[Theorem 4]{veys-motivic-zeta-functions-on-Q-Gorenstein-varieties} applies more generally for any embedded $\QQ$-desingularization $\uppi \colon Y \to X$ of $D_1+D_2 \subset X$ above a closed subscheme $W \subset X$: the proof in loc. cit. works verbatim, once one recognizes that: \begin{enumerate}
    \item[(a)] \cite[Theorem 2]{veys-motivic-zeta-functions-on-Q-Gorenstein-varieties} is a general change of variables rule for the $\QQ$-Gorenstein motivic zeta function via any proper and birational morphism of pure-dimensional $\QQ$-Gorenstein varieties.
    \item[(b)] After applying \cite[Theorem 2]{veys-motivic-zeta-functions-on-Q-Gorenstein-varieties}, the remainder of the proof of \cite[Theorem 4]{veys-motivic-zeta-functions-on-Q-Gorenstein-varieties} only uses the \ChMingHao{weaker hypothesis} that $\uppi^{-1}(D_1+D_2) \subset Y$ is a $\QQ$-simple normal crossings divisor at \ChMingHao{every point} in $\uppi^{-1}(W)$.
\end{enumerate}
We can therefore apply \cite[Theorem 4]{veys-motivic-zeta-functions-on-Q-Gorenstein-varieties} with $\uppi := \uppi_{\pmb{\Upsigma}(f)}$, $D_1 := V(f)$, $D_2 := 0$, and $W = \{\bzero\}$. Together with \eqref{EQ:relative-canonical-divisor} and \eqref{EQ:motivic-change-of-variables-II}, we deduce that $\Uptheta(f) = \{-1\} \cup \bigl\{-\frac{\abs{\bu_\rho}}{N_\rho} \colon \rho \in \Sigma(f)[1]$ with $N_\rho > 0\bigr\}$ is indeed a set of candidate poles for $Z_{\mot,\bzero}(f;s)$.
\end{xpar}

\subsection{A case study for Theorem~\ref{T:fake-poles}}\label{3.2}

\begin{xpar}\label{X:strategy}
In \S\ref{3.1} we explained why there is a set of candidate poles $\Uptheta(f)$ for $Z_{\mot,\bzero}(f;s)$ whose elements, with the \emph{possible} exception of $-1$, are naturally indexed by facets $\tau \prec^1 \Upgamma(f)$ satisfying $N_\tau > 0$. Namely, the preimage of $V(f) \subset \AA^n$ under the multi-weighted blow-up $\Uppi_{\Sigma(f)}$ of $\AA^n$ is a simple normal crossings divisor at every point above $\bzero \in \AA^n$, comprising of: \begin{enumerate}
    \item the proper transform \ChMingHao{of the irreducible components} of $V(f) \subset \AA^n$ \ChMingHao{that are not contained in $V(x_1x_2 \dotsm x_n) \subset \AA^n$}, 
    \item \ChMingHao{the proper transform of $V(x_i) \subset \AA^n$ for every $i \in [n]$ with $x_i \mid f$}, 
    \item and the irreducible exceptional divisors of $\Uppi_{\Sigma(f)}$,
\end{enumerate} 
\ChMingHao{where the irreducible components in (ii) and (iii) are naturally indexed by the facets $\tau \prec^1 \Upgamma(f)$ satisfying $N_\tau > 0$.}

It is therefore natural to imagine that a proof of Theorem~\ref{T:fake-poles} would involve showing that $V(f) \subset \AA^n$ is \emph{also} desingularized by the multi-weighted blow-up of $\AA^n$ along some Newton $\QQ$-polyhedron $\Upgamma^\dag$ obtained from $\Upgamma(f)$ by ``dropping the facets in $\dB$'' (cf. Theorem~\ref{T:refined-desingularization}). Ideally, one hopes that every supporting hyperplane of $\Upgamma(f)$, \emph{except} those intersecting $\Upgamma(f)$ in a face of some facet in $\dB$, should also be a supporting hyperplane of $\Upgamma^\dag$. In this section we show that this idea works for three non-degenerate polynomials.
\end{xpar}

\begin{example}\label{EX:x^2+xy^4+y^3z+z^3}
Let $f=x_1^2+x_1x_2^4+x_2^3x_3+x_3^3$. On the left side of the diagram below, we shaded the facets of $\Upgamma(f)$ that are not contained in any coordinate hyperplane $H_i$ in $M_\RR$. For now the red vertex and dashed lines, and the right side of the diagram, should be ignored. 
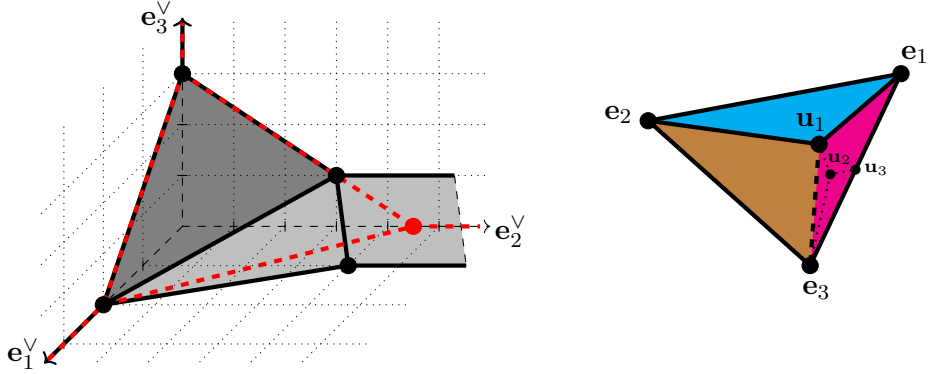
\[
	\begin{tikzpicture}[scale=0.675]

	
	\filldraw[lightgray] (0,0,4)--(3,1,0)--(4,0,2);
	\filldraw[gray] (0,0,4)--(0,3,0)--(3,1,0)--(0,0,4);
	\filldraw[lightgray] (6.3,0,2)--(4,0,2)--(3,1,0)--(5.3,1,0);
	
	\draw[->, thick] (0,3,0)--(0,4.1,0); 
	\draw[->, thick] (0,0,4)--(0,0,7); 
	\draw[->, dashed] (0,0,0)--(0,0,4); 
	\draw[<->, dashed] (0,4.1,0)--(0,0,0)--(6,0,0); 
	
	\node at (6.4,-0.1,0) {$\be_2^\vee$};
	\node at (-0.5,4.1,0) {$\be_3^\vee$};
	\node at (-0.6,0,6.5) {$\be_1^\vee$};
	
	\foreach \x in {1,...,5} 
	\draw (\x,0.1,0) -- (\x,-0.1,0);
	
	\foreach \x in {1,...,3} 
	\draw (0.1,\x,0) -- (-0.1,\x,0);
		
	\foreach \x in {2,4,6} 
	\draw (0,0.1,\x) -- (0,-0.1,\x);
	
	
	\foreach \x in {1,...,5}
	\draw[dotted] (\x,0,0) -- (\x,4.1,0);
	\foreach \x in {1,...,3}
	\draw[dotted] (0,\x,0) -- (6,\x,0);
	
	\foreach \x in {1,...,5}
	\draw[dotted] (\x,0,0) -- (\x,0,7);
	\foreach \x in {2,4,6}
	\draw[dotted] (0,0,\x) -- (6,0,\x);
	
	\foreach \x in {2,4,6}
	\draw[dotted] (0,0,\x) -- (0,4.3,\x);
	\foreach \x in {1,...,3}
	\draw[dotted] (0,\x,0) -- (0,\x,7.3);
	
	\draw[ultra thick] (0,0,6.8)--(0,0,4)--(3,1,0)--(4,0,2)--(0,0,4);
	\draw[ultra thick] (3,1,0)--(5.3,1,0);
	\draw[ultra thick] (4,0,2)--(6.3,0,2);
	\draw[ultra thick] (0,0,4)--(0,3,0)--(3,1,0);
	\draw[ultra thick] (0,3,0)--(0,4,0);
	\draw[dashed] (5.3,1,0)--(6.3,0,2);
	\draw[ultra thick, dashed, red] (0,0,6.8)--(0,0,4)--(0,3,0)--(3,1,0)--(4.5,0,0)--(0,0,4);
	\draw[ultra thick, dashed, red] (4.5,0,0)--(5.8,0,0);
	\draw[ultra thick, dashed, red] (0,3,0)--(0,4,0);

	\filldraw[black] (0,0,4) circle (4.5pt);
	\filldraw[black] (3,1,0) circle (4.5pt);
	\filldraw[black] (4,0,2) circle (4.5pt);
	\filldraw[black] (0,3,0) circle (4.5pt);
	\filldraw[red] (4.5,0,0) circle (4.5pt);
	
	
    \filldraw[cyan] (14,3,0)--(13.053,2.263,1.684)--(11,4,5);
    \filldraw[magenta] (14,3,0)--(13,0,2)--(13.053,2.263,1.684)--(14,3,0);
    \filldraw[brown] (13,0,2)--(13.053,2.263,1.684)--(11,4,5)--(13,0,2);
    
    \draw[ultra thick] (14,3,0)--(11,4,5)--(13,0,2)--(14,3,0);
    \draw[ultra thick] (14,3,0)--(13.053,2.263,1.684)--(11,4,5);
    \draw[ultra thick, dashed] (13.053,2.263,1.684)--(13,0,2);
    \draw[dotted, thick] (13.053,2.263,1.684)--(13.2,1.6,1.5)--(13.5,1.5,1);
    \draw[dotted, thick] (13.2,1.6,1.5)--(13,0,2);
    
    \filldraw[black] (14,3,0) circle (4.5pt);
    \node at (14.2,3.3,-0.2) {$\be_1$};
    
    \filldraw[black] (11,4,5) circle (4.5pt);
    \node at (10.5,4.2,5.2) {$\be_2$};
    
    \filldraw[black] (13,0,2) circle (4.5pt);
    \node at (13.2,-0.4,2.2) {$\be_3$};
    
    \filldraw[black] (13.053,2.263,1.684) circle (4.5pt);
    \node at (12.9,2.7,1.684) {$\bu_1$};
    
    \filldraw[black] (13.2,1.6,1.5) circle (2.5pt);
    \node at (13.38,1.9,1.5) {\tiny{$\bu_2$}};
    
    \filldraw[black] (13.5,1.5,1) circle (2.5pt);
    \node at (13.9,1.5,1) {\tiny{$\bu_3$}};
    
    
    \end{tikzpicture}
\]
Among the shaded facets, we used a \emph{darker} shade for the non--$B_1$-facet \[
    \tau_1 \; := \;\bigl\{\ba \in \Upgamma(f) \colon \ba \cdot \bu_1 = 18\bigr\} \qquad \textrm{where } \bu_1 \; := \; 9\be_1+4\be_2+6\be_3 
\]
with \ChMingHao{candidate pole} $-\frac{19}{18}$, and used a \emph{lighter} shade for the two $B_1$-facets \begin{align*}
    \tau_2 \; &:= \; \bigl\{\ba \in \Upgamma(f) \colon \ba \cdot \bu_2 = 8\bigr\} \qquad \textrm{where } \bu_2 \; := \; 4\be_1+\be_2+5\be_3 \\
    \tau_3 \; &:= \; \bigl\{\ba \in \Upgamma(f) \colon \ba \cdot \bu_3 = 1\bigr\} \qquad \textrm{where } \bu_3 \; := \; \be_1+\be_3
\end{align*}
with \ChMingHao{candidate poles} $-\frac{5}{4}$ and $-2$ respectively. Together $\tau_2$ and $\tau_3$ form a pair $\dB$ of adjacent $B_1$-facets of $\Upgamma(f)$ with consistent base direction $3$. Then Theorem~\ref{T:fake-poles} asserts that $\Uptheta^{\dag,\dB}(f) = \{-1,-\frac{19}{18}\} \subsetneq \{-1,-\frac{19}{18},-\frac{5}{4},-2\} = \Uptheta(f)$ is also a set of candidate poles for $Z_{\mot,\bzero}(f;s)$. 

To show that we execute our idea in \ref{X:strategy}. Indeed, we first note that \[
    \Upgamma(f) \; = \; H_{\bu_1,18}^+ \; \cap \; H_{\bu_2,8}^+ \; \cap \; H_{\bu_3,1}^+ \qquad \textrm{\small{(cf. \ref{D:newton-Q-polyhedra} for notation)}}.
\]
Since $H_{\bu_2,8}$ and $H_{\bu_3,1}$ intersect $\Upgamma(f)$ in the two $B_1$-facets $\tau_2$ and $\tau_3$, we ``drop'' $H_{\bu_2,8}^+$ and $H_{\bu_3,1}^+$ from $\Upgamma(f)$ to define the Newton $\QQ$-polyhedron: \[
    \Upgamma^\dag \; = \; H_{\bu_1,18}^+
\]
which we have outlined in \emph{red} on the left side of the above diagram. 

Illustrated on the right side of the diagram is a cross-section of the normal fan $\Sigma^\dag$ of $\Upgamma^\dag$, except that the rays $\ChMingHao{\langle\bu_2\rangle}$ and $\ChMingHao{\langle\bu_3\rangle}$, as well as the $2$-dimensional cones $\ChMingHao{\langle\bu_1,\bu_2 \rangle}$, $\ChMingHao{\langle\bu_2,\bu_3\rangle}$ and $\ChMingHao{\langle\be_3,\bu_2\rangle}$ which are outlined by \emph{dotted line segments}, are not in $\Sigma^\dag$ but originally in $\Sigma(f)$. In comparison, the $2$-dimensional cone $\ChMingHao{\langle\be_3,\bu_1\rangle}$ in $\Sigma^\dag$, which is outlined by the \emph{dashed thick line segment}, is originally not in $\Sigma(f)$.

Finally, we consider the multi-weighted blow-up of $\AA^3$ along $\Upgamma^\dag$: \[
    \Uppi_{\Sigma^\dag} \; \colon \; \sX_{\Sigma^\dag} \; = \; \left[\Spec\bigl(\kk[x_1',x_2',x_3',u_1]\bigr) \smallsetminus V(x_1',x_2',x_3') \q \Gm\right] \; \xrightarrow{\quad\quad} \; \AA^3
\]
induced by the homomorphism $\Uppi_{\Sigma^\dag}^\# \colon \kk[x_1,x_2,x_3] \to \kk[x_1',x_2',x_3',u_1]$ mapping $x_1 \mapsto x_1'u_1^9$, $x_2 \mapsto x_2'u_1^4$ and $x_3 \mapsto x_3'u_1^6$. We show next that $\Uppi_{\Sigma^\dag}$ is a stack-theoretic embedded desingularization of $V(f) \subset \AA^3$ above $\bzero \in \AA^3$. We first compute that $\Uppi_{\Sigma^\dag}^\#(f) = u_1^{18} \cdot f'$, where the proper transform of $f$ under $\Uppi_{\Sigma^\dag}$ is given by $f' := x_1'^2 + x_1'x_2'^4u_1^7 + x_2'^3x_3' + x_3'^3$. Since $\abs{\Uppi_{\Sigma^\dag}^{-1}(\bzero)} = \abs{V(u_1)} \subset \abs{\sX_{\Sigma^\dag}}$, it suffices to show $V(f'|_{V(u_1)}) = V\bigl(x_1'^2 + x_2'^3x_3' + x_3'^3\bigr) \subset V(u_1)$ is smooth. Indeed, if $J\bigl(f'|_{V(u_1)}\bigr)$ denotes the Jacobian ideal of $f'|_{V(u_1)}$, note \[
    \sqrt{\bigl(f'|_{V(u_1)}\bigr) + J\bigl(f'|_{V(u_1)}\bigr)} \; = \; \sqrt{\bigl(x_1',x_2'^2x_3',x_2'^3+3x_3'^2,x_3'^3\bigr)} \; = \; \bigl(x_1',x_2',x_3'\bigr)
\]
is the unit ideal on $\sX_{\Sigma^\dag}$, as desired.
\end{example}

\begin{remark}
In the above example, note that unlike $\Upgamma(f)$, $\Upgamma^\dag$ has a vertex with non-integer coordinates, namely the \emph{red vertex} $\frac{9}{2}\be_2^\vee$.
\end{remark}

\begin{remark}
Moreover, the morphism $\Uppi_{\Sigma^\dag}$ is the stack-theoretic weighted blow-up of $\AA^3$ along the center $\bigl(x_1^{1/9},x_2^{1/4},x_3^{1/6}\bigr)$, cf. first paragraph of \cite[Example 2.2.1]{abramovich-quek-multi-weighted-resolution}. We also remark $V(f)$ has a semi-quasihomogeneous singularity at $\bzero \in \AA^3$ (with the same weights $9$ on $x_1$, $4$ on $x_2$ and $6$ on $x_3$), and the strong monodromy conjecture is known for semi-quasihomogeneous hypersurfaces, cf. \cite{blanco-budur-robin-monodromy-conjecture}. In fact, the proof also uses weighted blow-ups.
\end{remark}

The next example, together with its remark, shows that the hypothesis in Theorem~\ref{T:fake-poles} that ``$\dB$ has consistent base directions'' \emph{cannot be dropped}:

\begin{example}\label{EX:x^2+yz}
Let $f = x_1^2+x_2x_3$. In the diagram below we shaded only the facets of $\Upgamma(f)$ that are not contained in any coordinate hyperplane $H_i$ in $M_\RR$. As with the previous example we ignore the red/blue vertices and dashed lines for now. 
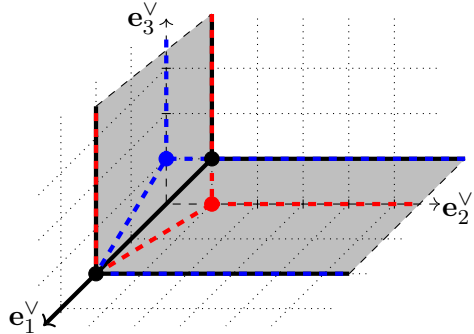
\[
	\begin{tikzpicture}[scale=0.6]
	
	\filldraw[lightgray] (5.55,0,4)--(0,0,4)--(1,1,0)--(6.55,1,0);
	\filldraw[lightgray] (0,3.7,4)--(0,0,4)--(1,1,0)--(1,4.2,0);
	
	\draw[->, thick] (0,0,4)--(0,0,7); 
	\draw[->, dashed] (0,0,0)--(0,0,4); 
	\draw[<->, dashed] (0,4.1,0)--(0,0,0)--(6,0,0); 
	
	\node at (6.4,-0.1,0) {$\be_2^\vee$};
	\node at (-0.5,4.1,0) {$\be_3^\vee$};
	\node at (-0.6,0,6.5) {$\be_1^\vee$};
	
	\foreach \x in {1,...,5} 
	\draw (\x,0.1,0) -- (\x,-0.1,0);
	
	\foreach \x in {1,...,3} 
	\draw (0.1,\x,0) -- (-0.1,\x,0);
		
	\foreach \x in {2,4,6} 
	\draw (0,0.1,\x) -- (0,-0.1,\x);
	
	
	\foreach \x in {1,...,5}
	\draw[dotted] (\x,0,0) -- (\x,4.1,0);
	\foreach \x in {1,...,3}
	\draw[dotted] (0,\x,0) -- (6,\x,0);
	
	\foreach \x in {1,...,5}
	\draw[dotted] (\x,0,0) -- (\x,0,7);
	\foreach \x in {2,4,6}
	\draw[dotted] (0,0,\x) -- (6,0,\x);
	
	\foreach \x in {2,4,6}
	\draw[dotted] (0,0,\x) -- (0,4.3,\x);
	\foreach \x in {1,...,3}
	\draw[dotted] (0,\x,0) -- (0,\x,7.3);
	
	\draw[ultra thick] (0,0,4)--(5.55,0,4);
	\draw[ultra thick] (1,1,0)--(6.55,1,0);
	\draw[ultra thick] (0,0,4)--(1,1,0);
	\draw[ultra thick] (0,0,4)--(0,0,6.8);
	\draw[ultra thick] (0,0,4)--(0,3.7,4);
	\draw[ultra thick] (1,1,0)--(1,4.2,0);
	\draw[dashed] (5.55,0,4)--(6.55,1,0);
	\draw[dashed] (0,3.7,4)--(1,4.2,0);
	\draw[ultra thick, dashed, red] (1,4.2,0)--(1,1,0)--(1,0,0)--(0,0,4)--(0,3.7,4);
	\draw[ultra thick, dashed, red] (1,0,0)--(5.5,0,0);
	\draw[ultra thick, dashed, blue] (6.55,1,0)--(1,1,0)--(0,1,0)--(0,0,4)--(5.55,0,4);
	\draw[ultra thick, dashed, blue] (0,1,0)--(0,3.7,0);
	
	\filldraw[black] (0,0,4) circle (4.5pt);
	\filldraw[black] (1,1,0) circle (4.5pt);
	\filldraw[red] (1,0,0) circle (4.5pt);
	\filldraw[blue] (0,1,0) circle (4.5pt);
    
    \end{tikzpicture}
\]
The two shaded facets of $\Upgamma(f)$: \begin{align*}
    \tau_1 \; &:= \; \bigl\{\ba \in \Upgamma(f) \colon \ba \cdot \bu_1 = 2\bigr\} \qquad \textrm{where } \bu_1 \; := \; \be_1+2\be_2 \\
    \tau_2 \; &:= \; \bigl\{\ba \in \Upgamma(f) \colon \ba \cdot \bu_2 = 2\bigr\} \qquad \textrm{where } \bu_2 \; := \; \be_1+2\be_3
\end{align*}
are adjacent $B_1$-facets with the same \ChMingHao{candidate pole} $-\frac{3}{2}$, but together they form a set of $B_1$-facets with \emph{inconsistent} base directions $2$ and $3$.

Thus, Theorem~\ref{T:fake-poles} does not apply to the set $\dB = \{\tau_1,\tau_2\}$. In fact, our idea in \ref{X:strategy} \emph{fails} in this scenario. Indeed, ``dropping'' both $H_{\bu_1,2}^+$ and $H_{\bu_2,2}^+$ from $\Upgamma(f) = H_{\bu_1,2}^+ \cap H_{\bu_2,2}^+$ yields $\Upgamma^\dag = M_\RR^+$, but the multi-weighted blow-up of $\AA^3$ along $M_\RR^+$ is the identity morphism on $\AA^3$!

Nevertheless, in Theorem~\ref{T:fake-poles} one could take $\dB$ to be either $\{\tau_1\}$ or $\{\tau_2\}$, although in either case $\Uptheta^{\dag,\dB}(f) = \{-1,-\frac{3}{2}\}$ is the same set as $\Uptheta(f)$. In spite of that, our idea in \ref{X:strategy} should still say something of consequence. Namely, for $\dB = \{\tau_1\}$ (resp. $\dB = \{\tau_2\}$), we claim that the multi-weighted blow-up of $\AA^3$ along the Newton polyhedron \[
    \Upgamma^{\dag,\tau_1} \; = \; H_{\bu_2,2}^+ \qquad \textrm{(resp. $\Upgamma^{\dag,\tau_2} \; = \; H_{\bu_1,2}^+$)}
\]
is a stack-theoretic embedded desingularization of $V(f) \subset \AA^3$ above $\bzero \in \AA^3$.

To verify this claim, let us first outline, in the diagram above, the Newton polyhedra $\Upgamma^{\dag,\tau_1}$ and $\Upgamma^{\dag,\tau_2}$ in \emph{blue} and \emph{red} respectively. On the left (resp. right) side of the diagram below, we also sketched a cross-section of the normal fan $\Sigma^{\dag,\tau_1}$ (resp. $\Sigma^{\dag,\tau_2}$) of $\Upgamma^{\dag,\tau_1}$ (resp. $\Upgamma^{\dag,\tau_2}$), keeping the same conventions as before in Example~\ref{EX:x^2+xy^4+y^3z+z^3}. \[
    \begin{tikzpicture}[scale=0.6]
    
    \filldraw[magenta] (11,4,5)--(13,0,2)--(12,3.667,3.333);
    \filldraw[cyan] (14,3,0)--(12,3.667,3.333)--(13,0,2);
    
    \draw[ultra thick] (14,3,0)--(13.333,1,1.333)--(13,0,2)--(12,2,3.5)--(11,4,5)--(12,3.667,3.333)--(14,3,0);
    \draw[thick, dotted] (12,3.667,3.333)--(13.333,1,1.333);
    \draw[ultra thick, dashed] (12,3.667,3.333)--(13,0,2);

    \filldraw[black] (14,3,0) circle (4.5pt);
    \node at (14.2,3.3,-0.2) {$\be_1$};
    
    \filldraw[black] (11,4,5) circle (4.5pt);
    \node at (10.5,4.2,5.2) {$\be_2$};
    
    \filldraw[black] (13,0,2) circle (4.5pt);
    \node at (13.2,-0.4,2.2) {$\be_3$};
    
    \filldraw[black] (12,3.667,3.333) circle (4.5pt);
    \node at (12,4.1,3.333) {$\bu_1$};
    
    \filldraw[black] (13.333,1,1.333) circle (2.5pt);
    \node at (13.75,1,1.333) {\tiny{$\bu_2$}};
    
    
	
    \filldraw[magenta] (0,4,5)--(2,0,2)--(2.333,1,1.333);
    \filldraw[cyan] (3,3,0)--(0,4,5)--(2.333,1,1.333);
    
    \draw[ultra thick] (3,3,0)--(2.333,1,1.333)--(2,0,2)--(1,2,3.5)--(0,4,5)--(1,3.667,3.333)--(3,3,0);
    \draw[dotted, thick] (1,3.667,3.333)--(2.333,1,1.333);
    \draw[ultra thick, dashed] (2.333,1,1.333)--(0,4,5);

    \filldraw[black] (3,3,0) circle (4.5pt);
    \node at (3.2,3.3,-0.2) {$\be_1$};
    
    \filldraw[black] (0,4,5) circle (4.5pt);
    \node at (-0.5,4.2,5.2) {$\be_2$};
    
    \filldraw[black] (2,0,2) circle (4.5pt);
    \node at (2.2,-0.4,2.2) {$\be_3$};
    
    \filldraw[black] (1,3.667,3.333) circle (2.5pt);
    \node at (1,3.95,3.333) {\tiny{$\bu_1$}};
    
    \filldraw[black] (2.333,1,1.333) circle (4.5pt);
    \node at (2.9,1,1.333) {$\bu_2$};
    
    
    \end{tikzpicture}
\]
From this diagram we see that the multi-weighted blow-up of $\AA^3$ along $\Upgamma^{\dag,\tau_2}$: \[
    \Uppi_{\Sigma^{\dag,\tau_2}} \; \colon \; \sX_{\Sigma^{\dag,\tau_2}} \; = \; \left[\Spec\bigl(\kk[x_1',x_2',x_3,u_1]\bigr) \smallsetminus V(x_1',x_2') \q \Gm\right] \; \xrightarrow{\quad\quad} \; \AA^3
\]
is induced by the homomorphism $\Uppi_{\Sigma^{\dag,\tau_2}}^\# \colon \kk[x_1,x_2,x_3] \to \kk[x_1',x_2',x_3,u_1]$ mapping $x_1 \mapsto x_1'u_1$, $x_2 \mapsto x_2'u_1^2$ and $x_3 \mapsto x_3$. (This is the stack-theoretic weighted blow-up of $\AA^3$ along the center $\bigl(x_1,x_2^{1/2}\bigr)$.) Thus, $\Uppi_{\Sigma^{\dag,\tau_2}}^\#(f) = u_1^2 \cdot f'$, where $f' := x_1'^2+x_2'x_3$ defines the proper transform of $f$ under $\Uppi_{\Sigma^{\dag,\tau_2}}$. It remains to note the Jacobian ideal $J(f')$ of $f'$ is $\bigl(x_1',x_2',x_3\bigr)$, i.e. the unit ideal on $\sX_{\Sigma^{\dag,\tau_2}}$. The same can be shown with $\tau_2$ replaced by $\tau_1$.
\end{example}

\begin{remark}\label{R:x^2+yz}
In fact, for $f = x_1^2+x_2x_3$, $\Uptheta(f) = \{-1,-\frac{3}{2}\}$ is the \emph{smallest} set of candidate poles for $Z_{\mot,\bzero}(f;s)$. To see this, it suffices to show that $-\frac{3}{2}$ is a pole of $Z_{\tg,\bzero}(f;s)$ (cf. Remark~\ref{R:topological-zeta-function}), which we compute via the embedded resolution of $V(f) \subset \AA^3$ given by the blow-up of $\AA^3$ in $\bzero \in \AA^3$: \[
    \uppi \; \colon \; \bl_\bzero\AA^3 \; = \; \Proj_{\AA^3}\left(\sO_{\AA^3}\bigl[x_1' := x_1t, x_2':=x_2t,x_3':=x_3t,t^{-1}\bigr]\right) \; \to \; \AA^3.
\]
Here, $\uppi^{-1}\bigl(V(f)\bigr) = 2 \cdot E_1 + E_2$, where $E_1 := V(t^{-1})$ is the exceptional divisor of $\uppi$, and $E_2 := V\bigl(x_1'^2+x_2'x_3'\bigr)$ is the proper transform of $V(f)$ under $\uppi$. Moreover, the relative canonical divisor is $K_\uppi = 2 \cdot E_1$. Since $\uppi^{-1}(\bzero) = E_1 \simeq \PP^2$, we have by definition \cite[Chapter 1, \S 3.3]{chambert-loir-nicaise-sebag-motivic-integration} that: \allowdisplaybreaks\begin{align*}
    Z_{\tg,\bzero}(f;s) \; &= \; \frac{\Eu(E_1 \smallsetminus E_2)}{2s+3} \; + \; \frac{\Eu(E_1 \cap E_2)}{(s+1)(2s+3)} \; \\
    &= \; \frac{\Eu\bigl(\PP^2 \smallsetminus V(x_1'^2+x_2'x_3')\bigr)}{2s+3} \; + \; \frac{\Eu\bigl(V(x_1'^2+x_2'x_3') \subset \PP^2\bigr)}{(s+1)(2s+3)} \\
    &= \; \frac{\Eu\bigl(\PP^2 \smallsetminus V(x_1'^2+x_2'x_3')\bigr)s+\Eu(\PP^2)}{(s+1)(2s+3)} \; \stackrel{\textrm{\cite{M2}}}{\longeq} \; \frac{s+3}{(s+1)(2s+3)}.
\end{align*}
\end{remark}

\begin{example}\label{EX:yz+x^2y^2+x^2z^2}
Let $f = x_2x_3+x_1^2x_2^2+x_1^2x_3^2$. Depicted on the left side of the diagram below is $\Upgamma(f)$, where a darker shade is used for the non--$B_1$-facet: \[
    \tau_1 \; := \; \bigl\{\ba \in \Upgamma(f) \colon \ba \cdot \bu_1 = 2\bigr\} \qquad \textrm{where } \bu_1 \; := \; \be_2+\be_3
\]
with \ChMingHao{candidate pole} $-1$, and a lighter shade is used for the two $B_1$-facets: \begin{align*}
    \tau_2 \; &:= \; \bigl\{\ba \in \Upgamma(f) \colon \ba \cdot \bu_2 = 2\bigr\} \qquad \textrm{where } \bu_2 \; := \; \be_1+2\be_2 \\
    \tau_3 \; &:= \; \bigl\{\ba \in \Upgamma(f) \colon \ba \cdot \bu_3 = 2\bigr\} \qquad \textrm{where } \bu_3 \; := \; \be_1+2\be_3
\end{align*}
each with \ChMingHao{candidate pole} $-\frac{3}{2}$. Although $\tau_2$ and $\tau_3$ have \emph{different} base directions $2$ and $3$, they are \emph{non-adjacent} and hence still form a set $\dB$ of $B_1$-facets of $\Upgamma(f)$ with consistent base directions.
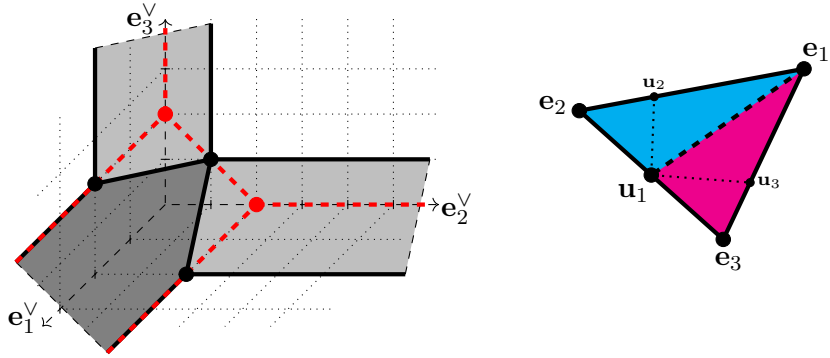
\[
	\begin{tikzpicture}[scale=0.6]
	
	\filldraw[gray] (0,2,8.5)--(0,2,4)--(1,1,0)--(2,0,4)--(2,0,8.5);
	\filldraw[lightgray] (0,5,4)--(0,2,4)--(1,1,0)--(1,4,0);
	\filldraw[lightgray] (6.8,0,4)--(2,0,4)--(1,1,0)--(5.8,1,0);
	
	\draw[->, dashed] (0,0,0)--(0,0,7); 
	\draw[<->, dashed] (0,4.1,0)--(0,0,0)--(6,0,0); 
	
	\node at (6.4,-0.1,0) {$\be_2^\vee$};
	\node at (-0.5,4.1,0) {$\be_3^\vee$};
	\node at (-0.6,0,6.5) {$\be_1^\vee$};
	
	\foreach \x in {1,...,5} 
	\draw (\x,0.1,0) -- (\x,-0.1,0);
	
	\foreach \x in {1,...,3} 
	\draw (0.1,\x,0) -- (-0.1,\x,0);
		
	\foreach \x in {2,4,6} 
	\draw (0,0.1,\x) -- (0,-0.1,\x);
	
	
	\foreach \x in {1,...,5}
	\draw[dotted] (\x,0,0) -- (\x,4.1,0);
	\foreach \x in {1,...,3}
	\draw[dotted] (0,\x,0) -- (6,\x,0);
	
	\foreach \x in {1,...,5}
	\draw[dotted] (\x,0,0) -- (\x,0,7);
	\foreach \x in {2,4,6}
	\draw[dotted] (0,0,\x) -- (6,0,\x);
	
	\foreach \x in {2,4,6}
	\draw[dotted] (0,0,\x) -- (0,4.3,\x);
	\foreach \x in {1,...,3}
	\draw[dotted] (0,\x,0) -- (0,\x,7.3);
	
	\draw[ultra thick] (0,5,4)--(0,2,4)--(1,1,0)--(2,0,4)--(6.8,0,4);
	\draw[ultra thick] (1,4,0)--(1,1,0)--(5.8,1,0);
	\draw[ultra thick] (0,2,4)--(0,2,8.5);
	\draw[ultra thick] (2,0,4)--(2,0,8.5);
	\draw[dashed] (6.8,0,4)--(5.8,1,0);
	\draw[dashed] (0,5,4)--(1,4,0);
	\draw[dashed] (0,2,8.5)--(2,0,8.5);
	\draw[dashed, ultra thick, red] (2,0,8.5)--(2,0,4)--(2,0,0)--(1,1,0);
    \draw[dashed, ultra thick, red] (0,2,8.5)--(0,2,4)--(0,2,0)--(1,1,0);
    \draw[dashed, ultra thick, red] (2,0,0)--(5.8,0,0);
    \draw[dashed, ultra thick, red] (0,2,0)--(0,3.9,0);
    
	\filldraw[black] (0,2,4) circle (4.5pt);
	\filldraw[black] (1,1,0) circle (4.5pt);
	\filldraw[black] (2,0,4) circle (4.5pt);
	\filldraw[red] (0,2,0) circle (4.5pt);
	\filldraw[red] (2,0,0) circle (4.5pt);
	
	
    \filldraw[magenta] (14,3,0)--(12,2,3.5)--(13,0,2)--(14,3,0);
    \filldraw[cyan] (14,3,0)--(12,2,3.5)--(11,4,5)--(14,3,0);
    
    \draw[ultra thick] (14,3,0)--(13.333,1,1.333)--(13,0,2)--(12,2,3.5)--(11,4,5)--(12,3.667,3.333)--(14,3,0);
    \draw[ultra thick, dashed] (14,3,0)--(12,2,3.5);
    \draw[dotted, thick] (12,3.667,3.333)--(12,2,3.5)--(13.333,1,1.333);
    
    \filldraw[black] (14,3,0) circle (4.5pt);
    \node at (14.2,3.3,-0.2) {$\be_1$};
    
    \filldraw[black] (11,4,5) circle (4.5pt);
    \node at (10.5,4.2,5.2) {$\be_2$};
    
    \filldraw[black] (13,0,2) circle (4.5pt);
    \node at (13.2,-0.4,2.2) {$\be_3$};
    
    \filldraw[black] (12,2,3.5) circle (4.5pt);
    \node at (11.6,1.6,3.5) {$\bu_1$};
    
    \filldraw[black] (12,3.667,3.333) circle (2.5pt);
    \node at (12,3.95,3.333) {\tiny{$\bu_2$}};
    
    \filldraw[black] (13.333,1,1.333) circle (2.5pt);
    \node at (13.75,1,1.333) {\tiny{$\bu_3$}};
    

    \end{tikzpicture}
\]
Consequently, Theorem~\ref{T:fake-poles} says that $\Uptheta^{\dag,\dB}(f) = \{-1\} \subsetneq \{-1,-\frac{3}{2}\} = \Uptheta(f)$ is also a set of candidate poles for $Z_{\mot,\bzero}(f;s)$. To see this, we proceed as with previous examples: we ``drop'' $H_{\bu_2,2}^+$ and $H_{\bu_3,2}^+$ from $\Upgamma(f) = \bigcap\{H_{\bu_i,2}^+ \colon 1 \leq i \leq 3\}$ to define the Newton polyhedron $\Upgamma^\dag = H_{\bu_1,2}^+$, which we have outlined in red on the left side of the above diagram. 

On the right side we sketched the normal fan $\Sigma^\dag$ of $\Upgamma^\dag$, keeping the same conventions as before in Example~\ref{EX:x^2+xy^4+y^3z+z^3}. The multi-weighted blow-up of $\AA^3$ along $\Upgamma^\dag$ is then simply the blow-up of $\AA^3$ along $V(x_2,x_3) \subset \AA^3$: \[
    \Uppi_{\Sigma^\dag} \; \colon \; \sX_{\Sigma^\dag} \; = \; \Proj_{\AA^3}\left(\sO_{\AA^3}\bigl[x_2':=x_2t,x_3':=x_3t,u_1:=t^{-1}\bigr]\right) \; \xrightarrow{\quad\quad} \; \AA^3.
\]
We have $\Uppi_{\Sigma^\dag}^\#(f) = u_1^2 \cdot f'$, where $f' := x_2'x_3' + x_1^2x_2'^2 + x_1^2x_3'^2$ defines the proper transform of $f$ under $\Uppi_{\Sigma^\dag}$. If $J(f')$ denotes the Jacobian of $f'$, we then have: \begin{align*}
    \sqrt{(f') + J(f')} \; &= \; \sqrt{\bigl(x_1x_2'^2+x_1x_3'^2,x_3'+2x_1^2x_2',x_2'+2x_1^2x_3',x_2'x_3'\bigr)} \\
    &= \; \sqrt{\bigl(x_3'+2x_1^2x_2',x_2'+2x_1^2x_3',x_2'x_3',x_1x_2',x_1x_3'\bigr)} \; = \; \bigl(x_2',x_3'\bigr) 
\end{align*}
which is the unit ideal on $\sX_{\Sigma^\dag}$, i.e. $\Uppi_{\Sigma^\dag}$ is a stack-theoretic embedded desingularization for $V(f) \subset \AA^3$ as desired.
\end{example}

\begin{remark}\label{R:yz+x^2y^2+x^2z^2}
While $\Uptheta(f) \smallsetminus \{-\frac{3}{2}\} = \{-1\}$ is a set of candidate poles for $Z_{\mot,\bzero}(f;s)$, note that $\frac{3}{2}$ \emph{still} induces a monodromy eigenvalue of $f$ near $\bzero$.
\end{remark}

\section{Proof of main theorem}\label{4}

\subsection{Dropping a set of facets from a Newton \texorpdfstring{$\QQ$}{Q}-polyhedron}\label{4.1}

In this section, we fix a Newton $\QQ$-polyhedron $\Upgamma$, with associated piecewise-linear, convex $\QQ$-function $\varphi$ (\S\ref{2.1}), and associated normal fan $\Sigma$ in $N_\RR$ (\S\ref{2.2}). \ChMingHao{We first fix the following conventions for the remainder of this paper}:

\begin{xpar}\label{X:convention-4.1-A} 
If two rays $\rho_1, \rho_2 \in \Sigma[1]$ satisfy $\rho_1 + \rho_2 \in \Sigma[2]$, we say that $\rho_1$ and $\rho_2$ are {\sffamily adjacent in $\Sigma$}, and write \[
    \rho_1 \; \frown \; \rho_2 \;\; \textrm{ in $\Sigma$.}
\]
Given $\tau, \tau' \prec^\ChMingHao{1} \Upgamma$, note that $\tau \frown \tau'$ (in the sense of \ref{X:convention-2.1.A}) if and only if $\rho_\tau \frown \rho_{\tau'}$ in $\Sigma$ (cf. \ref{X:cone-face-correspondence}).
\end{xpar}

\begin{xpar}
Throughout this section, consider a subset $\dB$ of facets of $\Upgamma$ that are \emph{not} contained in any translate $m\be_i+H_i$ of any coordinate hyperplane $H_i$ in $M_\RR^+$. For any such $\dB$, we set \[
    \ChMingHao{\Sigma[1]|_\dB} \; := \; \bigl\{\rho_\tau \in \Sigma[1] \colon \tau \in \dB\bigr\} \; \subset \; \Sigma[1].
\] 
As motivated in \S\ref{3.2}, we study in this section the Newton $\QQ$-polyhedron obtained from $\Upgamma$ by ``dropping the facets in $\dB$'':
\end{xpar}

\begin{definition}\label{D:B-cut}
Recalling from \eqref{EQ:associated-newton-Q-polyhedron} that \[
    \Upgamma \; = \; \bigcap_{\tau \prec^1 \Upgamma}{H_{\bu_\tau,N_\tau}^+}
\]
we define the {\sffamily $\dB$-cut} of $\Upgamma$ to be the following Newton $\QQ$-polyhedron: \[
    \Upgamma^{\dag,\dB} \; := \; \bigcap\bigl\{H_{\bu_\tau,N_\tau}^+ \colon \tau \prec^1 \Upgamma,\; \tau \notin \dB\bigr\} \; \supset \; \Upgamma.
\]
\ChMingHao{We call its normal fan in $N_\RR$ the $\dB$-cut of $\Sigma$,} and we denote it by $\Sigma^{\dag,\dB}$. \ChMingHao{When $\dB$ is unambiguous from context}, we write $\Upgamma^\dag$ for $\Upgamma^{\dag,\dB}$ and $\Sigma^\dag$ for $\Sigma^{\dag,\dB}$. 
\end{definition} 

\begin{lemma}\label{L:facets-dagger}
Let $\tau \prec^1 \Upgamma$ such that $\tau \notin \dB$. Then: \begin{enumerate}
    \item There exists a (unique) facet $\tau^\dag \prec^1 \Upgamma^\dag$ such that $\tau^\dag \cap \Upgamma = \tau$.
    \item If moreover $\tau$ is not adjacent to any facet in $\dB$, then $\tau^\dag = \tau$. In other words, $\tau$ remains a facet of $\Upgamma^\dag$.
\end{enumerate}
\end{lemma}

\begin{proof}
For (i), note that \[
    \tau \; = \; H_{\bu_\tau,N_\tau} \cap\; \bigcap\bigl\{H_{\bu_{\tau'},N_{\tau'}}^+ \colon \tau' \prec^1 \Upgamma\bigr\}.
\]
Set \[
    \tau^\dag \; := \; H_{\bu_\tau,N_\tau} \cap \Upgamma^\dag \; = \; H_{\bu_\tau,N_\tau} \cap \;\bigcap\bigl\{H_{\bu_\tau,N_\tau}^+ \colon \tau'\prec \Upgamma,\; \tau' \notin \dB\bigr\}
\]
from which it follows that $\tau^\dag \cap \Upgamma = \tau$. Since $\tau \subset \tau^\dag \subset H_{\bu_\tau,N_\tau}$, $\dim(\tau^\dag) = n-1$, i.e. $H_{\bu_\tau,N_\tau}$ is a supporting hyperplane for $\Upgamma^\dag$, and $\tau^\dag$ is a facet of $\Upgamma$. For (ii), note that since every face of $\tau$ is the intersection of a subset of facets of $\tau$, we have: \[
    \tau \; = \; H_{\bu_\tau,N_\tau} \cap \; \bigcap\bigl\{H_{\bu_{\tau'},N_{\tau'}}^+ \colon \tau' \prec^1 \Upgamma,\; \tau' \frown \tau\bigr\}.
\]
By hypothesis, $\{\tau' \prec \Upgamma \colon \tau' \frown \tau\} \subset \{\tau' \prec \Upgamma, \colon \tau' \notin \dB\}$. Therefore, $\tau \supset \tau^\dag$, which proves (ii).
\end{proof}

\begin{xpar}[A correspondence]\label{X:facet-correspondence}
\ChMingHao{The preceding lemma sets up an injection \begin{align}\label{EQ:facet-correspondence}
    \begin{split}
    \bigl\{\textrm{facets of $\Upgamma$}\bigr\} \smallsetminus \dB \; &\xhookrightarrow{\qquad} \; \bigl\{\textrm{facets of $\Upgamma^\dag$}\bigr\} \\
    \tau \; &\xmapsto{\qquad} \; \tau^\dag
    \end{split}
\end{align}
which is in fact a bijection, since we have assumed that each facet in $\dB$ is not contained in $m\be_i + H_i$ for any $i \in [n]$ and $m \in \QQ_{>0}$}. We will freely adopt this correspondence for the remainder of this paper. Note that in particular, $\Sigma^\dag[1] = \Sigma[1] \smallsetminus \ChMingHao{\Sigma[1]|_\dB}$. For $\rho \in \Sigma^\dag[1]$, we may therefore consider $\rho$ as a ray in $\Sigma[1]$ --- in that case, we continue to denote by $\tau_\rho$ the facet of $\Upgamma$ dual to $\rho$ in $\Sigma[1]$. On the other hand, we denote by $\tau_\rho^\dag$ the facet of $\Upgamma^\dag$ dual to $\rho$ in $\Sigma^\dag[1]$. This does not contradict the notation in \eqref{EQ:facet-correspondence}.
\end{xpar}

\begin{xpar}\label{X:piecewise-linear-convex-function-dagger}
Let $\varphi^\dagger \colon N_\RR^+ \to \RR_{\geq 0}$ be the piecewise-linear, convex $\QQ$-function corresponding to the Newton $\QQ$-polyhedron $\Upgamma^\dag$, cf. \ref{X:piecewise-linear-convex-function}. By \ref{X:alt-characterization} and \ref{X:facet-correspondence}, $\varphi^\dag$ can be explicated as \[
    \varphi^\dag \; = \; \min\sS^\dag
\]
where \[
    \sS^\dag \; := \; \left\{\parbox{8cm}{linear functions $\ell \colon N_\RR^+ \to \RR_{\geq 0}$ such that $\ell(\bu_\tau) \geq N_\tau$ for every facet $\tau \prec^1 \Upgamma$ not in $\dB$}\right\}.
\]
We also note that for every facet $\tau \prec^1 \Upgamma$ not in $\dB$, \begin{equation}\label{EQ:piecewise-linear-convex-function-dagger}
    \varphi(\bu_\tau) \; = \; N_\tau \; = \; \varphi^\dag(\bu_\tau).
\end{equation}
For the remainder of this section, we switch our focus to the cones in $\Sigma^\dag$. For later purposes (e.g. in \S\ref{4.3}), we occasionally state some of our definitions and results for cones in the augmentation $\overline{\Sigma}^\dag$ of $\Sigma^\dag$, cf. \ref{X:cox-construction}.
\end{xpar}

\begin{definition}\label{D:new-old-cones}
We say a cone $\sigma$ in $\overline{\Sigma}^\dag$ is {\sffamily old} if $\sigma$ can be inscribed in some cone $\sigma'$ in $\Sigma$ (in which case one writes $\sigma \sqsubset \sigma'$). If not, we say $\sigma$ is {\sffamily new}.
\end{definition}

\begin{lemma}\label{L:old}\
\begin{enumerate}
    \item For any cone $\sigma$ in $\Sigma$, the cone \ChMingHao{$\sigma^\dag$ in $N_\RR$ generated by rays in} \[
        \ChMingHao{\sigma[1] \smallsetminus \Sigma[1]|_\dB}
    \]
    is a cone in $\Sigma^\dag$ (hence, all its faces are old cones in $\Sigma^\dag$). Moreover, for every $\bu \in \sigma^\dag$, $\varphi^\dag(\bu) = \varphi(\bu)$.
    \item For every facet $\tau \prec^1 \Upgamma$ with $\tau \in \dB$, we have: \[
        \varphi^\dag(\bu_\tau) \; < \; \varphi(\bu_\tau).
    \]
\end{enumerate}
\end{lemma}

\begin{proof}
For (i), let $\ba \in \relint(\varsigma_\sigma)$, so that $\sigma = \sigma_\ba = \{\bu \in N_\RR^+ \colon \varphi(\bu) = \ba \cdot \bu\}$, cf. \ref{X:cone-face-correspondence}. Let $\sigma^\dag_\ba := \{\bu \in N_\RR^+ \colon \varphi^\dag(\bu) = \ba \cdot \bu\}$, which by definition is a cone in $\Sigma^\dag$. We claim that $\sigma^\dag_\ba[1] = \sigma_\ba[1] \smallsetminus \ChMingHao{\Sigma[1]|_\dB}$, which would prove (i). Indeed, given any $\tau \prec^1 \Upgamma$ not in $\dB$, we have $\varphi^\dag(\bu_\tau) = \varphi(\bu_\tau)$, cf. \eqref{EQ:piecewise-linear-convex-function-dagger}, and hence we have $\varphi^\dag(\bu_\tau) = \ba \cdot \bu_\tau$ if and only if $\tau$ is dual to a ray in $\sigma_\ba[1] \smallsetminus \ChMingHao{\Sigma[1]|_\dB}$. By Corollary~\ref{C:structure-of-cones} and \ref{X:facet-correspondence}, this proves our claim.

For (ii), we apply the above argument to the case where $\sigma$ is the ray $\rho_\tau$ in $\Sigma$ dual to $\tau \prec \Upgamma$, and we obtain that for $\ba \in \relint(\tau)$, we have $\{\bu \in N_\RR^+ \colon \varphi^\dag(\bu) = \ba \cdot \bu\} = \{\bzero\}$. Combining that with the fact that $\ba \in \Upgamma \subset \Upgamma^\dag$, we must have $\varphi^\dag(\bu_\tau) < \ba \cdot \bu_\tau = \varphi(\bu_\tau)$.
\end{proof}

\begin{lemma}\label{L:old-cone}
Let $\sigma$ be a cone in $\Sigma^\dag$. \begin{enumerate}
    \item If there is an extremal ray $\rho$ of $\sigma$ that is not adjacent in $\Sigma$ to any ray in $\ChMingHao{\Sigma[1]|_\dB}$, then $\sigma$ is old.
    \item If \underline{moreover} $\dim(\sigma) = 2$, then $\sigma$ is a cone in $\Sigma$.
\end{enumerate}
\end{lemma}

\begin{proof}
By Lemma~\ref{L:facets-dagger}(ii), the facet $\tau_\rho \prec^1 \Upgamma$ dual to $\rho \in \Sigma[1]$ remains a facet of $\Upgamma^\dag$. Therefore, the face $\varsigma \prec \Upgamma^\dag$ dual to $\sigma$, being a face of $\tau_\rho$, remains a face of $\Upgamma^\dag$. Consequently, for every $\tau \prec^1 \Upgamma$ such that $\tau \notin \dB$, we have the following equivalences: \begin{equation}\label{EQ:old-cone}
    \varsigma \; \prec \; \tau \; \iff \; \varsigma \; \subset \; H_{\bu_\tau,N_\tau} \; \iff \; \varsigma \; \prec \; \tau^\dag.
\end{equation}
The reverse implication in \eqref{EQ:old-cone} means that $\sigma$ is inscribed in the cone in $\Sigma$ dual to the face $\varsigma \prec \Upgamma$, as desired.

If $\dim(\sigma) = 2$, let $\rho$ and $\rho'$ be the extremal rays of $\sigma$. By Corollary~\ref{C:cone-face-correspondence-A}, $\varsigma = \tau_\rho^\dag \cap \tau_{\rho'}^\dag$. By \eqref{EQ:old-cone}, $\varsigma$ is a face of both $\tau_\rho$ and $\tau_{\rho'}$. Since $\varsigma$ is a $(n-2)$-dimensional face of $\Upgamma$, $\varsigma$ is a face of \emph{exactly two} facets of $\Upgamma$, which by the preceding sentence are necessarily $\tau_\rho$ and $\tau_{\rho'}$. This means that the cone in $\Sigma$ dual to the face $\varsigma \prec \Upgamma$ is generated by $\rho$ and $\rho'$, i.e. is equal to $\sigma$. In particular, $\sigma$ is a cone in $\Sigma$.
\end{proof}

By part (i) of the preceding lemma, we see that if $\sigma$ is a new cone in $\Sigma^\dag$, then all its extremal rays must be adjacent in $\Sigma$ to some ray in $\ChMingHao{\Sigma[1]|_\dB}$. The next proposition refines that observation. We first introduce some notation:

\begin{xpar}[An equivalence relation]
Let $\sim$ denote the equivalence closure of $\frown$ (cf. \ref{X:convention-2.1.A} and \ref{X:convention-4.1-A}) on \emph{either $\dB$ or $\ChMingHao{\Sigma[1]|_\dB}$}. We also let $\dB_{/\sim}$ denote the set of equivalence classes of $\dB$ under $\sim$.
\end{xpar}


\begin{proposition}\label{P:new-cone-detailed}
Let \ChMingHao{$k := \#\dB_{/\sim}$}, and let $\dB_{/\sim} = \{\intercal_1,\intercal_2,\dotsc,\intercal_k\}$ be any total order on $\dB_{/\sim}$. For each $\ell \in [k]$, let $\intercal_{\leq \ell} := \bigcup\{\intercal_j \colon j \leq \ell \}$. Then for any new cone $\sigma$ in $\Sigma^{\dag,\ChMingHao{\dB}}$, there exists a unique $\ell \in [k]$ such that: \begin{enumerate}
    \item $\sigma$ cannot be inscribed in any cone in $\Sigma^{\dag,\intercal_{\leq \ell-1}}$ and
    \item $\sigma$ is a cone in $\Sigma^{\dag,\intercal_{\leq \ell}}$. 
\end{enumerate}
Moreover, every extremal ray of $\sigma$ is adjacent \underline{in $\Sigma$} to some ray in \ChMingHao{$\Sigma[1]|_{\intercal_\ell}$}.
\end{proposition}

\begin{remark}
\ChMingHao{We remind the reader that for any $\ell \in [k]$, $\Sigma^{\dag,\intercal_{\leq \ell}}$ is the $\intercal_{\leq \ell}$-cut of $\Sigma$, as in Definition~\ref{D:B-cut}. Note that if $\ell > 1$, $\Sigma^{\dag,\intercal_{\leq \ell}}$ is also the $\intercal_\ell$-cut of $\Sigma^{\dag,\intercal_{\leq \ell-1}}$. This observation will be used for the purposes of induction in the proof below. Finally, note that $\Sigma^{\dag,\intercal_{\leq k}}$ is simply $\Sigma^{\dag,\dB}$.}
\end{remark}



\begin{proof}
Proceed by induction on $k = \#\dB_{/\sim}$. The base case $k = 1$ is supplied by Lemma~\ref{L:old-cone}(i). If $k > 1$, we consider two cases: \begin{enumerate}
    \item[(a)] If $\sigma$ can be inscribed in some cone in $\Sigma^{\dag,\intercal_{\leq k-1}}$, then let $\sigma'$ be the \emph{smallest} cone in $\Sigma^{\dag,\intercal_{\leq k-1}}$ such that $\sigma \sqsubset \sigma'$. Then we claim $\sigma = \sigma'$. Indeed, by Lemma~\ref{L:old-cone}(i), all the extremal rays of $\sigma'$ are adjacent in $\Sigma$ to some ray in \ChMingHao{$\Sigma[1]|_{\intercal_{\leq k-1}}$}, and hence, $\sigma'[1] \cap \ChMingHao{\Sigma[1]|_{\intercal_k}} = \varnothing$. By Lemma~\ref{L:old}(i), $\sigma'$ is therefore \ChMingHao{also} a cone in $\Sigma^{\dag,\intercal_{\leq k}} = \Sigma^{\dag,\ChMingHao{\dB}}$. Given that $\sigma \subset \sigma'$ are both cones in $\Sigma^{\dag,\ChMingHao{\dB}}$, and $\sigma$ \emph{does not lie in any proper face of $\sigma'$} but can be inscribed in $\sigma'$, we must have $\sigma = \sigma'$, as desired. Therefore, $\sigma$ was already a new cone in $\Sigma^{\dag,\intercal_{\leq k-1}}$, and the proposition follows by induction hypothesis.
    
    \item[(b)] Otherwise, \ChMingHao{we only need to check the last sentence of the proposition.} By Lemma~\ref{L:old-cone}(i), every extremal ray of $\sigma$ is adjacent \underline{in $\Sigma^{\dag,\intercal_{\leq k-1}}$} to some ray in \ChMingHao{$\Sigma[1]|_{\intercal_k}$}. Since every ray in \ChMingHao{$\Sigma[1]|_{\intercal_k}$} is by definition not adjacent \underline{in $\Sigma$} to any ray in \ChMingHao{$\Sigma[1]|_{\intercal_{\leq k-1}}$}, Lemma~\ref{L:old-cone}(ii) says that every extremal ray of $\sigma$ is in fact adjacent \underline{in $\Sigma$} to some ray in \ChMingHao{$\Sigma[1]|_{\intercal_k}$}.
\end{enumerate}
This completes the induction.
\end{proof}

\begin{remark}\label{R:new-cone-detailed}
Given that the total order on $\dB_{/\sim}$ plays an auxiliary role in the above proof, the following stronger assertion should be true. Namely, for any new cone $\sigma$ in $\Sigma^\dag$, there exists a unique $\intercal \in \dB_{/\sim}$ such that $\sigma$ was already a new cone in $\Sigma^{\dag,\intercal}$ (so every extremal ray of $\sigma$ is adjacent to some ray in $\Sigma[\intercal]$). However, this stronger assertion is not needed for this paper.
\end{remark}

We conclude this section with one more crucial observation:

\begin{lemma}\label{L:new-cone}
For a cone $\sigma$ in $\overline{\Sigma}^\dag$, the following statements are equivalent: \begin{enumerate}
    \item $\sigma$ is new.
    \item $\bigcap\{\tau_\rho \colon \rho \in \sigma[1]\} = \varnothing$.
\end{enumerate}
Moreover, if $\sigma$ is old and not contained in any coordinate hyperplane $\{\be_i^\vee = 0\}$ in $N_\RR$, then $\bigcap\{\tau_\rho \colon \rho \in \sigma[1]\}$ is a compact face of $\Upgamma$.
\end{lemma}

\begin{proof}
For (i)$\impliedby$(ii), suppose $\sigma$ is inscribed in a cone $\sigma'$ in $\Sigma$. By Corollary~\ref{C:cone-face-correspondence-A}, the face $\varsigma' \prec \Upgamma$ dual to $\sigma'$ is $\varsigma' = \bigcap\{\tau_\rho \colon \rho \in \sigma'[1]\}$. Since $\sigma[1] \subset \sigma'[1]$, we have $\varsigma' = \bigcap\bigl\{\tau_\rho \colon \rho \in \sigma'[1]\bigr\} \subset \bigcap\bigl\{\tau_\rho \colon \rho \in \sigma[1]\bigr\}$, so that in particular, the latter must be non-empty.

For (i)$\implies$(ii), set $\underline{\varsigma} := \bigcap\{\tau_\rho \colon \rho \in \sigma[1]\}$. If $\underline{\varsigma} \neq \varnothing$, then $\underline{\varsigma}$ is a (non-empty) face of $\Upgamma$. In that case we claim that $\sigma$ is inscribed in the cone $\overline{\sigma}$ in $\Sigma$ dual to $\underline{\varsigma} \prec \Upgamma$, a contradiction. Indeed, letting $\varsigma$ denote the face of $\Upgamma^\dag$ dual to $\sigma \in \Sigma^\dag$, the claim amounts to the following implication for every $\tau \prec^1 \Upgamma$: \[
    \varsigma \prec \tau^\dag \; \implies \; \underline{\varsigma} \prec \tau.
\]
That implication follows from $\{\tau^\dag \prec^1 \Upgamma^\dag \colon \varsigma \prec \tau^\dag\} = \{\tau_\rho^\dag \colon \rho \in \sigma[1]\}$ (Corollary~\ref{C:cone-face-correspondence-A}) and the definition of $\underline{\varsigma}$. Finally, for the last statement, $\overline{\sigma}$ is also not contained in any coordinate hyperplane in $N_\RR$. By Corollary~\ref{C:cone-face-correspondence-B}, $\underline{\varsigma}$ is therefore compact.
\end{proof}

\subsection{Dropping a set of \texorpdfstring{$B_1$}{B1}-facets with consistent base directions}\label{4.2}

In this section, let $f \in \kk[x_1,\dotsc,x_n]$ be a non-degenerate polynomial. We specialize the earlier discussion in \S\ref{4.1} to the case when $\Upgamma$ is the Newton polyhedron $\Upgamma(f)$ of $f$, and $\dB$ is a set of $B_1$-facets of $\Upgamma(f)$ with consistent base directions, cf. Definition~\ref{D:consistent-B1-facets}. As in \ref{3.1}, let $\Sigma(f)$ denote the normal fan of $\Upgamma(f)$. Before that, we state (without proof) some easy observations:

\begin{xpar}\label{X:B_1-facets}
Suppose $\Upgamma(f)$ has a $B_1$-facet $\tau$ with apex $\bv$ and corresponding base direction $i \in [n]$. Let $J(\tau) := \{j \in [n] \colon \tau$ is non-compact in the $j$\textsuperscript{th} coordinate$\}$ (cf. Corollary~\ref{C:cone-face-correspondence-B}), so that by definition $i \notin J(\tau)$. Then: \begin{enumerate}
    \item Let $\tau^c$ denote the convex hull of $\vertex(\tau) = \vertex(H_i \cap \tau) \cup \{\bv\}$ in $M_\RR^+$. Then $\tau = \tau^c + \ChMingHao{\sum_{j \in J(\tau)}}{\RR_{\geq 0}\be_j}$.
    \item $H_i \cap \tau \prec^1 \tau$.
    \item $\tau$ is not contained in any translate $m\be_k+H_k$ of any coordinate hyperplane $H_k$ in $M_\RR$.
    \item The facet $\tau_i$ of $\Upgamma(f)$ dual to the ray $\ChMingHao{\langle\be_i\rangle}$ in $\Sigma(f)$ is $H_i \cap \Upgamma(f)$. In other words, $N_{\tau_i} = 0$ (recall \ref{X:alt-characterization} for definition of $N_{\tau_i}$).
\end{enumerate}
\end{xpar}

\begin{xpar}\label{X:consistent-B_1-facets}
For a set $\dB$ of $B_1$-facets of $\Upgamma(f)$, the following are equivalent: \begin{enumerate}
    \item $\dB$ is a set of $B_1$-facets of $\Upgamma(f)$ with consistent base directions.
    \item For every $\intercal \in \dB_{/\sim}$, there exists $\bv \in \bigcap\{\vertex(\tau) \colon \tau \in \intercal\}$ and $i \in [n]$ such that every $\tau$ in $\intercal$ is a $B_1$-facet with apex $\bv$ and corresponding base direction $i$.
\end{enumerate}
In (ii), we call $\bv$ an {\sffamily apex} of $\intercal$ with corresponding {\sffamily base direction} $i \in [n]$.
\end{xpar}

\begin{xpar}[Conventions for this section]\label{X:convention-4.2-A}
Let $\Upgamma^\dag$ denote the $\dB$-cut of $\Upgamma(f)$, and let $\Sigma^\dag$ denote its normal fan in $N_\RR$. We also fix, for each $\intercal \in \dB_{/\sim}$, an apex $\bv_\intercal$ of $\intercal$ and denote the corresponding base direction by $b(\intercal)$. For the remainder of this section, we fix a new cone $\sigma$ in $\Sigma^\dag$, and let $\varsigma$ denote the face of $\Upgamma^\dag$ dual to $\sigma$. With respect to an auxiliary total order $\dB_{/\sim} = \{\intercal_1,\intercal_2,\dotsc,\intercal_k\}$ on $\dB_{/\sim}$, let $\ell$ be the unique natural number in $[k]$ for which $\sigma$ satisfies the properties stated in Proposition~\ref{P:new-cone-detailed}. We then set $\intercal := \intercal_\ell$.
\end{xpar}

\begin{proposition}\label{P:new-cone-II}
For each $\rho \in \sigma[1]$, $\tau_\rho$ is adjacent to some facet in $\intercal$. Moreover: \begin{enumerate}
    \item $\ChMingHao{\langle} \be_{b(\intercal)} \ChMingHao{\rangle}$ is an extremal ray of $\sigma$.
    
    \item \ChMingHao{The cone $\sigma^\circ$ in $N_\RR$ generated by the rays in} \[
        \qquad \qquad \ChMingHao{\sigma[1] \smallsetminus \{\langle \be_{b(\intercal)}\rangle\}}
    \]
    is a face of $\sigma$ (and hence is a cone in $\Sigma^\dag$) that can be inscribed in the maximal cone in $\Sigma(f)$ dual to the vertex $\bv_\intercal$ of $\Upgamma(f)$.
    
    \item The face \[
        \qquad \qquad \underline{\varsigma}^\circ \; := \; \bigcap\bigl\{\tau_\rho \colon \rho \in \sigma^\circ[1]\bigr\} \; \prec \; \Upgamma(f)
    \]
    has empty intersection with $H_{b(\intercal)}$. Moreover, for every $\tau \in \intercal$, $\underline{\varsigma}^\circ \cap \tau$ is either $\{\bv_\intercal\}$ or a non-compact face of $\tau$ containing $\bv_\intercal$.
\end{enumerate}
\end{proposition}

\begin{proof}
The first statement is a restatement of the last property in Proposition~\ref{P:new-cone-detailed}. For $\rho \in \sigma[1]$, let $\tau$ be a facet in $\intercal$ adjacent to $\tau_\rho$, so $\tau \cap \tau_\rho$ is a facet of $\tau$. If $\rho \neq \ChMingHao{\langle} \be_{b(\intercal)} \ChMingHao{\rangle}$, $\tau \cap \tau_\rho$ cannot be equal to $H_{b(\intercal)} \cap \tau$, and hence must contain $\bv_\intercal$ (cf. \ref{X:B_1-facets}(i)). In particular, $\bv_\intercal \in \tau_\rho$. Therefore, \begin{equation}\label{EQ:new-cone-II}
    \bv_\intercal \; \in \; \bigcap\left\{\tau_\rho \colon \ChMingHao{\rho \in \sigma[1] \smallsetminus \{\langle \be_{b(\intercal)} \rangle\}} \right\}.
\end{equation}
If $\ChMingHao{\langle} \be_{b(\intercal)} \ChMingHao{\rangle} \notin \sigma[1]$, then \eqref{EQ:new-cone-II} becomes $\bv_\intercal \in \bigcap\bigl\{\tau_\rho \colon \rho \in \sigma[1]\bigr\}$, which contradicts Lemma~\ref{L:new-cone}. This proves (i).

For (ii), \eqref{EQ:new-cone-II} already shows that $\sigma^\circ$ can be inscribed in the cone in $\Sigma(f)$ dual to the vertex $\bv_\intercal$ of $\Upgamma(f)$. It remains to show $\sigma^\circ$ is a cone in $\Sigma^\dag$. More precisely, we show $\sigma^\circ$ is dual to the face \begin{equation}\label{EQ:varsigma-circ}
    \varsigma^\circ \; := \; \bigcap\bigl\{\tau_\rho^\dag \colon \rho \in \sigma^\circ[1]\bigr\} \; \prec \; \Upgamma^\dag.
\end{equation}
By Corollary~\ref{C:cone-face-correspondence-A}, this amounts to showing that $\bigl\{\tau_\rho^\dag \colon \rho \in \sigma^\circ[1]\bigr\}$ are the only facets $\tau^\dag \prec^1 \Upgamma^\dag$ containing $\varsigma^\circ$. Indeed, any facet $\tau^\dag \prec^1 \Upgamma^\dag$ containing $\varsigma^\circ$ must also contain the face $\varsigma \prec \Upgamma^\dag$ dual to $\sigma \in \Sigma^\dag$, and hence, must be dual to \ChMingHao{an} extremal ray $\rho$ in $\sigma[1]$. It remains to observe that that $\rho$ cannot be $\ChMingHao{\langle} \be_{b(\intercal)} \ChMingHao{\rangle}$, since $\bv_\intercal \in \varsigma^\circ$ \eqref{EQ:new-cone-II} but $\bv_\intercal \notin H_{b(\intercal)} \cap \Upgamma(f)$. 

Finally, we prove (iii). By Lemma~\ref{L:new-cone}, we obtain: \begin{align*}
    \varnothing \; = \; \bigcap\bigl\{\tau_\rho \colon \rho \in \sigma[1] \bigr\} \; &= \; \bigl(H_{b(\intercal)} \cap \Upgamma(f)\bigr) \cap \; \bigcap\bigl\{\tau_\rho \colon \rho \in \sigma^\circ[1]\bigr\} \\
    &= \; H_{b(\intercal)} \cap \underline{\varsigma}^\circ.
\end{align*}
In particular, for every $\tau \in \intercal$, $\underline{\varsigma}^\circ \cap \tau$ is a face of $\tau$ that does not intersect the facet $H_{b(\intercal)} \cap \tau \prec^1 \tau$. By \eqref{EQ:new-cone-II}, $\underline{\varsigma}^\circ \cap \tau$ also contains $\bv_\intercal$. Since the only compact face of $\tau$ satisfying those two conditions is $\{\bv_\intercal\}$ (cf. \ref{X:B_1-facets}(i)), this proves (iii).
\end{proof}

As an immediate consequence of the preceding proposition, we have:

\begin{corollary}\label{C:new-cone-II}
Every $\ba \in \underline{\varsigma}^\circ$ has $b(\intercal)$\textsuperscript{th} coordinate $\geq 1$.
\end{corollary}

\begin{proof}
Since $\underline{\varsigma}^\circ \cap H_{b(\intercal)} = \varnothing$, all vertices of $\underline{\varsigma}^\circ$ have $b(\intercal)$\textsuperscript{th} coordinate $>0$. On the other hand, since $\Upgamma(f)$ is a Newton polyhedron, all vertices of $\underline{\varsigma}^\circ$ have integer coordinates, and hence, must have $b(\intercal)$\textsuperscript{th} coordinate $\geq 1$.
\end{proof}

For later purposes, the preceding corollary is however not sufficient. We instead need the following refinement:

\begin{proposition}\label{P:new-cone-dual-to-compact}
If the face $\varsigma \prec \Upgamma^\dag$ dual to $\sigma$ is compact, then every $\ba \in \underline{\varsigma}^\circ \smallsetminus \{\bv_\intercal\}$ has $b(\intercal)$\textsuperscript{th} coordinate $> 1$.
\end{proposition}

\begin{xpar}\label{X:new-cone-dual-to-compact}
We prove the preceding proposition after a few observations and results. For the remainder of this section, let $\varsigma^\circ$ denote the face of $\Upgamma^\dag$ dual to $\sigma^\circ$, cf. \eqref{EQ:varsigma-circ}. By Corollary~\ref{C:cone-face-correspondence-A}, we have: \begin{align}\label{EQ:new-cone-dual-to-compact-1}
    \begin{split}
    \varsigma \; = \; \bigcap\bigl\{\tau_\rho^\dagger \colon \rho \in \sigma[1]\bigr\} \; &= \; \bigl(H_{b(\intercal)} \cap \Upgamma^\dag\bigr) \cap \; \bigcap\bigl\{\tau_\rho^\dagger \colon \rho \in \sigma^\circ[1]\bigr\} \\
    &= \; H_{b(\intercal)} \cap \varsigma^\circ
    \end{split}
\end{align}
and \begin{align}\label{EQ:new-cone-dual-to-compact-2}
    \begin{split}
    \underline{\varsigma}^\circ \; = \; \bigcap\bigl\{\tau_\rho \colon \rho \in \sigma^\circ[1]\bigr\} \; &= \; \bigcap\bigl\{\tau_\rho^\dag \cap \Upgamma(f) \colon \rho \in \sigma^\circ[1]\bigr\} \\
    &= \; \varsigma^\circ \cap \Upgamma(f).
    \end{split}
\end{align}
From these equalities we deduce the next lemma. In particular, note that part (ii) of the next lemma refines Proposition~\ref{P:new-cone-II}(iii).
\end{xpar}

\begin{lemma}\label{L:new-cone-dual-to-compact}
If $\varsigma$ is compact, then: \begin{enumerate}
    \item Both $\varsigma^\circ$ and $\underline{\varsigma}^\circ$ are either non-compact in the $b(\intercal)$\textsuperscript{th} coordinate, or compact.
    \item For any $\tau \in \intercal$, we have $\varsigma^\circ \cap \tau = \underline{\varsigma}^\circ \cap \tau = \{\bv_\intercal\}$.
\end{enumerate}
\end{lemma}

\begin{proof}
(i) follows from \eqref{EQ:new-cone-dual-to-compact-1} and \eqref{EQ:new-cone-dual-to-compact-2}, since $H_{b(\intercal)}$ is non-compact in the $i$\textsuperscript{th} coordinate for $i \in [n] \smallsetminus \{b(\intercal)\}$. For (ii), we note, from (i) and the fact that any $\tau \in \intercal$ cannot be non-compact in the $b(\intercal)$\textsuperscript{th} coordinate (Definition~\ref{D:B1-facets}(ii)), that $\underline{\varsigma}^\circ \cap \tau$ is a compact face of $\tau$, and hence is $\{\bv_\intercal\}$ by Proposition~\ref{P:new-cone-II}(iii). Note finally that $\varsigma^\circ \cap \tau = \underline{\varsigma}^\circ \cap \tau$ by \eqref{EQ:new-cone-dual-to-compact-2}.
\end{proof}

\begin{proposition}\label{P:new-cone-dual-to-compact-II}
If $\varsigma$ is compact, $\underline{\varsigma}^\circ$ is $\{\bv_\intercal\}$ or $1$-dimensional. In the latter case, the affine span of $\underline{\varsigma}^\circ$ contains $\bv_\intercal$, and intersects $H_{b(\intercal)}$ at a point.
\end{proposition}

\begin{proof}
By Lemma~\ref{L:new-cone-dual-to-compact}(ii), we have: \begin{equation}\label{EQ:new-cone-dual-to-compact-3}
    \varsigma^\circ \cap \; \bigcup\bigl\{\tau \colon \tau \in \intercal\bigr\} \; = \; \{\bv_\intercal\}.
\end{equation}
To exploit the above equation, we consider the $(\dB \smallsetminus \intercal)$-cut of $\Upgamma(f)$, i.e. \begin{equation}\label{EQ:new-cone-dual-to-compact-4}
    \Upgamma^\ddag \; = \; \Upgamma^\dag \cap \; \bigcap_{\tau \in \intercal}{H_{\bu_\tau,N_\tau}^+} \; \subset \; \Upgamma^\dag
\end{equation}
and let $\Sigma^\ddag$ be its normal fan in $N_\RR$. For $\rho \in \Sigma^\ddag[1] = \Sigma(f)[1] \smallsetminus \ChMingHao{\Sigma(f)[1]|_{\dB \smallsetminus \intercal}}$, we let $\tau_\rho^\ddag$ denote the facet of $\Upgamma^\ddag$ dual to $\rho$. We make some crucial observations: \begin{enumerate}
    \item[(a)] Firstly, by Lemma~\ref{L:facets-dagger}(ii), each $\tau \in \intercal$ is still a facet of $\Upgamma^\ddag$. That is, for $\rho \in \Sigma[\intercal]$, $\tau_\rho^\ddag = \tau_\rho$.
    
    \item[(b)] Secondly, \ChMingHao{recall from \ref{X:convention-4.2-A} that $\dB_{/\sim} = \{\intercal_1,\intercal_2,\dotsc,\intercal_k\}$ with $\intercal = \intercal_\ell$ for some $\ell \in [k]$}. By replacing $\dB$ by $\dB \smallsetminus \bigcup\{\intercal_j \colon j > \ell\}$, we may assume that $\sigma$ cannot be inscribed in any cone in $\Sigma^\ddag$, cf. Proposition~\ref{P:new-cone-detailed}. Then: \begin{align*}
        \qquad \qquad \varnothing \; &= \; \bigcap\bigl\{\tau_\rho^\ddag \colon \rho \in \sigma[1]\bigr\} \quad \;\; \textrm{\small{by Lemma~\ref{L:new-cone}}} \\
        &= \; \bigl(H_{b(\intercal)} \cap \Upgamma^\ddag\bigr) \cap \; \bigcap\bigl\{\tau_\rho^\ddag \colon \rho \in \sigma^\circ[1]\bigr\} \\
        &= \; \bigl(H_{b(\intercal)} \cap \Upgamma^\ddag\bigr) \cap \; \bigcap\bigl\{\tau_\rho^\dag \cap \Upgamma^\ddag \colon \rho \in \sigma^\circ[1]\bigr\} \quad \textrm{\small{by Lemma~\ref{L:facets-dagger}(i)}} \\
        &= \; \Upgamma^\ddag \cap \bigl(H_{b(\intercal)} \cap \Upgamma^\dag\bigr) \cap \; \bigcap\bigl\{\tau_\rho^\dag \colon \rho \in \sigma^\circ[1]\bigr\} \\
        &= \; \Upgamma^\ddag \cap \bigl(H_{b(\intercal)} \cap \Upgamma^\dag\bigr) \cap \varsigma^\circ \; \stackrel{\eqref{EQ:new-cone-dual-to-compact-1}}{\longeq} \; \varsigma \cap \Upgamma^\ddag
    \end{align*}
    i.e. $\varsigma \subset \varsigma^\circ \smallsetminus \Upgamma^\ddag$. In particular, $\varsigma^\circ \smallsetminus \Upgamma^\ddag \neq \varnothing$. We also note that \[
        \qquad \qquad \underline{\varsigma}^\circ \; \stackrel{\eqref{EQ:new-cone-dual-to-compact-2}}{\longeq} \; \varsigma^\circ \cap \Upgamma(f) \; \subset \; \varsigma^\circ \cap \Upgamma^\ddag
    \]
    i.e. in particular, $\varsigma^\circ \cap \Upgamma^\ddag$ is a (non-empty) face of $\Upgamma^\ddag$.
    
    \item[(c)] Thirdly, by \eqref{EQ:new-cone-dual-to-compact-4}, any line segment connecting a point in $\Upgamma^\dag \smallsetminus \Upgamma^\ddag$ to a point in $\Upgamma^\ddag$ must pass through a point in \[
        \qquad \qquad \bigcup\bigl\{\Upgamma^\ddag \cap H_{\bu_\tau,N_\tau} \colon \tau \in \intercal\bigr\} \; \stackrel{\textup{(a)}}{\longeq} \; \bigcup\bigl\{\tau \colon \tau \in \intercal\bigr\}.
    \]
    By \eqref{EQ:new-cone-dual-to-compact-3}, we therefore deduce that any line segment connecting a point in $\varsigma^\circ \smallsetminus \Upgamma^\ddag$ to a point in $\varsigma^\circ \cap \Upgamma^\ddag \prec \Upgamma^\ddag$ must pass through $\bv_\intercal$. 
\end{enumerate}
We can now conclude the proof by considering two cases. \begin{enumerate}
    \item[\underline{Case 1}:] Suppose that $\bv_\intercal$ is always one of the two vertices of \emph{every} line segment connecting a point in $\varsigma^\circ \smallsetminus \Upgamma^\ddag$ to a point in $\varsigma^\circ \cap \Upgamma^\ddag$. Then we claim $\varsigma^\circ \cap \Upgamma^\ddag = \{\bv_\intercal\}$. If not, choose a point $\ba_1 \in (\varsigma^\circ \cap \Upgamma^\ddag) \smallsetminus \{\bv_\intercal\}$. By (b), we may also choose a point $\ba_2$ in $\varsigma^\circ \smallsetminus \Upgamma^\ddag$. By (c), the line segment connecting $\ba_1$ to $\ba_2$ must contain $\bv_\intercal$ in its relative interior, contradicting the hypothesis of this case. From our claim we obtain: \[
        \qquad \qquad \{\bv_\intercal\} \; = \; \varsigma^\circ \cap \Upgamma^\ddag \; \supset \; \varsigma^\circ \cap \Upgamma(f) \; \stackrel{\eqref{EQ:new-cone-dual-to-compact-2}}{\longeq} \; \underline{\varsigma}^\circ \; \supset \; \{\bv_\intercal\}
    \]
    which forces $\underline{\varsigma}^\circ = \{\bv_\intercal\}$.
    
    \item[\underline{Case 2}:] Suppose \emph{there exists} a line segment $\fl$ connecting some $\ba_1 \in \varsigma^\circ \smallsetminus \Upgamma^\ddag$ to some $\ba_2 \in \varsigma^\circ \cap \Upgamma^\ddag$ that contains $\bv_\intercal$ in its relative interior. In particular, note $\ba_1 \neq \bv_\intercal \neq \ba_2$, so that $\dim(\varsigma^\circ) \geq \dim(\varsigma^\circ \cap \Upgamma^\ddag) \geq 1$. We claim that in fact \[
        \qquad \qquad \dim(\varsigma^\circ) \; = \; \dim(\varsigma^\circ \cap \Upgamma^\ddag) \; = \; 1.
    \]
    Indeed, given any $\ba_1' \in \varsigma^\circ \smallsetminus \Upgamma^\ddag$, (c) implies that the line segment connecting $\ba_1'$ to $\ba_2$ must contain $\bv_\intercal$, and thus $\ba_1'$ must lie on the affine span of $\fl$. Likewise, given any $\ba_2' \in \varsigma^\circ \cap \Upgamma^\ddag$, the line segment connecting $\ba_1$ to $\ba_2'$ must contain $\bv_\intercal$, and thus $\ba_2'$ must lie on the affine span of $\fl$. 
    
    Finally, by \eqref{EQ:new-cone-dual-to-compact-2}, $\underline{\varsigma}^\circ = \varsigma^\circ \cap \Upgamma(f) \subset \varsigma^\circ \cap \Upgamma^\ddag$, so $\dim(\underline{\varsigma}^\circ) \leq \dim(\varsigma^\circ \cap \Upgamma^\ddag) = 1$. Since $\underline{\varsigma}^\circ$ always contains $\bv_\intercal$, we conclude that $\underline{\varsigma}^\circ$ is either $\{\bv_\intercal\}$ or $1$-dimensional.
\end{enumerate}
Together these two cases prove the first statement of the proposition. For the second statement, first note that $\dim(\underline{\varsigma}^\circ) = 1$ only occurs in \underline{Case 2}. In that case, we also have $\dim(\varsigma^\circ) = 1$ and $\varsigma^\circ \cap \Upgamma(f) = \underline{\varsigma}^\circ$, so the affine span of $\underline{\varsigma}^\circ$ must be equal to the affine span of $\varsigma^\circ$. By \eqref{EQ:new-cone-dual-to-compact-1}, $\varsigma^\circ$ has non-empty intersection with $H_{b(\intercal)}$ (namely, the face $\varsigma \prec \Upgamma^\dag$). That intersection must be a point since $\bv_\intercal \in \underline{\varsigma}^\circ \subset \varsigma^\circ$ has $b(\intercal)$\textsuperscript{th} coordinate $1$.
\end{proof}

\begin{remark}
From the proof above, one may supplement Proposition~\ref{P:new-cone-dual-to-compact-II} as follows. If $\dim(\underline{\varsigma}^\circ) = 1$, then $\dim(\varsigma^\circ) = 1$ and $\dim(\varsigma) = 0$, i.e. $\sigma \in \Sigma^\dag[\max]$. Note however that if $\underline{\varsigma}^\circ = \{\bv_\intercal\}$, $\dim(\varsigma^\circ)$ and $\dim(\varsigma)$ are arbitrary.
\end{remark}

\begin{proof}[Proof of Proposition~\ref{P:new-cone-dual-to-compact}]
We saw that $\underline{\varsigma}^\circ$ is either $\{\bv_\intercal\}$ or $1$-dimensional. There is nothing to show in the former case. In the latter case, we saw that $\bv_\intercal$ is the \emph{only} point in $\underline{\varsigma}^\circ$ with $b(\intercal)$\textsuperscript{th} coordinate $1$. Combining this with Corollary~\ref{C:new-cone-II} finishes the proof.
\end{proof}

\subsection{A \ChMingHao{refined} desingularization of all non-degenerate polynomials above the origin}\label{4.3}

In this section, let $f \in \kk[x_1,\dotsc,x_n]$ be a non-degenerate polynomial, and we continue adopting the conventions outlined at the start of \S\ref{4.2} and in \ref{X:convention-4.2-A}. We show next that the multi-weighted blow-up of $\AA^n$ along $\Upgamma^\dag$ (cf. Definition~\ref{D:multi-weighted-blow-up}): \[
    \Uppi_{\Sigma^\dag} \; \colon \; \sX_{\Sigma^\dag} \; \xrightarrow{\quad\quad} \; \AA^n
\]
supplies a stack-theoretic embedded desingularization of $V(f) \subset \AA^n$ above the origin $\bzero \in \AA^n$ (Definition~\ref{D:stack-desingularization}). Let us first make this goal concrete.

\begin{xpar}
For the remainder of this section, we write \[
    f \; = \; \sum_{\ba \in \NN^n}{c_\ba \cdot \pmb{x}^\ba} \; \in \; \kk[x_1,\dotsc,x_n] 
\]
where $c_\bzero = f(\bzero) = 0$, and adopt the notations in \ref{X:explicating-multi-weighted-blow-up} (but with $\Sigma$ there replaced by $\Sigma^\dag$ here). By \ref{X:explicating-multi-weighted-blow-up}(i), the total transform of $f$ under $\Uppi_{\Sigma^\dag}$ is: \[
    \Uppi_{\Sigma^\dag}^\#(f) \; = \; \sum_{\ba \in \NN^n}{c_\ba \cdot (\pmb{x}')^\ba \cdot \prod_{\rho \in \Sigma^\dag[\exc]}{(x_\rho')^{\ba \cdot \bu_\rho}}}
\]
where for each $\ba = (a_1,\dotsc,a_n) \in \NN^n$, $(\pmb{x}')^\ba := (x_1')^{a_1} \dotsm (x_n')^{a_n}$. Next, for each $\rho \in \Sigma^\dag[1] = [n] \sqcup \Sigma^\dag[\exc]$ (cf. conventions at the start of \ref{X:explicating-multi-weighted-blow-up}), we set: \begin{equation}\label{EQ:N-rho}
    N_\rho \; := \; N_{\tau_\rho} \; = \; \inf_{\ba \in \Upgamma(f)}{\ba \cdot \bu_\rho} \; = \; \inf_{\ba \in \Upgamma^\dag}{\ba \cdot \bu_\rho}
\end{equation}
cf. \ref{X:alt-characterization}, \ref{X:convention-2.2-B} and \ref{X:piecewise-linear-convex-function-dagger}. \ChMingHao{In the same way as \ref{X:toric-desingularization}}, we define the proper transform of $f$ under $\Uppi_{\Sigma^\dag}$ as: \begin{equation}\label{EQ:proper-transform}
    f' \; := \; \frac{\Uppi^\#_{\Sigma^\dag}(f)}{\prod_{\rho \in \Sigma^\dag[1]}{(x_\rho')^{N_\rho}}} \; = \; \sum_{\ba \in \NN^n}{c_\ba \cdot (\pmb{x}')^{\ba - \bn} \cdot \prod_{\rho \in \Sigma^\dag[\exc]}{(x_\rho')^{\ba \cdot \bu_\rho - N_\rho}}}
\end{equation}
where $\bn := \bigl(N_i \colon i \in [n]\bigr)$. We can now state our goal more precisely in the following theorem:
\end{xpar}

\begin{theorem}\label{T:desingularization-new}
At points in $\Uppi_{\Sigma^\dag}^{-1}(\bzero) \subset \sX_{\Sigma^\dag}$, the divisor \[
    V(f') \; \subset \; \sX_{\Sigma^\dag}
\]
is smooth and intersects the divisors $\bigl\{ V(x_\rho') \subset \sX_{\Sigma^\dag} \colon \rho \in \Sigma^\dag[1],\; N_\rho > 0 \bigr\}$ transversely. In other words, \[
    \Uppi_{\Sigma^\dag}^{-1}\bigl(V(f)\bigr) \; \subset \; \sX_{\Sigma^\dag}
\]
is a simple normal crossings divisor at \ChMingHao{every point} in $\Uppi_\Sigma^{-1}(\bzero) \subset \sX_{\Sigma^\dag}$.
\end{theorem}

\ChMingHao{\begin{remark}
The theorem says $\Uppi_{\Sigma^\dag}$ is a stack-theoretic embedded desingularization of $V(f)$ above $\bzero \in \AA^n$. However, unlike Theorem~\ref{T:toric-desingularization}, $\Uppi_{\Sigma^\dag}$ is not going to be a stack-theoretic embedded desingularization of $V(f) \cup V(x_1x_2\dotsm x_n)$ above $\bzero \in \AA^n$. This will manifest in the proof below.
\end{remark}}

\begin{proof}
We prove this \ChMingHao{theorem} in steps.

\begin{xpar}
Let $\overline{\Sigma}^\dag$ be the augmentation of $\Sigma^\dag$, cf. \ref{X:cox-construction}. For an arbitrary cone $\sigma$ in $\overline{\Sigma}^\dag$, we will need a simplified presentation for the $\sigma$-chart $D_+(\sigma)$ of $\sX_{\Sigma^\dag}$. Let us first recall from \ref{X:explicating-multi-weighted-blow-up}(ii) that \[
    D_+(\sigma) \; = \; \left[\Spec\left(\kk[x_1',\dotsc,x_n']\bigl[x_\rho' \colon \rho \in \Sigma^\dag[\exc]\bigr]\bigl[x_\sigma'^{-1}\bigr]\right) \q \Gm^{\Sigma^\dag[\exc]}\right]
\]
where $x_\sigma' = \prod_{\rho \in \Sigma^\dag[1] \smallsetminus \sigma[1]}{x_\rho'}$. Since $x_\rho'$ is invertible on $D_+(\sigma)$ for $\rho \in \Sigma^\dag[\exc] \smallsetminus \sigma[1]$, and their $\ZZ^{\Sigma^\dag[\exc]}$-weights $\{-\be_\rho \colon \rho \in \Sigma^\dag[\exc] \smallsetminus \sigma[1]\}$ are linearly independent over $\ZZ$ (\ref{X:explicating-multi-weighted-blow-up}(iii)), we observe from \cite[Lemma 1.3.1]{quek-rydh-weighted-blow-up} that by setting \[
    x_\rho' = 1 \qquad \textrm{for every } \rho \in \Sigma^\dag[\exc] \smallsetminus \sigma[1]
\]
we obtain an isomorphism: \begin{equation}\label{EQ:simplified-presentation}
    D_+(\sigma) \; = \; \left[\Spec\bigl(\kk[x_1',\dotsc,x_n']\bigl[x_\rho' \colon \rho \in \sigma[\exc]\bigr]\bigl[x_\sigma'^{-1}\bigr]\bigr) \q \Gm^{\sigma[\exc]}\right]
\end{equation}
where: \begin{enumerate}
    \item $\sigma[\exc] := \Sigma^\dag[\exc] \cap \sigma[1]$.
    \item $x_\sigma'$ becomes $\prod_{i \in [n] \smallsetminus \sigma[1]}{x_i'}$.
    \item The action $\Gm^{\sigma[\exc]} \curvearrowright \Spec(\kk[x_1',\dotsc,x_n'][x_\rho' \colon \rho \in \sigma[\exc][x_\sigma'^{-1}])$ is specified as follows. For each $i \in [n]$, the $\ZZ^{\sigma[\exc]}$-weight of $x_i'$ is $(u_{\rho,i})_{\rho \in \sigma[\exc]}$, and for each $\rho \in \sigma[\exc]$, the $\ZZ^{\sigma[\exc]}$-weight of $x_\rho'$ is $-\be_\rho \in \ZZ^{\sigma[\exc]}$. 
\end{enumerate}
On the right hand side of \eqref{EQ:simplified-presentation}, the expression for the proper transform $f'$ of $f$ under $\Uppi_{\Sigma^\dag}$ becomes: \begin{equation}\label{EQ:proper-transform-II}
    f' \; = \; \sum_{\ba \in \NN^n}{c_\ba \cdot (\pmb{x}')^{\ba - \bn} \cdot \prod_{\rho \in \sigma[\exc]}{(x_\rho')^{\ba \cdot \bu_\rho - N_\rho}}}.
\end{equation}
\end{xpar}

\begin{xpar}
For an arbitrary cone $\sigma$ in $\overline{\Sigma}^\dag$, we deduce from \eqref{EQ:simplified-presentation} an expression for the $(\Gm^{\Sigma^\dag[1]} \q \Gm^{\Sigma^\dag[\exc]})$-orbit $O(\sigma)$ of $\sX_{\Sigma^\dag}$ corresponding to $\sigma$, cf. \ref{X:explicating-multi-weighted-blow-up}(iv): \begin{align}\label{EQ:orbit-sigma}
    \begin{split}
    O(\sigma) \; &= \; \left[\Spec\bigl(\kk\bigl[x_i^\pm \colon i \in [n] \smallsetminus \sigma[1]\bigr]\bigr) \q \Gm^{\sigma[\exc]} \right] \\
    &= \; V\bigl(x_\rho' \colon \rho \in \sigma[1]\bigr) \; \xhookrightarrow{\quad\textrm{closed}\quad} \; D_+(\sigma).
    \end{split}
\end{align}
For $\sigma \in \overline{\Sigma}^\dag$ not contained in any coordinate hyperplane $\{\be_i^\vee = 0\}$ in $N_\RR$, we claim that at \ChMingHao{every point} in $O(\sigma)$, the divisor $V(f') \subset \sX_{\Sigma^\dag}$ is smooth and intersects the divisors in $\{ V(x_\rho') \subset \sX_{\Sigma^\dag} \colon \rho \in \sigma[1],\; N_\rho > 0 \}$ transversely. By Corollary~\ref{C:preimage-of-origin}, this claim proves \ChMingHao{Theorem~\ref{T:desingularization-new}}. We consider two cases. 
\end{xpar}

\begin{xpar}[\underline{\emph{Case A}}]\label{X:case-A}
Assume that $\sigma$ is old. Using the simplified expression for $D_+(\sigma)$ in \eqref{EQ:simplified-presentation} and the corresponding expression for $f'$ in \eqref{EQ:proper-transform-II}, we claim: \begin{align}\label{EQ:proper-transform-III}
    \begin{split}
    f'|_{V(x_\rho' \colon \rho \in \sigma[1])} \; &= \; \sum_{\ba \in \NN^n \cap \; \bigcap\{\tau_\rho \colon \rho \in \sigma[1]\}}{c_\ba \cdot (\pmb{x}')^{\ba-\bn}} \\
    &= \; \sum_{\ba \in \NN^n \cap \; \bigcap\{\tau_\rho \colon \rho \in \sigma[1]\}}{c_\ba \cdot \prod_{i \in [n] \smallsetminus \sigma[1]}{(x_i')^{a_i - N_i}}}.
    \end{split}
\end{align}
Indeed, the only $\ba \in \NN^n$, whose corresponding monomial \[
    (\pmb{x}')^{\ba-\bn} \cdot \prod_{\rho \in \sigma[\exc]}{(x_\rho')^{\ba \cdot \bu_\rho - N_\rho}}
\]
in $f'$ remains non-zero after setting $x_\rho' = 0$ for all $\rho \in \sigma[1]$, must satisfy: \begin{enumerate}
    \item $\ba \cdot \bu_\rho = N_\rho$ for every $\rho \in \sigma[\exc]$, i.e. $\ba \in \tau_\rho$ for every $\rho \in \sigma[\exc]$;
    \item $\ba \cdot \be_i = a_i = N_i$ for every $i \in [n] \cap \sigma[1]$, i.e. $\ba \in \tau_i$ for every $i \in [n] \cap \sigma[1]$.
\end{enumerate}
Next, since $\sigma$ is old, $\bigcap\{\tau_\rho \colon \rho \in \sigma[1]\}$ is a (non-empty) compact face $\underline{\varsigma} \prec \Upgamma$, cf. \ref{X:preimage-of-origin} and Lemma~\ref{L:new-cone}. Then the expression for $f'|_{V(x_\rho' \colon \rho \in \sigma[1])}$ in \eqref{EQ:proper-transform-III} matches the expression for $f_{\underline{\varsigma}} / \pmb{x}^\bn$ \eqref{EQ:f-tau}, after replacing $x_i'$ in the former with $x_i$ for each $i \in [n] \smallsetminus \sigma[1]$. By the non-degeneracy assumption on $f$, $f_{\underline{\varsigma}}/\pmb{x}^\bn$ is smooth on the torus $\Gm^n \subset \AA^n$, so that \[
    V\bigl(f'|_{V(x_\rho' \colon \rho \in \sigma[1])}\bigr) \; \subset \; O(\sigma)
\]
is smooth, i.e. at \ChMingHao{every point} in $O(\sigma) \subset \sX_{\Sigma^\dag}$, $V(f') \subset \sX_{\Sigma^\dag}$ is smooth and intersects the divisors in $\{V(x_\rho') \subset \sX_{\Sigma^\dag} \colon \rho \in \sigma[1]\}$ transversely.
\end{xpar}

\begin{xpar}[\underline{\emph{Case B}}]\label{X:case-B}
Assume that $\sigma$ is new. Let $\sigma'$ be the \emph{smallest} cone in $\Sigma^\dag$ such that $\sigma \sqsubset \sigma'$. With respect to $\sigma'$, we fix, as in \ref{X:convention-4.2-A}, a corresponding $\intercal \in \dB_{/\sim}$ with apex $\bv_\intercal$ and base direction $b(\intercal)$, such that all the hypotheses, observations and results in \S\ref{4.2} hold. In particular, $\RR_{\geq 0}\be_{b(\intercal)}$ must be an extremal ray of $\sigma$, or else $\sigma$ is old by Proposition~\ref{P:new-cone-II}(ii). Letting $\sigma^\circ$ \ChMingHao{be the cone in $N_\RR$ generated by the rays in $\sigma[1] \smallsetminus \{\langle\be_{b(\intercal)}\rangle\}$}, we consider the following factorization of \eqref{EQ:orbit-sigma}: \allowdisplaybreaks\begin{align*}
    O(\sigma) \; & = \; V\bigl(x_\rho' \colon \rho \in \sigma[1]\bigr) \; = \; \left[\Spec\bigl(\kk\bigl[x_i^\pm \colon i \in [n] \smallsetminus \sigma[1]\bigr]\bigr) \q \Gm^{\sigma[\exc]} \right] \\
    &\hookrightarrow \; V\bigl(x_\rho' \colon \rho \in \sigma^\circ[1]\bigr) \; = \; \left[\Spec\bigl(\kk[x_{b(\intercal)}']\bigl[x_i^\pm \colon i \in [n] \smallsetminus \sigma[1]\bigr]\bigr) \q \Gm^{\sigma[\exc]} \right] \\
    &\hookrightarrow \; D_+(\sigma)
\end{align*}
where the expression for $V(x_\rho' \colon \rho \in \sigma^\circ[1])$ is similarly deduced from \eqref{EQ:simplified-presentation}. Next, set $\underline{\varsigma}^\circ := \bigcap\{\tau_\rho \colon \rho \in \sigma^\circ[1]\}$. Similar to \underline{Case A}, we have: \begin{align}\label{EQ:proper-transform-IV}
    \begin{split}
    f'|_{V(x_\rho' \colon \rho \in \sigma^\circ[1])} \; &= \; \sum_{\ba \in \NN^n \cap \; \underline{\varsigma}^\circ}{c_\ba \cdot (\pmb{x}')^{\ba-\bn}} \\
    &= \; \sum_{\ba \in \NN^n \cap \; \underline{\varsigma}^\circ}{c_\ba \cdot (x_{b(\intercal)}')^{a_{b(\intercal)}} \cdot \prod_{i \in [n] \smallsetminus \sigma[1]}{(x_i')^{a_i - N_i}}}.
    \end{split}
\end{align}
(Recall that $N_{b(\intercal)} = 0$, cf. \ref{X:B_1-facets}(iii).) We now \emph{\underline{claim}} that there exists $g \in \kk[x_{b(\intercal)}']\bigl[x_i' \colon i \in [n] \smallsetminus \sigma[1]\bigr]$ such that \begin{equation}\label{EQ:new}
    f'|_{V(x_\rho' \colon \rho \in \sigma^\circ[1])} \; = \; c_{\bv_\intercal} \cdot x_{b(\intercal)}' \cdot \prod_{i \in [n] \smallsetminus \sigma[1]}{(x_i')^{v_i-N_i}} \; + \; (x_{b(\intercal)}')^2 \cdot g
\end{equation}
where each $v_i$ is the $i$\textsuperscript{th} coordinate of $\bv_\intercal$. If $\sigma \in \Sigma^\dag$ (i.e. $\sigma = \sigma'$), then this follows from \eqref{EQ:proper-transform-IV}, Proposition~\ref{P:new-cone-dual-to-compact}, and \ref{X:preimage-of-origin}. \ChMingHao{The general case can be reduced to the aforementioned case where $\sigma \in \Sigma^\dag$. This reduction is standard (similar to \ref{X:preimage-of-origin}), so we have chosen to explicate this separately in Remark~\ref{R:addition} below.} 

Consequently, we deduce from \eqref{EQ:new} that \[
    \frac{\partial f'|_{V(x_\rho' \colon \rho \in \sigma^\circ[1])}}{\partial x_{b(\intercal)}'}\Big|_{V(x_{b(\intercal)}')} \; = \; c_{\bv_\intercal} \cdot \prod_{i \in [n] \smallsetminus \sigma[1]}{(x_i')^{v_i - N_i}}
\]
which is a unit on $O(\sigma) = V(x_\rho' \colon \rho \in \sigma^\circ[1]) \cap V(x_{b(\intercal)}')$, since $c_{\bv_\intercal} \neq 0$ ($\bv_\intercal$ is a vertex of $\Upgamma(f)$) and $x_i'$ is invertible on $O(\sigma)$ for each $i \in [n] \smallsetminus \sigma[1]$ \eqref{EQ:orbit-sigma}. Thus, $V\bigl(f'|_{V(x_\rho' \colon \rho \in \sigma^\circ[1])}\bigr)$ is smooth at \ChMingHao{every point} in $O(\sigma) \subset V\bigl(x_\rho' \colon \rho \in \sigma^\circ[1]\bigr)$, i.e. at \ChMingHao{every point} in $O(\sigma) \subset \sX_{\Sigma^\dag}$, $V(f') \subset \sX_{\Sigma^\dag}$ is smooth and intersects the divisors in \[
    \bigl\{V(x_\rho') \subset \sX_{\Sigma^\dag} \colon \rho \in \sigma[1],\; N_\rho > 0\bigr\} \; \subset \; \bigl\{V(x_\rho') \subset \sX_{\Sigma^\dag} \colon \rho \in \sigma^\circ[1]\bigr\}
\]
transversely. This completes the proof.
\end{xpar}
\end{proof}

\begin{remark}\label{R:addition}
\ChMingHao{In this remark, we prove \eqref{EQ:new} for all cones $\sigma \in \overline{\Sigma}^\dag$. Retain the notation in the above proof.} Letting $(\sigma')^\circ$ \ChMingHao{be the cone in $N_\RR$ generated by the rays in $\sigma'[1] \smallsetminus \{\langle \be_{b(\intercal)}\rangle\}$} (as in Proposition~\ref{P:new-cone-II}(ii)), we have $\sigma^\circ \sqsubset (\sigma')^\circ$. In fact, $(\sigma')^\circ$ is also the \emph{smallest} cone in $\Sigma^\dag$ such that $\sigma^\circ \sqsubset (\sigma')^\circ$. If not, $\sigma^\circ$ lies in a proper face of $(\sigma')^\circ$. Since $(\sigma')^\circ \prec \sigma'$ (Proposition~\ref{P:new-cone-II}(ii)), $\sigma = \sigma^\circ + \ChMingHao{\langle} \be_{b(\intercal)} \ChMingHao{\rangle}$ must also lie in a proper face of $\sigma' = (\sigma')^\circ + \ChMingHao{\langle}\be_{b(\intercal)}\ChMingHao{\rangle}$, contradicting our choice of $\sigma'$. Consequently, \[
    \bigcap\bigl\{\tau_\rho^\dag \colon \rho \in \sigma^\circ[1]\bigr\} \; = \; \bigcap\bigl\{\tau_\rho^\dag \colon \rho \in (\sigma')^\circ[1]\bigr\}
\]
cf. \ref{X:cone-face-correspondence} and Lemma~\ref{C:cone-face-correspondence-A}. Intersecting both sides of the above equality by $\Upgamma(f)$, we obtain $\underline{\varsigma}^\circ = \bigcap\{\tau_\rho \colon \rho \in (\sigma')^\circ[1]\}$. Then \eqref{EQ:new} follows from the preceding sentence together with \eqref{EQ:proper-transform-IV}, Proposition~\ref{P:new-cone-dual-to-compact}, \ChMingHao{and \ref{X:preimage-of-origin}}.
\end{remark}

We conclude this section by proving the main theorems of this paper:

\begin{proof}[Proof of Theorems~\ref{T:fake-poles} and \ref{T:refined-desingularization}]
After replacing $\Sigma(f)$ with $\Sigma^\dag$, the argument in \ref{X:motivic-change-of-variables-II} works verbatim. Fixing a frugal simplicial subdivision $\pmb{\Upsigma}^\dag$ of $\Sigma^\dag$ (\ref{X:simplicial-refinement}), we have: \[
    \begin{tikzcd}
    \Uppi_{\pmb{\Upsigma}^\dag}^{-1}\bigl(V(f)\bigr) \arrow[to=2-1, twoheadrightarrow, swap, "\textrm{coarse space}"] \arrow[to=1-2, "\textrm{closed}", hookrightarrow] & \sX_{\pmb{\Upsigma}^\dag} \arrow[to=2-2, swap, twoheadrightarrow, "\textrm{coarse space}"] \arrow[to=1-3, hookrightarrow, "\textrm{open}"] & \sX_{\Sigma^\dag} \arrow[to=1-4, "\Uppi_{\Sigma^\dag}"] & \AA^n \\
    \uppi_{\pmb{\Upsigma}^\dag}^{-1}\bigl(V(f)\bigr) \arrow[to=2-2, "\textrm{closed}", hookrightarrow] & X_{\pmb{\Upsigma}^\dag} \arrow[to=1-4, "\uppi_{\Sigma^\dag}", swap, bend right=15]
    \end{tikzcd}
\]
where: \begin{enumerate}
    \item $\uppi_{\pmb{\Upsigma}^\dag}$ is proper and birational.
    \item $X_{\pmb{\Upsigma}^\dag}$ has finite quotient singularities (\ref{X:finite-quotient-singularities}).
    \item $\uppi_{\pmb{\Upsigma}^\dag}^{-1}\bigl(V(f)\bigr)$ is a $\QQ$-simple normal crossings divisor \ChMingHao{\cite[Definition 1.6]{veys-motivic-zeta-functions-on-Q-Gorenstein-varieties}} at \ChMingHao{every point} in $\uppi_{\pmb{\Upsigma}^\dag}^{-1}(\bzero) \subset X_{\pmb{\Upsigma}^\dag}$. Indeed, $\Uppi_{\pmb{\Upsigma}^\dag}$ \ChMingHao{factors as $\sX_{\pmb{\Upsigma}^\dag} \xhookrightarrow{\textrm{open}} \sX_{\Sigma^\dag} \xrightarrow{\Uppi_{\Sigma^\dag}} \AA^n$ in the above diagram}. We therefore deduce, from \eqref{EQ:proper-transform}, that: \begin{equation}\label{EQ:proof-thm-A}
        \qquad \qquad \Uppi_{\pmb{\Upsigma}^\dag}^{-1}\bigl(V(f)\bigr) \; = \; V(f') \; + \; \sum_{\rho \in \Sigma^\dag[1]}{N_\rho \cdot V(x_\rho')}
    \end{equation}
    where each $V(x_\rho')$, as well as $V(f')$, is now regarded as a divisor in $\sX_{\pmb{\Upsigma}^\dag} \xhookrightarrow{\textrm{open}} \sX_{\Sigma^\dag}$. By Theorem~\ref{T:desingularization-new}, $\Uppi_{\pmb{\Upsigma}^\dag}^{-1}\bigl(V(f)\bigr)$ is a simple normal crossings divisor at \ChMingHao{every point} in $\Uppi_{\pmb{\Upsigma}^\dag}^{-1}(\bzero) = \Uppi_{\Sigma^\dag}^{-1}(\bzero) \cap \sX_{\pmb{\Upsigma}^\dag}$. It remains to note that $\uppi_{\pmb{\Upsigma}^\dag}^{-1}\bigl(V(f)\bigr)$ is the coarse space of $\Uppi_{\pmb{\Upsigma}^\dag}^{-1}\bigl(V(f)\bigr)$, since the coarse space $\sX_{\pmb{\Upsigma}^\dag} \to X_{\pmb{\Upsigma}^\dag}$ maps the latter onto the former.
\end{enumerate}
That is, $\uppi_{\pmb{\Upsigma}^\dag} \colon X_{\pmb{\Upsigma}^\dag} \to \AA^n$ is an {\sffamily embedded $\QQ$-desingularization of $V(f) \subset \AA^n$ above the origin $\bzero \in \AA^n$}, in the sense that it satisfies (i), (ii) and (iii) above. As noted in \ref{X:motivic-change-of-variables-II}, \cite[Theorem 4]{veys-motivic-zeta-functions-on-Q-Gorenstein-varieties} applies more generally to our case of $\uppi := \uppi_{\pmb{\Upsigma}^\dag}$, $D_1 := V(f)$, $D_2 := 0$, and $W = \{\bzero\}$. Together with \eqref{EQ:relative-canonical-divisor} and \eqref{EQ:proof-thm-A}, we deduce that $Z_{\mot,\bzero}(f;s)$ lies in
\[
    \sM_\kk\bigl[\LL^{-s}\bigr]\left[\frac{1}{1-\LL^{-(s+1)}}\right]\left[\frac{1}{1-\LL^{-(N_\rho s + \abs{\bu_\rho})}} \colon \rho \in \Sigma^\dag[1] = \Sigma(f)[1] \smallsetminus \Sigma(f)[\dB] \right]
\]
i.e. $\Uptheta^{\dag,\dB}(f) = \{-1\}\ \cup \bigl\{-\frac{\abs{\bu_\rho}}{N_\rho} \colon \rho \in \Sigma(f)[1] \smallsetminus \Sigma(f)[\dB]$ with $N_\rho > 0 \bigr\}$ is indeed a set of candidate poles for $Z_{\mot,\bzero}(f;s)$.
\end{proof}

\section{Further remarks and future directions}\label{5}

\subsection{On a potential refinement of \texorpdfstring{Theorem~\ref{T:fake-poles}}{Theorem A} in the case of \texorpdfstring{$B_1$}{B1}-facets}\label{5.1} 
In this section, we revisit Theorem~\ref{T:fake-poles} and explain why the theorem does not seem to give a complete answer even in the case of $B_1$-facets. Recall that $Z_{\tg,\bzero}(f;s)$ denotes the topological zeta function of $f$ at the origin $\bzero \in \AA^n$, cf. \ref{X:difficulty} and Remark~\ref{R:topological-zeta-function}.

\begin{xpar}\label{X:ELT}
Using our conventions, \cite[Proposition 3.8]{esterov-lemahieu-takeuchi-monodromy-conjecture} can be stated as follows. Let $\cS_\circ \subset \Uptheta(f) \smallsetminus \{-1\}$. If $\cF(f;s_\circ)$ is a set of $B_1$-facets with consistent base directions for every $s_\circ \in \cS_\circ$, then every pole of $Z_{\tg,\bzero}(f;s)$ is contained in $\Uptheta(f) \smallsetminus \cS_\circ$. This can be seen as a consequence of our Theorem~\ref{T:fake-poles} as follows. Indeed, we first note an immediate consequence of Theorem~\ref{T:fake-poles}:
\end{xpar}

\begin{corollary}\label{C:fake-poles}
Let $s_\circ \in \Uptheta(f) \smallsetminus \{-1\}$. If $\cF(f;s_\circ)$ is a set of $B_1$-facets with consistent base directions, then $\Uptheta(f) \smallsetminus \{s_\circ\}$ is a set of candidate poles for $Z_{\mot,\bzero}(f;s)$. 
\end{corollary}

\begin{proof}[Proof of statement in \ref{X:ELT}]
In view of Remark~\ref{R:topological-zeta-function}, Corollary~\ref{C:fake-poles} in particular implies that for every $s_\circ \in \cS_\circ$, every pole of $Z_{\tg,\bzero}(f;s)$ is contained in $\Uptheta(f) \smallsetminus \{s_\circ\}$. Thus, every pole of $Z_{\tg,\bzero}(f;s)$ is contained in $\Uptheta(f) \smallsetminus \cS_\circ = \bigcap\bigl\{\Uptheta(f) \smallsetminus \{s_\circ\} \colon s_\circ \in \cS_\circ\bigr\}$.
\end{proof}

\begin{xpar}\label{X:motivic-ELT}
Unfortunately, it is not immediate that the motivic analogue of \ref{X:ELT} is true. Namely, for $\cS_\circ \subset \Uptheta(f) \smallsetminus \{-1\}$, one could pose the following question. If $\cF(f;s_\circ)$ is a set of $B_1$-facets with consistent base directions for every $s_\circ \in \cS_\circ$, then is $\Uptheta(f) \smallsetminus \cS_\circ$ a set of candidate poles for $Z_{\mot,\bzero}(f;s)$? 

One key difficulty behind this question lies in our current lack of understanding of the zero divisors in the localized Grothendieck ring of $\kk$-varieties $\sM_\kk = K_0(\Var_\kk)[\LL^{-1}]$. More precisely, while Corollary~\ref{C:fake-poles} says that $\Uptheta(f) \smallsetminus \{s_\circ\}$ is a set of candidate poles for $Z_{\mot,\bzero}(f;s)$ for each $s_\circ \in \cS_\circ$, it is not clear if that would imply that $\Uptheta(f) \smallsetminus \cS_\circ = \bigcap\bigl\{\Uptheta(f) \smallsetminus \{s_\circ\} \colon s_\circ \in \cS_\circ\bigr\}$ is a set of candidate poles for $Z_{\mot,\bzero}(f;s)$. 
\end{xpar}

\begin{xpar}\label{X:motivic-ELT-2}
Nevertheless, one could try the following different line of attack to the question posed in \ref{X:motivic-ELT}. Namely, for $\cS_\circ \subset \Uptheta(f) \smallsetminus \{-1\}$, the following would be ideal: if $\cF(f;s_\circ)$ is a set of $B_1$-facets of $\Upgamma(f)$ with consistent base directions for each $s_\circ \in \cS_\circ$, then so is $\cF(f;\cS_\circ) := \bigsqcup\{\cF(f;s_\circ) \colon s_\circ \in \cS_\circ\}$. If this is true, Theorem~\ref{T:fake-poles} would give a positive answer to the question in \ref{X:motivic-ELT}. Unfortunately, in general this statement is just not true. For that reason among others, we believe that the notion of ``consistent base directions'' is \emph{still incomplete} for the case of $B_1$-facets. In what follows, we present a broader notion that is motivated by \cite[Conjecture 1.3(i)]{esterov-lemahieu-takeuchi-monodromy-conjecture}, although for the case of $B_1$-facets, ours is slightly broader than theirs.
\end{xpar}

\begin{definition}\label{D:compatible-B1-facets}
A set $\dB$ of $B_1$-facets of $\Upgamma(f)$ has {\sffamily compatible apices} if there exists, for each facet $\tau \in \dB$, a choice of a distinguished apex $\bv_\tau$ with corresponding base direction $b(\tau)$, such that $b(\tau_1) = b(\tau_2)$ for every pair of adjacent facets $\tau_1, \tau_2 \in \dB$ sharing the same distinguished apex $\bv_{\tau_1} = \bv_{\tau_2}$. In this case we call $\{\bv_\tau \colon \tau \in \dB\}$ a {\sffamily set of compatible apices} for $\dB$. 
\end{definition}

\begin{remark}\label{R:consistent-implies-compatible}
If $\dB$ has consistent base directions, then $\dB$ has compatible apices, cf. \ref{X:consistent-B_1-facets}. 
\end{remark}

In view of \ref{X:motivic-ELT-2}, the next lemma supports the narrative that the notion of ``compatible apices'' is \emph{possibly} the correct notion to consider:

\begin{lemma}\label{L:union-compatible}
Let $\cS_\circ \subset \Uptheta(f) \smallsetminus \{-1\}$. If $\cF(f;s_\circ)$ is a set of $B_1$-facets of $\Upgamma(f)$ with compatible apices for each $s_\circ \in \cS_\circ$, then so is $\cF(f;\cS_\circ) := \bigsqcup\{\cF(f;s_\circ) \colon s_\circ \in \cS_\circ\}$.
\end{lemma}

\begin{proof}
For every $s_\circ \in \cS_\circ$, fix a compatible set of apices $\{\bv_\tau \colon \tau \in \cF(f;s_\circ)\}$ for $\cF(f;s_\circ)$. We claim that $\{\bv_\tau \colon \tau \in \cF(f;\cS_\circ)\}$ is a compatible set of apices for $\cF(f;\cS_\circ)$. Suppose not. Then there exists adjacent facets $\tau_1, \tau_2 \in \cF(f;\cS_\circ)$ such that $\bv_{\tau_1} = \bv_{\tau_2} =: \bv$ but $b(\tau_1) \neq b(\tau_2)$. Letting $\varsigma := \tau_1 \cap \tau_2$, observe that: \begin{enumerate}
    \item $\bv \in \vertex(\varsigma)$, and the $b(\tau_1)$\textsuperscript{th} and $b(\tau_2)$\textsuperscript{th} coordinates of $\bv$ are both $1$.
    \item Any $\bw \in \vertex(\varsigma) \smallsetminus \{\bv\}$ lies in $H_{b(\tau_1)} \cap H_{b(\tau_2)}$.
    \item $\varsigma$ is compact in the $b(\tau_1)$\textsuperscript{st} and $b(\tau_2)$\textsuperscript{th} coordinates.
\end{enumerate} 
Together, these imply that $\varsigma$ is contained in the hyperplane $H$ in $M_\RR$ defined by $\be_{b(\tau_1)} - \be_{b(\tau_2)} = 0$. \ChMingHao{In fact, since $\varsigma \prec^1 \tau_1, \tau_2$, we have $\varsigma = \tau_1 \cap H = \tau_2 \cap H$. For $i=1,2$, $s_{\tau_i}$ is the unique positive rational number for which $s_{\tau_i}^{-1} \cdot (1,1,\dotsc,1)$ lies on the affine span of $\tau_i$, or equivalently, $s_{\tau_i}^{-1} \cdot (1,1,\dotsc,1)$ lies on the affine span of $\varsigma = \tau_i \cap H$. Since that last condition is independent of $i$, we deduce $s_{\tau_1} = s_{\tau_2}$, a contradiction to the first sentence of this proof.} 
\end{proof}

Motivated by \cite[Conjecture 1.3(i)]{esterov-lemahieu-takeuchi-monodromy-conjecture}, one could ask the following:

\begin{question}\label{Q:compatible}
Are the following statements true?
\begin{enumerate}
    \item Let $\dB$ be a set of $B_1$-facets of $\Upgamma(f)$ with compatible apices. Then \[
        \qquad \qquad \Uptheta^{\dag,\dB}(f) \; := \; \{-1\} \; \cup \; \bigl\{s_\tau \colon \tau \prec^1 \Upgamma(f) \textrm{ with } N_\tau > 0 \textrm{ and } \tau \not\in \dB\bigr\}
    \]
    is a set of candidate poles for $Z_{\mot,\bzero}(f;s)$.
    \item Let $\cS_\circ \subset \Uptheta(f) \smallsetminus \{-1\}$. If $\cF(f;s_\circ)$ is a set of $B_1$-facets of $\Upgamma(f)$ with compatible apices for each $s_\circ \in \cS_\circ$, then $\Uptheta(f) \smallsetminus \cS_\circ$ is a set of candidate poles for $Z_{\mot,\bzero}(f;s)$.
\end{enumerate}
Note \emph{(i)} is a generalization of Theorem~\ref{T:fake-poles}, \emph{(i)} implies \emph{(ii)} by Lemma~\ref{L:union-compatible}, and \emph{(ii)} in particular gives a positive answer to the question posed in \ref{X:motivic-ELT}.
\end{question}

Unfortunately, these are false, as indicated by a counterexample \cite[Example 2.2.2]{larson-payne-stapledon-simplicial-nondegenerate-monodromy-conjecture}. Nevertheless, some refinement should be true, and this is \ChMingHao{currently being pursued} in a separate sequel. For now, we have:

\begin{theorem}[$=$ Theorem~\ref{T:fake-poles-n=3}]\label{T:fake-poles-n=3:compatible}
If $n=3$, Question~\ref{Q:compatible}(i) is positive.
\end{theorem}

Indeed, this follows from Theorem~\ref{T:fake-poles} and the following lemma:

\begin{lemma}\label{L:compatible-vs-consistent-n=3}
Let $n=3$, and let $\dB$ be a set of $B_1$-facets of $\Upgamma(f)$. Then $\dB$ has consistent base directions if and only if $\dB$ has compatible apices.
\end{lemma}

\begin{proof}
Suppose there exists a compatible set of apices $\{\bv_\tau \colon \tau \in \dB\}$ for $\dB$. We then claim that whenever two facets $\tau_1, \tau_2 \in \dB$ are adjacent and $b(\tau_1) \neq b(\tau_2)$, then one of $\tau_1$ or $\tau_2$, \emph{say $\tau_2$}, satisfies the following: \begin{enumerate}
    \item[(a)] $\tau_1$ is the \emph{only} facet in $\dB$ adjacent to $\tau_2$.
    \item[(b)] $\bv_{\tau_1}$ is also an apex for $\tau_2$, with corresponding base direction $b(\tau_1)$.
\end{enumerate}
Admitting this claim, we re-assign $\tau_2$ with the base direction $b(\tau_1)$. Repeating this re-assignment of base direction for all such pairs $(\tau_1,\tau_2)$ in $\dB$ would then culminate in a set of consistent base directions for $\dB$. To prove the claim, we make three successive observations: \begin{enumerate}
    \item Firstly, every facet of $\tau_1$, with the exception of $H_{b(\tau_1)} \cap \tau_1 \prec^1 \tau_1$, contains $\bv_{\tau_1}$ (cf. \ref{X:B_1-facets}(i)). Thus, $\bv_{\tau_1}$ is a vertex of $\tau_1 \cap \tau_2 \prec^1 \tau_1$. Likewise, $\bv_{\tau_2}$ is a vertex of $\tau_1 \cap \tau_2 \prec^1 \tau_2$. We conclude $\tau_1 \cap \tau_2$ is the line segment in $M_\RR^+$ connecting the vertex $\bv_{\tau_1}$ to the vertex $\bv_{\tau_2}$.
    
    \item Secondly, by re-ordering coordinates if necessary, we may assume $b(\tau_1) = 1$ and $b(\tau_2) = 2$. Since $\bv_{\tau_1} \in \vertex(\tau_1 \cap \tau_2) \smallsetminus \{\bv_{\tau_2}\} \subset \vertex(\tau_2) \smallsetminus \{\bv_{\tau_2}\}$, the $2$\textsuperscript{nd} coordinate of $\bv_{\tau_1}$ is $0$. Likewise, the $1$\textsuperscript{st} coordinate of $\bv_{\tau_2}$ is $0$. Summing up, we have $\bv_{\tau_1} = (1,0,a)$ and $\bv_{\tau_2} = (0,1,b)$ for some $a,b \in \NN$.

    \item Thirdly, we claim that besides $\bv_{\tau_1}$ and $\bv_{\tau_2}$, there can only be at most one other $\bv \in \vertex(\Upgamma(f))$ satisfying $\bv \cdot (\be_1+\be_2) \leq 1$, and moreover such a $\bv$ must equal $(0,0,c)$ for some $c \in \NN$. Indeed, if $\bv = (v_1,v_2,v_3) \in \vertex(\Upgamma(f))$ satisfies $\bv \cdot (\be_1+\be_2) = v_1+v_2 \leq 1$, then $(v_1,v_2) = (0,0)$, $(1,0)$ or $(0,1)$. The case $(v_1,v_2) = (1,0)$ cannot happen since otherwise $\bv - \bv_{\tau_1} \in \RR\be_3^\vee$, but no two distinct vertices of $\Upgamma(f)$ can differ by a vector in $\sum_{i=1}^n{\RR_{\geq 0}\be_i^\vee}$ or in $\sum_{i=1}^n{\RR_{\leq 0}\be_i^\vee}$. Likewise, $(v_1,v_2) \neq (0,1)$, or else $\bv - \bv_{\tau_2} \in \RR\be_3^\vee$. Thus, $(v_1,v_2) = (0,0)$. Note too that there cannot be two distinct $\bv, \bv' \in \vertex(\Upgamma(f))$ of the form $(0,0,c)$ for $c \in \NN$, or else $\bv - \bv' \in \RR\be_3^\vee$. 
\end{enumerate}
Returning back to the claim, we deduce from (iii) that the hyperplane $H_{\be_1+\be_2,1} = \{\ba \in M_\RR^+ \colon \ba \cdot (\be_1+\be_2) = 1\}$ intersects $\Upgamma(f)$ in $(\tau_1 \cap \tau_2) + \RR_{\geq 0}\be_3^\vee$. Thus, if $H_{\be_1+\be_2,1}$ is a supporting hyperplane for $\Upgamma(f)$, either $\tau_1$ or $\tau_2$ is $(\tau_1 \cap \tau_2) + \RR_{\geq 0}\be_3^\vee$. Otherwise, by (iii) there must exist a unique $\bv \in \vertex(\Upgamma(f))$ such that $\bv \cdot (\be_1+\be_2) < 1$, and $\bv = (0,0,c)$ for some $c \in \NN$. Then the convex hull of $(\tau_1 \cap \tau_2) \cup \{\bv\}$ in $M_\RR^+$ is a $2$-dimensional face of $\Upgamma(f)$ that contains $\tau_1 \cap \tau_2$ as a face, and hence, must be either $\tau_1$ or $\tau_2$. In either case one verifies from its respective conclusion that our claim holds.
\end{proof}

\subsection{Other remarks and directions}\label{5.2}

\begin{xpar}[Looking beyond $B_1$-facets]
It is natural to ask if the consideration of $B_1$-facets is sufficient for the monodromy conjecture for non-degenerate polynomials in $n \geq 4$ variables. The answer is \emph{no}: in \cite{esterov-lemahieu-takeuchi-monodromy-conjecture}, the authors described what they call a \emph{$B_2$-facet}, and showed that for the case $n=4$, certain configurations of $B_1$ and $B_2$-facets of $\Upgamma(f)$ contribute to fake poles of $Z_{\tg,\bzero}(f;s)$. For general $n$, the authors also gave, in \cite[Conjecture 1.3(i)]{esterov-lemahieu-takeuchi-monodromy-conjecture}, a conjectural description of when a configuration of facets of $\Upgamma(f)$ could culminate in fake poles of $Z_{\tg,\bzero}(f;s)$. There does not seem to be a clear connection between their conjectural description and our methods. In fact, Larson--Payne--Stapledon recently supplied a counterexample to that conjecture \cite[Example 2.2.1]{larson-payne-stapledon-simplicial-nondegenerate-monodromy-conjecture}. Nevertheless we anticipate the case of $B_2$-facets, which we are pursuing in a sequel, would demystify matters.
\end{xpar}

\begin{xpar}[On Corollary~\ref{C:motivic-nondegenerate-monodromy-conjecture-n=3}]
While half of the proof of Corollary~\ref{C:motivic-nondegenerate-monodromy-conjecture-n=3} was input from this paper, the other half uses observations that are proven separately in \cite{lemahieu-van-proeyen-nondegenerate-surface-singularities}. Nevertheless, we expect that one can use the stack-theoretic embedded desingularization $\Uppi_{\Sigma^\dag} \colon \sX_{\Sigma^\dag} \to \AA^n$ of $V(f) \subset \AA^n$ above $\bzero \in \AA^n$ in \S\ref{4.3} to re-prove the other half of Corollary~\ref{C:motivic-nondegenerate-monodromy-conjecture-n=3}, via a ``stack-theoretic analogue'' of A'Campo's formula \cite{a'campo-monodromy-zeta-function} for the monodromy zeta function, e.g. \cite[Theorem 2.8]{martin-morales-monodromy-zeta-formula-Q-resolns}. For brevity, we omit pursuing this here.
\end{xpar}


\bibliography{monodromy}
\bibliographystyle{dary}

\end{document}